\documentclass[10pt, reqno]{amsart}
\usepackage{amsthm, amsmath, amsfonts, amssymb, enumerate, textcmds, tensor, enumitem, mathdots}
\usepackage{geometry}
\usepackage{pst-node}
\usepackage{tikz-cd} 
\usepackage{appendix}
\usepackage{graphicx,tikz}
\usepackage{nicefrac}
\usepackage{faktor}
\usepackage[colorlinks=true]{hyperref}

\usepackage[style=numeric]{biblatex}
\hypersetup{
	colorlinks,
	citecolor=black,
	filecolor=black,
	linkcolor=black,
	urlcolor=black
}
\usepackage{cleveref}

\usepackage{physics}
\usepackage{enumerate}
\usepackage{pinlabel}
\usepackage[colorinlistoftodos]{todonotes}
\newcommand{\R}{\mathbb R}

\newtheorem{thm}{Theorem}[section]
\newtheorem{lem}[thm]{Lemma}

\newtheorem{prop}[thm]{Proposition}
\newtheorem{cor}[thm]{Corollary}

\newtheorem*{thma}{Theorem A}
\newtheorem*{thmb}{Theorem B}
\newtheorem*{thmc}{Theorem C}
\newtheorem*{thmd}{Theorem D}

\DeclareMathOperator{\C}{\mathbb{C}}

\DeclareMathOperator{\Diff}{\mathrm{Diff}}
\DeclareMathOperator{\CP}{\mathbb{C}\mathbb{P}}

\DeclareMathOperator{\SL}{\mathrm{SL}}
\DeclareMathOperator{\PSL}{\mathrm{PSL}}

\DeclareMathOperator{\inners}{\langle \cdot, \cdot \rangle}

\newcommand{\ccpair}{(c_1, \overline{c_2})}
\DeclareMathOperator{\fm}{\vb*{\overline {f_{--}}}}
\DeclareMathOperator{\f+}{\vb*{f_+}}
\newcommand{\hpair}{\vb*h_{(c_1, \overline{c_2})}}
\DeclareMathOperator{\phee}{\varphi}
\DeclareMathOperator{\eps}{\varepsilon}

\DeclareMathOperator{\Q}{\mathrm Q}
\DeclareMathOperator{\CAS}{\mathrm {CAS}}
\DeclareMathOperator{\CD}{\mathcal{CD}}

\usepackage{physics}
\usepackage{soul}

\usepackage{scalerel}

\usepackage{stmaryrd}

\newcommand{\veee}{\mathscr T}
\newcommand{\veeet}{\mathscr T_t}

\newcommand{\veet}{{\textstyle \mathstrut}\mathscr T_{Z_t}}

\newcommand{\CSCS}{\mathcal C(S)\times \mathcal C(\overline S)}
\newcommand{\TSTS}{\mathcal T(S)\times \mathcal T(\overline S)}
\newcommand{\Sph}{\mathbb{CAS}}

\usepackage{mathrsfs}
\usepackage{soul}

\theoremstyle{definition}
\newtheorem{defn}[thm]{Definition}
\newtheorem{remark}[thm]{Remark}
\newtheorem{example}[thm]{Example}
\newtheorem*{convention}{Convention}

\begin{document}
\title[Complex affine spheres and a Bers theorem for $\mathrm{SL}(3,\C)$]{Complex affine spheres and a Bers theorem for $\mathrm{SL}(3,\C)$}

\author[Christian El Emam]{Christian El Emam}
\address{Christian El Emam: University of Luxembourg, 
Maison du Nombre,
6 Avenue de la Fonte,
L-4364 Esch-sur-Alzette, Luxembourg.} \email{christian.elemam@uni.lu}

\author[Nathaniel Sagman]{Nathaniel Sagman}
\address{Nathaniel Sagman: University of Luxembourg, 
Maison du Nombre,
6 Avenue de la Fonte,
L-4364 Esch-sur-Alzette, Luxembourg.} \email{nathaniel.sagman@uni.lu}

\begin{abstract}
For $S$ a closed surface of genus at least $2$, let $\mathrm{Hit}_3(S)$ be the Hitchin component of representations to $\mathrm{SL}(3,\R),$ equipped with the Labourie-Loftin complex structure. We construct a mapping class group equivariant holomorphic map from a large open subset of $\mathrm{Hit}_3(S)\times \overline{\mathrm{Hit}_3(S)}$ to the $\mathrm{SL}(3,\C)$-character variety that restricts to the identity on the diagonal. The open subset contains $\mathrm{Hit}_3(S)\times \overline{\mathcal{T}(S)}$ and $\mathcal{T}(S)\times \overline{\mathrm{Hit}_3(S)}$ and the holomorphic map restricts to Bers' simultaneous uniformization on $\mathcal{T}(S)\times \overline{\mathcal{T}(S)}$.

The map is realized by associating pairs of Hitchin representations to interesting new immersions into $\C^3$ that we call complex affine spheres, which are obtained by solving a second-order complex elliptic PDE that resembles both the Beltrami and Tzitz{\'e}ica equations. To study this equation we establish analytic results that should be of independent interest.

\end{abstract}
\maketitle
\tableofcontents

\section{Introduction}
Given a closed oriented surface $S$ of genus $\mathrm g\geq 2$, let $\mathcal{T}(S)$ be the Teichm{\"u}ller space of marked complex structures on $S$. By the uniformization theorem, every point in $\mathcal{T}(S)$ determines a unique conjugacy class of Fuchsian representations of $\pi_1(S)$ into $\mathrm{PSL}(2,\R)$. 

Embedding $\mathrm{PSL}(2,\R)$ inside its complexification $\mathrm{PSL}(2,\C),$ any Fuchsian representation produces a representation to $\mathrm{PSL}(2,\C)$ that preserves a round circle in $\mathbb{CP}^1.$ If we deform such a representation inside the character variety $\chi(\pi_1(S),\mathrm{PSL}(2,\C))$ in a controlled manner, the resulting representation is quasi-Fuchsian, which means it preserves a Jordan curve in $\mathbb{CP}^1$. Quasi-Fuchsian representations have a number of nice properties; for example, quasi-Fuchsian representations are discrete and faithful, and hence the quotient of the action of a quasi-Fuchsian representation on $\mathbb{H}^3$ is a quasi-Fuchsian $3$-manifold. Given a quasi-Fuchsian representation $\rho$ preserving a Jordan curve $\gamma$, $\rho$ also preserves and acts properly discontinuously on the two connected domains $\Omega_\rho^+,\Omega_\rho^-\subset \mathbb{CP}^1$ that are complementary to $\gamma$. The resulting quotients $\faktor{\Omega_\rho^+}{\rho}$ and $\faktor{\Omega_\rho^-}{\rho}$ represent points in $\mathcal{T}(S)$ and $\mathcal{T}(\overline{S})$ respectively, where $\mathcal{T}(\overline{S})$ is the Teichm{\"u}ller space of the oppositely oriented surface $\overline{S}$. On the other hand, Bers proved that one can associate a quasi-Fuchsian representation to any pair of oppositely oriented complex structures, in a way that reverses the process above. Moreover, if we denote by $\mathcal{QF}(S)\subset {\chi}(\pi_1(S),\mathrm{PSL}(2,\C))$ the open subset of quasi-Fuchsian representations, and give $\mathcal{T}(S)$ its natural complex structure, then Bers' simultaneous uniformization theorem (Theorem \ref{thm: Bers thm} below) says that this construction descends to a biholomorphism $$\mathcal{T}(S)\times\mathcal{T}(\overline{S})\to \mathcal{QF}(S).$$  

In another direction, for every $n\geq 2$, given a Fuchsian representation into $\mathrm{PSL}(2,\R)$, we can turn it into a representation to $\mathrm{PSL}(n,\R)$ by post-composing with an embedding to $\mathrm{PSL}(n,\R)$ induced from an irreducible representation $\mathrm{PSL}(2,\R)\to \mathrm{PSL}(n,\R)$. Any deformation of such a representation inside $\chi(\pi_1(S),\mathrm{PSL}(n,\R))$ is called a Hitchin representation, and for $n$ odd, such representations fill up a connected component $\textrm{Hit}_n(S)$ of $\chi(\pi_1(S),\mathrm{PSL}(n,\R))$ called the Hitchin component (and when $n$ is even, there are two isomorphic Hitchin components). Introduced in Hitchin's seminal paper \cite{Hi}, Hitchin representations are Anosov (hence discrete and faithful) \cite{Lab1}, and are central objects of study in higher Teichm{\"u}ller theory.

For $n$ odd, $\mathrm{SL}(n,\R) = \mathrm{PSL}(n,\R)$, so for the rest of the paper we use $\mathrm{SL}(3,\cdot)$ over $\mathrm{PSL}(3,\cdot)$. Let $\mathcal{M}_3(S)$ be the holomorphic vector bundle over $\mathcal{T}(S)$ whose fiber over a class of a Riemann surface is identified with its space of holomorphic cubic differentials, i.e., it consists of pairs $[c,\Q]$, where $c$ is a complex structure on $S$ and $\Q$ is a holomorphic cubic differential on $(S,c)$. Labourie in \cite{Lab2} and Loftin in \cite{Lof} independently constructed a mapping class group equivariant and real analytic diffeomorphism $$\mathcal{L}: \mathcal{M}_3(S) \xrightarrow{\sim}\mathrm{Hit}_3(S)\ .$$ 
Both authors build this map using hyperbolic affine spheres. Labourie also gave a way to see $\mathcal{L}$ using minimal surfaces in symmetric spaces \cite{Lab3}, which generalizes to other higher Teichm{\"u}ller spaces for Lie groups of rank $2$ (see \cite{CTT}).

Similar to the case $n=2$, we can deform the Hitchin locus inside ${\chi}(\pi_1(S),\mathrm{SL}(3,\C)),$ and moreover we can inquire to what extent a Bers' phenomenon might hold for $\mathrm{SL}(3,\C),$ when we give $\mathrm{Hit}_3(S)$ the complex structure induced by $\mathcal{L}$. That is, one could--perhaps optimistically--try to holomorphically identify points in $\mathcal{M}_3(S)\times \mathcal{M}_3(\overline{S})$ with representations to $\mathrm{SL}(3,\C)$. In fact, embedding $\mathcal{M}_3(S)$ as the ``diagonal" inside $\mathcal{M}_3(S)\times \mathcal{M}_3(\overline{S})$, by analytic continuing power series, one can find neighbourhoods $U$ of this diagonal and $V$ of the Hitchin representations inside ${\chi}(\pi_1(S),\mathrm{SL}(3,\C))$ on which $\mathcal{L}$ extends uniquely to a biholomorphism $${\mathcal{L}}_{\C}: U\to V.$$ 

In this paper, we introduce maps from surfaces to $\C^3$, called complex affine spheres, that provide a concrete and geometrically meaningful realization of ${\mathcal{L}}_{\C}$. Since hyperbolic affine spheres have been used to study Hitchin representations to $\textrm{SL}(3,\R)$, we expect that complex affine spheres will become a useful piece of the theory of deformations of Hitchin representations inside $\textrm{SL}(3,\C)$. For our problem of interest, namely generalizing Bers' theorem, we use complex affine spheres to show that ${\mathcal{L}}_{\C}$ extends surprisingly far into $\mathcal{M}_3(S)\times \mathcal{M}_3(\overline{S})$, well past the diagonal as shown in Theorem A below, providing a more substantial and interesting Bers' phenomenon for $\mathrm{SL}(3,\C)$.

\subsection{Bers theorem for $\textrm{SL}(3,\C)$}
The Labourie-Loftin parametrization via hyperbolic affine spheres comes from the classical theory of affine differential geometry. See the book \cite{NS} for references. It follows, from work of Cheng-Yau in \cite{CY1} and \cite{CY2} on affine differential geometry,  and Goldman in \cite{Gol} and Choi-Goldman in \cite{ChG} on geometric structures, that every Hitchin representation to $\mathrm{SL}(3,\R)$ is uniquely associated with an equivariant hyperbolic affine sphere. Labourie and Loftin proved that every equivariant hyperbolic affine sphere is determined by the data of a complex structure on $S$ and an accompanying holomorphic cubic differential (see also Wang \cite{Wa}), determining the map $\mathcal L$ above.
 
We define (positive hyperbolic) complex affine spheres to be immersions of surfaces inside $\C^3$ satisfying a geometric condition and structural equation analogous to that of the real affine spheres in $\R^3$ (see Definition \ref{def: CASdef} and the subsection below). We show in Corollary \ref{cor: PDE of Gauss eq} that a complex affine sphere is equivalent to a pair of complex structures on $S$ with opposite orientations (not necessarily conjugate), holomorphic cubic differentials for the two complex structures, and a solution to an equation that depends on the complex structures and differentials. 
Through this formalism, we explicitly construct and study ${\mathcal{L}}_{\C}$. Denote the mapping class group by $\mathrm{MCG}(S).$ The following is proved in Section 6.

\begin{thma}\label{thm A}
The Labourie-Loftin parametrization extends uniquely to a $\mathrm{MCG}(S)$-equivariant holomorphic map
\[
\mathcal M_3(S)\times \mathcal M_3(\overline S)\supset \Omega \xrightarrow{\ {\mathcal L}_{\C}\  } \chi(\pi_1(S), \mathrm{SL}(3,\C) )
\]
where $\Omega$ is a $\mathrm{MCG}(S)$-invariant open subset 
containing the diagonal, $\mathcal{M}_3(S)\times \mathcal{T}(\overline{S})$, and ${\mathcal T}(S)\times \mathcal{M}_3(\overline{S})$. Moreover, ${\mathcal{L}}_{\C}$ agrees with Bers' simultaneous uniformization map on ${\mathcal T}(S)\times \mathcal{T}(\overline{S})$. Image points are holonomies of complex affine spheres.
\end{thma}
  Above, we are using $\mathcal{T}(S)$ to denote the embedding of $\mathcal{T}(S)$ into $\mathcal{M}_3(S)$ as the zero section. The fact that $\Omega$ contains $\mathcal{M}_3(S)\times \mathcal{T}(\overline{S})$ and $\mathcal{T}(\overline{S})\times \mathcal{M}_3(\overline{S})$ is very interesting to us because, a priori, ${\mathcal{L}}_{\C}$ is not expected to extend this far from the diagonal. In fact, $\Omega$ can be constructed to be invariant under what we call the $\C^*$ twistor action, defined by $\zeta\cdot ([c_1,\Q_1],[\overline{c_2},\overline{\Q_2}])= ([c_1,\zeta \Q_1],[\overline{c_2},\zeta^{-1} \overline{\Q_2}])$, which also shows that $\Omega$ is large. Note we need to be precise when talking about holomorphicity around the character variety (see Remark \ref{rem: character variety}).

We do not know whether the representations in the image of ${\mathcal{L}}_{\C}$ are Anosov. Note that real affine spheres descend to maps from $\widetilde{S}\to \mathbb{RP}^2$, which then pull back convex projective structures to $S$. In light of this point, it would be interesting to understand if one can use complex affine spheres to see geometric structures.

\begin{remark}\label{rem: previous version}
Labourie and Baraglia have shown that real hyperbolic affine spheres are equivalent to certain minimal surfaces in the
Riemannian symmetric space of $\textrm{SL}(3,\R)$ (see \cite[Section 9]{Lab2} and \cite[Section 3.4]{B}). 

In the original version of this paper posted on the arXiv, we introduced harmonic maps to the holomorphic Riemannian manifold $\textrm{SL}(n,\C)/\textrm{SO}(n,\C)$, with the holomorphic Riemannian metric induced by the Killing form, which give rise to gauge-theoretic objects that generalize Higgs bundles. We showed that the complex affine spheres of Theorem A are in correspondence with certain conformal harmonic maps to $\textrm{SL}(3,\C)/\textrm{SO}(3,\C)$, providing an alternative route to Theorem A and expanding the results of Labourie and Baraglia.

We have since removed that part from the paper, since it is not necessary for Theorem A, and because we treat this theory in more depth in an upcoming work. In that upcoming work, we generalize Theorem A to rank $2$ Hitchin components, use our tools to study Goldman's symplectic form on rank $2$ Hitchin components, and also give an interpretation for the points in $\mathcal{M}_3(S)\times \mathcal{T}(\overline{S})$ and its partner related to the theory of opers.
\end{remark}

\subsection{The space of complex affine spheres}\label{subsec: CAS space}
Before turning to Theorem A, we introduce complex affine spheres (Section 3), a new tool for studying representations to $\textrm{SL}(3,\C)$. We prove Theorem A by constructing a finite dimensional moduli space of complex affine spheres through which ${\mathcal{L}}_{\C}$ factors. 

The first ingredient in our theory is the notion of complex metrics on $S$ (discussed in detail in Section 2). A complex metric on $S$ is a non-degenerate bilinear symmetric form on the complexified tangent bundle $\C TS$, and it is called positive if it satisfies a further non-degeneracy condition (see Definition \ref{def: positive}). Given a positive complex metric $g$, one can associate two complex structures $c_1$ and $\overline{c_2}$ to $S$ of opposite orientations, and one can always locally find coordinates $z$ for $c_1$ and $\overline{w}$ for $\overline{c_2}$, and a function $\lambda$ to $\C^*$ such that $$g = \lambda dz d\overline{w}.$$

In \cite{BEE}, the authors formulate a covariant differential calculus for positive complex metrics, analogous to that of Riemannian metrics, in which we have Levi-Civita connections, curvature, Laplacians, and other usual objects. Finally, for every pair of oppositely oriented complex structures $(c_1, \overline{c_2})$, one can find a special positive complex metric $\hpair$ of constant curvature $-1$ in that conformal class, which is called the Bers metric, and every positive complex metric is conformally equivalent to a Bers metric.

Toward building ${\mathcal{L}}_{\C}$, we study (positive hyperbolic) complex affine spheres that can be deformed to real affine spheres. Complex affine spheres carry a number of tensors that generalize those from the real theory, such as (complex) Blaschke metrics and Pick tensors, satisfying Gauss and Codazzi-type equations. If the Blaschke metric has conformal class $\ccpair$, then the Codazzi equation is equivalent to the fact that the Pick tensor is of the form $C=\Q_1+\overline{\Q_2}$, with $\Q_1$ and $\overline{\Q_2}$ holomorphic cubic differentials for $c_1$ and $\overline{c_2}$ respectively (see Theorem \ref{thm: integrating GE for CAS}). For a distinctive subset of complex affine spheres, including those deformable to the real locus, solving the Gauss equation corresponds to finding complex solutions $u$ to the equation
\begin{equation}\label{eqn: firstgauss}
G(c_1, \overline{c_2}, \Q_1, \overline{\Q_2},u):=\Delta_{h} u-e^{2u}+\frac 1 4 h(\Q_1,\overline{\Q_2})\cdot   e^{-4u}+1\ .
\end{equation}
where $h=\hpair$.
Note that for a ``diagonal" point $(c,\overline{c},\Q,\overline{\Q})$, Equation \eqref{eqn: firstgauss} reduces to the ordinary Tzitz{\'e}ica equation for real hyperbolic affine spheres, as studied by Labourie and Loftin.

 Let $\mathcal{C}(S)$ be the space of complex structures on $S$, let $\mathcal{CD}(S)$ be the bundle of cubic differentials over $\mathcal{C}(S)$, and let $\mathcal{C}(\overline S)$ and $\mathcal{CD}(\overline{S})$ be the analogous spaces for $\overline{S}$. The product $\mathcal{CD}(S)\times \mathcal{CD}(\overline{S})$ is a complex Banach manifold, and when we take the quotient by $\mathrm{Diff}_0(S)\times \mathrm{Diff}_0(S),$ where $\mathrm{Diff}_0(S)$ is the space of diffeomorphisms of $S$ isotopic to the identity, we get $\mathcal{M}_3(S)\times \mathcal{M}_3(\overline{S})$. Our distinctive subset of affine spheres is parametrized by
 $$\mathrm{CAS}(S)=\{(c_1,\overline{c_2},\Q_1,\overline{\Q_2},u)\in \mathcal{CD}(S)\times \mathcal{CD}(\overline{S})\times C^\infty(S,\C): G(c_1,\overline{c_2},\Q_1,\overline{\Q_2},u)=0\}.$$
We say that a complex affine sphere $(c_1,\overline{c_2},\Q_1,\overline{\Q_2},u)$ is infinitesimally rigid if linearization of $G$ in the direction of $u$ is an isomorphism. We denote the space of infinitesimally rigid complex affine spheres by $\mathrm{CAS}^*(S)$. The space $\CAS^*(S)$ is a complex Fréchet manifold for which the natural projection $\hat\pi \colon \CAS^*(S)\to \CD(S)\times \CD(S)$ is a local biholomorphism and the holonomy map $hol\colon \CAS^*(S)\to \chi(\pi_1(S), \SL(3,\C))$ is holomorphic (see Section \ref{sec: BM and Gauss}).

An important issue concerning the spaces $\CAS(S)$ and $\CAS^*(S)$, which stems from a general issue around complex-metric related objects, is that their quotients under the diagonal action of $\mathrm{Diff}_0(S)$ are infinite-dimensional. To contend with this matter, in Section 4 we introduce notions of complex Lie derivative and transport through paths of complex vector fields, which allow to construct local deformations through ``double isotopies" (also see the introduction to Section 4). Defining the transport amounts to solving a Cauchy problem (see Equation (\ref{eq: transport problem})), which is analytically not so well-behaved. Precisely, the Cauchy problem fails H{\"o}rmander's criterion for local solvability (see Section \ref{subsec: trans}).

Using complex Lie derivatives and transport, we manage to take a finite-dimensional quotient $\Sph(S)$ of the space $\CAS^*(S)$ for which the fibers roughly correspond to double isotopies of the cubic differentials. The space $\Sph(S)$ keeps track of the geometric information of the complex affine spheres, as summarized in the following theorem (corresponding to the statements in Section \ref{sec: CAS space} of the text).

\begin{thmb}\label{thm b}
The space $\Sph(S)$ admits the structure of a finite dimensional complex manifold for which the quotient $\CAS^*(S)\to \Sph(S)$ is holomorphic. Moreover, $hol$ descends holomorphically to a map $hol\colon \Sph(S)\to \chi(\pi_1(S), \SL(3,\C))$ and $\hat \pi$ descends to a local biholomorphism $\pi\colon \Sph(S)\to \mathcal M_3(S)\times \mathcal M_3(\overline S)$.
\end{thmb}
As we show in Section 6, the map ${\mathcal L}_{\C}$ in Theorem \ref{thm A} is the composition of $hol$ with a local inverse of $\pi$. 

A natural question is whether $ {\mathcal L}_{\C}$ admits a global extension. The following statement provides evidence towards a negative answer.
\begin{thmc}
    There is no holomorphic map $$\mathcal{CD}(S)\times \mathcal{CD}(\overline{S})\to  \mathcal{CD}(S)\times \mathcal{CD}(\overline{S})\times C^\infty(S,\C)$$ with image in $\CAS(S)$ that inverts $\hat \pi$ and that coincides with the real affine spheres parametrization on the real locus.
\end{thmc}
   In particular, $\hat \pi\colon \CAS^*(S)\to \CD(S)\times\CD(\overline S)$ has no globally defined right inverse extending the real affine spheres parametrization.

 For input data $(c,\overline{c},\Q,-\overline{\Q})$, Equation (\ref{eqn: firstgauss}) is exactly the structural equation for minimal Lagrangians in $\mathbb{CH}^2$. Theorem C is a consequence of the fact that for $(c,\overline{c},\Q,-\overline{\Q})$, as $\Q$ grows larger the linearization of Equation (\ref{eqn: firstgauss}) eventually develops a kernel, which was proved by Huang-Loftin-Lucia in their paper about minimal Lagrangians \cite{HLM}. More discussion in Section \ref{sec: ThmD}.

\subsection{Analytic results}
As indicated above, the study of affine spheres and proving Theorem A require studying Equation (\ref{eqn: firstgauss}). In Section 5, we compute in local coordinates that if $g=\lambda dzd\overline{w}$, and $\mu$ is the Beltrami form of $w$ as a function of $z$, i.e.,  $\partial_{\overline z} w=\mu \partial_z w$, then the Laplace operator of $g$ is  
\begin{equation}\label{laplacian}
    \Delta_g = \frac{1}{\lambda\partial_{\overline z} \overline{w}}{\partial}_{\overline{z}}\left({\partial}_{ z}-\overline{\mu}{\partial}_{\overline{z}}\right).
\end{equation}
Note that the factor ${\partial}_{ z}-\overline{\mu}{\partial}_{\overline{z}}$ is the complex conjugation of the Beltrami operator. Since $|\mu|<1$, $\Delta_g$ is a second order complex elliptic operator, i.e., it's principle symbol never vanishes. The general class of such PDE's is not so well studied (see \cite[\S 7.8]{Ast} and \cite[\S 5]{Beg} for references), and hence we have to develop some general results ourselves. As a start, drawing on the analysis for the Beltrami operator, originating from Ahlfors-Bers, we prove interior elliptic estimates for complex operators such as $\Delta_g$ (Proposition \ref{interiorestimate}), from which we deduce Fredholm properties (Proposition \ref{fredholmness}). We record one of our more substantial results as a theorem below. 

 To extend $\Omega$ to $\mathcal{M}_3(S)\times \mathcal{T}(\overline{S})$ and $\mathcal{M}_3(\overline{S})\times \mathcal{T}(S)$, we need to show that we have infinitesimally rigid complex affine spheres lying over these loci, and doing so boils down to studying the spectrum of $\Delta_g.$ For the theorem below, see the more precise statement in Theorem \ref{thm F extended}.
\begin{thmd}
The operator $L_{\hpair}: C^{\infty}(S,\C)\to C^\infty(S,\C)$ defined by
    $$  L_{\hpair} u = \Delta_{\hpair} u - 2u\ $$ 
is invertible for an open and dense subset of elements $\ccpair\in \CSCS$, including all the pairs $\ccpair$ such that $c_1$ and $\overline{c_2}$ induce the same real analytic structure on $S$.
\end{thmd}
Theorem D can be recast as saying that a dense class of Bers metrics are infinitesimally rigid in their conformal class. In \cite{BEE}, it is proved that Bers metrics form a connected component of the space of positive complex metrics of constant curvature $-1$; this result is consistent with that one. 

The proof of Theorem D falls outside the standard analytic toolkit: we use a continuity argument, relying on the transport technique defined in Section 4 and on one of the main theorems proved in the first author's work on variations of positive complex metrics \cite[Theorem A]{ElE}, to connect pairs $(c_1, \overline{c_2})$ inducing the same real analytic structure to pairs $(c, \overline c)$, for which the result of Theorem D is elementary. In order to use the transport technique, we need to establish a regularity result, which is independently significant.

\begin{lem}\label{lem: analyticity}
       Assume that $g$ is a real analytic positive complex metric on $U\subset \C$, let $V\subset \C\times \C^2$ be an open domain, let $F:U\times V \to \C$ be a real analytic function that is linear in the second two coordinates. Suppose that  $u\in C^2(U,\C)$ with $(u,Du)(U)\subset V$ satisfies the inhomogeneous second order complex elliptic linear equation $$\Delta_g u +F(z,u,Du)=0.$$ 
Then $u$ is real analytic.
\end{lem}
We prove Lemma \ref{lem: analyticity} by adapting a proof of an analogous result for real elliptic PDE's, which is attributed to Kato \cite{Kato}, Hashimoto \cite{Hash}, and Blatt \cite{Blatt}. There are some genuine differences in our complex setting (see Remark \ref{rem: differences}), but still the idea pushes through.

Returning to Theorem D, we believe that $L_{\hpair}$ is an isomorphism for all Bers metrics, but we didn't need to prove the result in this generality. While Theorem D shows that $2$ is not in the spectrum, we do not know whether $u\mapsto \Delta_h u - ku$ is an isomorphism for every $k>0$. Studying the spectrum of $\Delta_h$ further seems interesting. 


\subsection{Outline of the paper} Toward understanding the main ideas, on a first reading one might want to focus on Sections 2, 3, and 6, and take the more analytic results from Sections 4 and 5 (such as Theorem D) for granted. Theorems A, B, and C are proved in Section 6.

In Section 2 we give the preliminaries on Teichm{\"u}ller theory and complex metrics, among other things. Section 3 contains the geometry that underpins the rest of the paper: we develop a general theory of complex affine immersions, define complex affine spheres, and investigate their basic properties.   We also give a number of explicit examples of complex affine immersions and complex affine spheres. Key results include Theorem \ref{thm: integrating GE for CAS}, which shows that every positive hyperbolic complex affine sphere determines a pair of complex structures and cubic differentials, and Corollary \ref{cor: PDE of Gauss eq} , which shows that if one can solve Equation (\ref{eqn: firstgauss}) built out of that data, then one can recover a complex affine sphere. 

In Section 4, which can be read independently of the rest of the paper, we develop the notion of complex Lie derivatives and transport along complex vector fields. These tools are used in the proof of Theorem D, and then eventually used to show in the proofs of Theorems A and C that local double isotopy deformations preserve the holonomy classes of infinitesimally rigid complex affine spheres. 

In Section 5, which is independent of the theory of affine spheres, we carry out our general analysis of complex elliptic operators. We first derive our expression for $\Delta_g$, prove interior elliptic estimates, and Fredholm properties. We then prove Lemma \ref{lem: analyticity} and Theorem D.

With the essential tools developed, in Section 6 we return to complex affine spheres. Here we formalize the strategy outlined in Section \ref{subsec: CAS space}, construct the moduli space of infinitesimally rigid complex affine spheres (thus proving Theorem B), and then give the proof of Theorem A. We then prove Theorem C in Section \ref{sec: ThmD}.

\subsection{Related work}
There have been a number of works related to representations of surface groups into complex Lie groups close to the Hitchin locus, although the perspectives of all of the works so far are distinct from ours. See for example \cite{DS1}, \cite{DS2}, \cite{ADL}, \cite{AMTW}, \cite{Dav}, and in particular \cite{DS2}, where Dumas-Sanders find a local continuation of Bers' theorem in a setting that falls more in line with the theory of Anosov representations.

We could frame the first half of Section 3 as the formulation of a basic theory of affine real submanifolds of dimension $n$ inside $\mathbb{C}^{n+1}$, which then leads to the definition of complex affine spheres. We'd like to point out that there is an established theory of complex affine differential geometry, but centered around complex hypersurfaces in $\C^{n+1}$, developed in \cite{Abe} and \cite{DVV} among other sources.

After writing this paper, we were notified about the paper \cite{Kim}, in which Kim studies a complex extension of the Riemannian Laplacian, which seems to agree with our Bers Laplacians, and studies their determinants.
Our analytic results, such as elliptic estimates and Fredholmness for our Laplacians, Lemma \ref{lem: analyticity}, and especially Theorem D, can be seen as contributing to the study initiated in \cite{Kim}.

Around the end of writing this paper, we learned of independent work of Rungi and Tamburelli \cite{RT}, who introduce complex minimal Lagrangian surfaces in the bi-complex hyperbolic space, which generalize minimal Lagrangian surfaces in $\mathbb{CH}^2$, hyperbolic affine spheres in $\R^3$, and Bers
embeddings in the holomorphic space form $\mathbb{\CP}^1 \times\mathbb{\CP}^1 \backslash \Delta$. In \cite{RT}, the authors define what they call $\mathrm{SL}(3,\C)$-bi-complex Higgs bundles, and they give a parameterization of a subset of $\mathrm{SL}(3,\C)$-Anosov representations by an open set in $\mathcal{M}_3(S)\times \mathcal{M}_3(\overline{S})$, from which they deduce that this subset of representations is endowed with a bi-complex structure. The approach of their work seems to give a different point of view and it would be interesting to analyze the interplay further. Nevertheless, we underline that the two works are completely independent.

\subsection{Acknowledgements}
We thank Kari Astala, Simon Blatt, and Peter Smillie for helpful discussions and correspondence. N.S. thanks the organizers and participants of the conference \emph{Advances in Higgs Bundles}, which took place at the Brin Mathematics Research Center at the University of Maryland in April 2024, for conversations and sharing references from which this work benefited. We also thank the Geometry and Topology group in Luxembourg for the support and for creating such a nice working environment. 

Both authors are funded by the FNR grant O20/14766753, \textit{Convex Surfaces in Hyperbolic Geometry.}

\section{Preliminaries}
For the rest of the paper, let $S$ be a closed oriented surface of genus $\mathrm g\geq 2,$ and let $\overline{S}$ be that same surface with the opposite orientation. Throughout the paper, we use $W^{k,p}(\cdot)$, $k\geq 1,$ $1\leq p \leq \infty$, for Sobolev spaces.
As we will see, the analysis of the equations that we will eventually encounter require precise choices of Sobolev spaces.
\subsection{Teichm{\"u}ller space}
Fix a basepoint complex structure on $S$ that's compatible with the orientation, say $c_0,$ and let $\mathcal{K}$ be the canonical bundle. A \textbf{Beltrami form} $\mu$ on $S$ is an $L^\infty$-measurable section of the bundle $\mathcal{K}^*\otimes \mathcal{K}^{-1}$ with $L^\infty$ norm strictly less than $1$. Recall that marked complex structures are parametrized by Beltrami forms. Explicitly, given a Beltrami form $\mu$ on $(S,c_0)$, the Measurable Riemann Mapping Theorem produces a unique pair consisting of a compatible complex structure $c_\mu$ together with a quasiconformal map $f_\mu:(S,c_0)\to (S,c_\mu)$. We say that a complex structure is in $W^{k,p}$ if the Beltrami form is $W^{k,p}$. For a positive integer $l>1$, denote by $\mathcal{C}^l(S)$ the space of $W^{l,\infty}$ complex structures on $S$ that are compatible with the orientation. We write $\mathcal{C}(S)$ for the $C^\infty$ complex structures. Let $\mathrm{Diff}_+(S)$ be the space of orientation preserving diffeomorphisms of $S$ and let $\mathrm{Diff}_0(S)$ be the normal subgroup of diffeomorphisms isotopic to the identity. The quotient $\faktor{\mathrm{Diff}_+(S)}{\mathrm{Diff}_0(S)}$ is the mapping class group $\mathrm{MCG}(S).$ 
\begin{defn}
    The \textbf{Teichm{\"u}ller space} of marked complex structures on $S$ is the quotient $\mathcal{T}(S)=\faktor{\mathcal{C}(S)}{\mathrm{Diff}_0(S)}$.
\end{defn}

For $l<\infty,$ $
\mathcal C^l(S)$ is a complex Banach manifold modelled on the space of $W^{l,\infty}$ Beltrami forms, and for $l=\infty$, $\mathcal{C}(S)$ is a complex Fréchet manifold, with semi-norms coming from the $W^{l,\infty}$ norms (by Sobolev embedding theorems, this Fréchet structure is biholomorphic to all of the more common ones). The following is due to Earle-Eells. \cite{EE}.
\begin{thm}[Earle-Eells]\label{thm: earle eells}
  Under the quotient topology, $\mathcal{T}(S)$ inherits a complex structure with respect to which the map from $\mathcal{C}(S)\to \mathcal{T}(S)$ is holomorphic and a principle $\mathrm{Diff}_0(S)$-fiber bundle.
\end{thm}
Throughout the paper, we insist on viewing $\mathcal{C}(S)$ as a Fr{\'e}chet manifold, so that we can directly apply the result of Earle-Eells.

Note that the complex structure on the Teichm{\"u}ller space $\mathcal{T}(\overline{S})$ of the oppositely oriented surface $\overline{S}$,  identifies with the complex structure conjugate to the usual one on $\mathcal{T}(S)$.

\subsection{Holomorphic Riemannian manifolds}
\label{section: holo Riem mflds}
Let $\mathbb X$ be a complex manifold with almost complex structure $\mathbb J$. A \textbf{holomorphic Riemannian metric} $\inners$ on $\mathbb X$ is a holomorphic assignment of a symmetric non-degenerate $\C$-bilinear form on each fiber of the holomorphic tangent bundle of $\mathbb X$, or more concisely a holomorphic nowhere degenerate section of $Sym_2(T^{1,0}(\mathbb X))$. These objects have been widely studied in literature with different approaches, see for instance \cite{DZ}. 

A holomorphic affine connection $D$ on $T\mathbb{X}$ is an affine connection such that $D\mathbb J=0$ and for local holomorphic vector fields $Z_1,Z_2$ one has $D_{\mathbb JZ_1} Z_2=\mathbb J\left( D_{Z_1} Z_2\right)$. Holomorphic Riemannian manifolds $(\mathbb X, \inners)$ have a natural notion of \textbf{Levi-Civita connection $D$}, which is a holomorphic affine connection on $T\mathbb X$ that makes the metric parallel. Such a connection induces a natural notion of $\C$-multilinear curvature tensor and, for all complex vector subspace $V< T_p \mathbb X$ such that $dim_{\C} V=2$ and that $\inners|_{V}$ is non-degenerate, one can define the sectional curvature.

Given a complex semisimple Lie group, its complex Killing form extends to the whole tangent bundle by left or right translation and hence determines a holomorphic Riemannian metric $\inners$. For $G=\mathrm{SL}(n, \C)$, at $A\in \mathrm{SL}(n,\C)$, the metric is given by, for $V,W\in T_A \mathrm{SL}(n, \C)$, $\langle V,W\rangle|_A= 2n \mathrm{tr}(A^{-1}V A^{-1} W)$.  We highlight two examples that arise from taking quotients of complex Lie groups.
\begin{enumerate}[leftmargin=*]
    \item $\mathbb C^n$ with the canonical symmetric $\C$-bilinear form given by $\langle\underline z, \underline w\rangle=\sum_{k=1}^n z_k w_k$, for $\underline z =(z_1,\dots, z_n),\underline w =(w_1,\dots, w_n)\in \C^n$. 
     $(\mathbb \C^n, \inners)$ is the unique simply connected complete holomorphic Riemannian metric with constant sectional curvature zero \cite[Theorem 2.6]{BEE}.
\item The hyperboloid
\[
\mathbb X_n =\{\underline z \in \C^{n+1}\ |\ \langle \underline z, \underline z \rangle =-1 \}
\]
is a complex submanifold of $\mathbb C^n$ and it inherits a holomorphic Riemannian metric, under which it identifies with $\faktor{\mathrm{SO}(n+1,\C)}{\mathrm{SO}(n,\C)}$. For $n\ge 2$, $\mathbb X_n$ is the unique complete simply connected holomorphic Riemannian manifold with constant sectional curvature -1 \cite[Theorem 2.6]{BEE}. Moreover, for all $p,q$ with $p+q=n$, $\mathbb{H}^{p,q}=\faktor{\mathrm{SO}(p,q)}{\mathrm{SO}(p)\times \mathrm{SO}(q)}$ embeds isometrically in $\mathbb X_n$ in a unique way up to composition with ambient isometries. 
\end{enumerate}

Observe that, since all non-degenerate bilinear forms on $\C ^n$ are isomorphic, one can define several isometric models for $\C^{n+1}$ and $\mathbb X_n$ by changing the ambient bilinear form. For instance, in Example \ref{bersmaps}, we will refer to the isometric model for $(\C^{n+1}, \inners)$ given by $\C^{n,1}$, namely $\C^{n+1}$ endowed with the bilinear form $\langle \underline{z}, \underline w\rangle_{n,1}= -z_{n+1}w_{n+1}+ \sum_{k=1}^n z_k w_k$. Via the isometry between these two spaces one gets that $\mathbb X_n$ is equivalent to $(\hat {\mathbb X}_{n}, \inners_{n,1})$ with 
\begin{equation}
\hat {\mathbb X}_{n}=\{\underline z \in \C^{n+1}\ |\ \langle \underline z, \underline z\rangle_{n,1}=-1\}\ .
\end{equation}

Finally, $\mathbb X_2$ models the space of oriented geodesics of hyperbolic space $\mathbb{H}^3$, which will be relevant below. Indeed, identifying $\mathrm{SO}(3,\C)$ with $\PSL(2,\C)$ and viewing $\mathbb X_2$ as $\faktor{\mathrm{PSL}(2,\C)}{\mathrm{SO}(2,\C)}$, $\mathrm{SL}(2,\C)$ acts transitively on the space of oriented geodesics, and the stabilizer of any oriented geodesic is conjugate to $\mathrm{SO}(2,\C)$. By identifying a geodesic in $\mathbb{H}^3$ with its endpoints in $\mathbb{CP}^1$, we obtain a biholomorphism from $$\mathbb X_2\to \mathbb{G}:=\mathbb{CP}^1\times \mathbb{CP}^1\backslash \Delta.$$ Under this biholomorphism,  the pushed-forward holomorphic Riemannian metric has the following description: given any complex affine chart $(U,z)$ on $\CP^1$, the metric on $(U\times U\setminus \Delta, (z_1,z_2))$  has the form \[
 \inners_{\mathbb G}= -\frac 4 {(z_1-z_2)^2} dz_1\cdot dz_2 \ .
 \]
See \cite[Section 2.4]{BEE} for details. 
Several relations between immersions of surfaces into $\mathbb H^3$ and $\mathbb G$ are treated in \cite{EleSeppi}.

\subsection{Complex metrics}
\label{subsection: complex metrics}
Complex metrics first appeared in a paper by the first author on immersions of (real) smooth manifolds into holomorphic Riemannian manifolds \cite{BEE}. Throughout the paper, given a smooth manifold $M,$ we denote the complexified tangent bundle $TM\otimes_{\R}\C$ by $\C TM,$ and likewise for the complexified cotangent bundle $\C T^*M$.
\begin{defn}
    A \textbf{complex metric} on a smooth (real) manifold $M$ is a smooth section of $\mathrm{Sym}^2(\mathbb C TM)$ that determines a non-degenerate symmetric bilinear form in each fiber.
\end{defn}
Of course, the bilinear extension of a Riemannian metric to the complexified tangent bundle is a complex metric. Throughout the paper, we will view a complex metric on a manifold $M$ as the same thing as an invariant complex metric on the universal cover $\widetilde{M},$ and we won't distinguish notation.

A lot of the usual constructions in Riemannian geometry extend to the setting of complex metrics and the proofs apply verbatim. Here are a few objects and facts that we will use in this paper.
\begin{enumerate}[leftmargin=*]
    \item Every complex metric $g$ has a natural \textbf{Levi-Civita connection}, namely an affine connection $\nabla^g\colon \Gamma( \C TM) \to \Gamma(\mathrm{End}(\C TM))$ uniquely determined by the fact that it is torsion-free and compatible with the metric, i.e. $\nabla^g g=0$.
    \item  Through $g$ and its Levi-Civita connection $\nabla^g$, one can define the $(0, 4)$-type and the $(1, 3)$-type \textbf{curvature tensors}, which we both denote by $\mathrm R^g$ when there is no ambiguity, as 
    \[
    \mathrm R^g(X,Y,Z,W)= g(\mathrm R^g(X,Y)Z, W):=g(\nabla^g_X \nabla^g_Y Z-\nabla^g_Y\nabla^g_X Z- \nabla^g_{[X,Y]}Z,W) 
    \]
    for all $X,Y,Z,W\in\Gamma(\C TM)$, which is $\C$-multilinear and has the usual symmetries as in the Riemannian case. If $X$ and $Y$ span a $g$-non-degenerate $2$-plane in $\C TM$, we have the sectional curvature 
     \[\mathrm K(X,Y)=\frac{\mathrm R^g(X,Y,Y,X)}{g(X,X)g(Y,Y)-(g(X,Y))^2}.\]
     When $M$ is a surface, this sectional curvature is called the \textbf{Gauss curvature} and defines a function $\mathrm K_g: M\to \C.$
    \item For any $C^1$ function $f\colon M\to \C$, one can define the \textbf{gradient} of $f$ as the unique vector field $\nabla^g f$ defined by $df=g(\nabla^g f, \cdot)$. If $f$ is also $C^2$, we can define the \textbf{Hessian} and the \textbf{Laplacian} of $f$ as
    \begin{align*}
        &\mathrm{Hess}_g(f) (X,Y)= g(\nabla^g_X \nabla^g f, Y),\\
        &\Delta_g(f) = tr_g(\mathrm{Hess}_g(f)).
    \end{align*}
    \item Assume $M$ is a surface, and let $\varrho\colon M\to \C^*$. Then, the conformal metric $\hat g= \varrho \cdot g$ satisfies $\Delta_{\hat g}= \frac 1 {\varrho} \Delta_g$ and
    \begin{equation}\label{eq: conf curvature} 
    \mathrm K_{\hat g}= \frac 1 \varrho (\mathrm K_g-\frac 1 2\Delta_g \log(\varrho)).
   \end{equation}
    While the function $\log(\varrho)$ is only locally well-defined, up to choices, $\Delta_g \log(\varrho)$ is globally well-defined.
\end{enumerate}
In the following, we restrict to the case of surfaces $M=S$. The fact that a complex metric $g$ on $S$ is non-degenerate implies that $g$ has two isotropic directions at each point. In other words, for each point $p\in S$, the set $\{v\in \C T_p S\ |\ g_p(v,v)=0 \}$ consists of two complex lines of $\mathbb C T_p S$, corresponding to two points in $\mathbb P \C T_pS$. Observe that $\mathbb P T_p S$ embeds into the sphere $\mathbb P \C T_pS$ as an equatorial circle, which cuts $\mathbb P \C T_pS$ into two connected components.

\begin{defn}\label{def: positive}
    A complex metric $g$ is \textbf{positive} if for each point $p\in S$, no isotropic direction of $g_p$ lies in $\mathbb P T_p S$, and the two isotropic directions are contained in distinct components of $\mathbb P \C T_pS\setminus \mathbb P T_p S$.
\end{defn}
Riemannian metrics complexify to positive complex metrics: the isotropic directions are the eigen-lines of the almost complex structure and are hence antipodal.

We denote the space of $C^{\infty}$ -- resp. $W^{l,2}$ -- positive complex metrics on $S$ by $\mathcal{CM}(S)$ -- resp. $\mathcal{CM}^l(S)$. Being open subsets of the space of $C^{\infty}$ -- resp. $W^{l,2}$ -- sections of $\mathrm{Sym}^2(\mathbb C TS)$, one can see that $\mathcal{CM}(S)$ is a complex Fréchet manifold and $\mathcal{CM}^l(S)$ is a complex Banach manifold. 

We further define $\mathcal{CM}_{-1}(S)$ and $\mathcal{CM}_{-1}^l(S)$ to be the subsets of positive complex metrics of constant curvature $-1$ in $\mathcal{CM}(S)$ and $\mathcal{CM}^l(S)$ respectively.

The isotropic directions of a positive complex metric determine a pair of oppositely oriented complex structures $c_1$ and $\overline{c_2}$ on $S$. In fact, if $z$ and $\overline w$ are local holomorphic coordinates for $c_1$ and $\overline{c_2}$ respectively, then $g(\partial_{\overline z}, \partial_{\overline z})=g(\partial_w, \partial_w)=0$, and $g$ can be locally written as $g=\lambda dz d\overline w$, for some non-vanishing complex valued function $\lambda$ (see \cite[Section 6.1]{BEE} for details).

This defines a projection map 
\[
(\vb*c_+, \vb* c_-)\colon \mathcal{CM}(S)\to \mathcal C(S)\times \mathcal C(\overline S)
\]
whose fibers correspond to conformal classes of positive complex metrics.

The locally defined basis for $\C TS$ given by $\{\partial_{w},\partial_{\overline{z}}\}$ will be particularly useful for computations, so we record a few elementary facts here. Let $\nabla^g$ be the Levi-Civita connection for $g$, and let $\Pi_1$ and $\Pi_2$ be the projections of $\C TS$ onto the isotropic directions of $g$ for $c_1$ and $\overline{c_2}$ respectively, which are locally spanned by $\partial_{\overline{z}}$ and $\partial_w$ respectively.
\begin{lem}\label{parallellem}
 For any locally defined complex vector field $X$, $\nabla_X^g \partial_{\overline{z}}$ is parallel to $\partial_{\overline{z}}$, and  $\nabla_X^g \partial_{w}$ is parallel to $\partial_{w}$. 
\end{lem}
\begin{proof}
By metric compatibility, $g(\nabla_X^g \partial_{\overline{z}}, \partial_{\overline{z}})=\frac{1}{2}\nabla_X^g g(\partial_{\overline{z}},\partial_{\overline{z}})=0,$ and likewise for $\partial_{w}$.
\end{proof}
\begin{lem}\label{projectionlem}
$\nabla_{\overline{z}}^g\partial_w=\Pi_2([\partial_{\overline{z}},\partial_w])$ and $\nabla_{w}^g\partial_{\overline{z}}=\Pi_1([\partial_{w},\partial_{\overline{z}}]).$
\end{lem}
\begin{proof}
Since $\nabla^g$ is torsion free, $\nabla_{\overline{z}}^g\partial_w - \nabla_{w}^g\partial_{\overline{z}} = [\partial_{\overline{z}},\partial_w]$. We apply Lemma \ref{parallellem}.

\end{proof}

\subsection{Bers' theorem and Bers metrics}
In the introduction we defined quasi-Fuchsian representations. Bers' Simultaneous Uniformization Theorem is more formally stated below. If $\rho$ is quasi-Fuchsian, preserving a Jordan curve $\gamma\subset \mathbb{CP}^1$, the domains complementary to $\gamma$ are called the domains of discontinuity. Recall that, identifying $\pi_1(S)$ with the group of covering transformations of $\widetilde{S}\to S$, given an action $\rho$ of $\pi_1(S)$ on a space $X,$ we say that $f:\widetilde{S}\to X$ is $\rho$-$\textbf{equivariant}$ if for all $\gamma\in \pi_1(S)$, $f\circ \gamma = \rho(\gamma)\circ f$. 
\begin{thm} [Bers' Simultaneous Uniformization Theorem \cite{SU}]
\label{thm: Bers thm}
	For all $(c_1, \overline{c_2})\in \mathcal{C}^l(S)\times \mathcal{C}^l(\overline{S})$, there exists a quasi-Fuchsian representation $\rho\colon \pi_1(S)\to \PSL(2, \mathbb C)$, unique up to conjugation, with domains of discontinuity $\Omega_\rho^+$ and $\Omega_\rho^-$, together with unique $\rho$-equivariant holomorphic  diffeomorphisms $$\vb*{f_+}(c_1,\overline {c_2})\colon (\widetilde S, c_1) \to \Omega_\rho^+,\qquad \qquad  \vb*{\overline{f_-}}(c_1,\overline {c_2})\colon (\widetilde S, \overline{c_2}) \to \Omega_\rho^-.$$
	This correspondence determines a biholomorphism \[ \TSTS\xrightarrow{\sim}\mathrm{QF}(S).\]
\end{thm}
Fix a pair $(c_1,\overline{c_2})\in \CSCS$. The two maps $f_1= \f+\ccpair$ and $\overline{f_2}=\fm \ccpair$ define an immersion $(f_1, \overline{f_2})\colon \widetilde S\to \mathbb G\supset \Omega^+_\rho\times \Omega^-_\rho$, and the pull-back of $\inners_{
\mathbb G}$ is the positive complex metric 
\begin{equation}
\label{eq: def Bers metric}
\hpair:= -\frac 4 {(f_1-\overline{f_2})^2}df_1d\overline{f_2}
\end{equation}
that, as a result of Theorem 6.7 in \cite{BEE}, has constant curvature $-1$. We refer to the metrics obtained in this fashion as \textbf{Bers metrics}. In \cite{BEE}, the first author with Bonsante proved that one can uniformize positive complex metrics via Bers metrics.

\begin{prop}[Theorem 6.11 in \cite{BEE}]
    A complex metric $g$ is positive if and only if there exists a (unique) Bers metric $\vb*h_{\ccpair}$ and a function $\varrho\colon S\to \mathbb C^*$ such that 
    \[g=\varrho\cdot \vb*h_{(c_1, \overline{c_2})}\ .\] 
\end{prop}
As proved in \cite{BEE}, in the smooth category, Bers metrics form a connected component of the subset of smooth complex metrics of constant curvature $-1$, and $\ccpair \mapsto \vb*h_{(c_1,\overline{c_2})}$ defines a bijective correspondence between pairs of complex structures and Bers metrics. 

In the paper \cite{ESbelt}, we proved a holomorphic dependence result for solutions to the Beltrami equation in higher Sobolev spaces \cite[Theorem A]{ESbelt}. Our work there shows that if the complex structures are $W^{l,\infty}$, then the Bers metric, viewed as a metric on $S$, is a $W^{l,2}$ bilinear form on $\C TS$. Moreover, we proved the following.

\begin{prop}[Theorem B in \cite{ESbelt}]\label{appendixholomorphicity}
The maps $\mathcal{C}^l(S)\times \mathcal{C}^l(\overline{S})\to \mathcal{CM}^{l}(S)$ and $\mathcal{C}(S)\times \mathcal{C}(\overline{S})\to \mathcal{CM}(S)$ defined by $(c_1,\overline{c_2})\mapsto \vb*{h}_{(c_1,\overline{c_2})}$, with image contained in $\mathcal{CM}_{-1}^l(S)$ and $\mathcal{CM}_{-1}(S)$ respectively, are holomorphic.
\end{prop}


Finally, we record some results from \cite{ElE} on holomorphically perturbing positive complex metrics, which we will use in the proof of Theorem D in Section \ref{sec: thmE}. For all $c\in \mathcal C(\mathrm S)$ (resp. $\overline c\in \mathcal C(\overline{\mathrm S})$), denote by $\mathrm{QD}(c)$ (resp. $\mathrm{QD}(\overline c)$) the space of holomorphic quadratic differentials for $c$ (resp. $\overline c$). If $g$ is a positive complex metric, then by openness of the positive property, for all $q_1\in \mathrm{QD}(c)$ and $\overline{q_2}\in \mathrm{QD}(\overline c)$ small enough, $g+q_1$ and $g+\overline{q_2}$ are positive metrics as well.
\begin{prop}
\label{prop: Kg and Kg+q}
    Let $g$ be a positive complex metric, with $(c_1, \overline{c_2})= (\vb*{c_+}, \vb*{c_-})(g)$, and $q_1\in \mathrm{QD}(c)$ and $\overline{q_2}\in \mathrm{QD}(\overline c)$ so that $g+q_1$ and $g+\overline{q_2}$ are both positive. Then $$\mathrm K_g= \mathrm K_{g+q_1}, \quad \mathrm{ and }\quad \vb*{c_+}(g+q_1)=c_1,$$ and $$\mathrm K_g=\mathrm K_{g+\overline{q_2}},\quad \mathrm{ and }\quad \vb*{c_-}(g+\overline{q_2})= \overline{c_2}\ .$$
\end{prop}
This Proposition is stated as Proposition 3.8 in \cite{ElE} for Bers metrics, but the proof works verbatim for any positive complex metric.
Moreover, in \cite{ElE} it is proved that deformations from $h$ to $h+q_1$ and $h+\overline{q_2}$ descend to give complex charts on Teichm{\"u}ller space.
\begin{thm} [Theorems A and B in \cite{ElE}]\label{thm: ElE deform}
        Let $\ccpair \in \CSCS$, $h=\hpair$, and let $U_{\ccpair}\subset \mathrm{QD}(c_1)$ be an open neighborhood of zero such that $\hpair+q_1$ is a Bers metric. The map
            \begin{align*}
                U&\to \mathcal T(\overline { S})\\
                q_1&\mapsto [\vb*{{c_-}}(h+q_1)]
            \end{align*}
        is a biholomorphism onto its image. The analogous statement holds for deformations of the form $\hpair\mapsto\hpair+\overline{q_2}$.
\end{thm}
Although we will not use this fact in the paper, it is worthwhile to note that the charts from the theorem above are affine reparametrizations of the Schwarzian parametrization.

\section{Complex affine spheres}
In this section we develop the theory of complex affine immersions and complex affine spheres as an extension of the theory of (real) affine immersions and affine spheres. We don't recall the real theory here, but we highlight it as a special case of our construction (see Example \ref{ex: ras}, and also Section \ref{sec: LL parametrization} later on). Sources on the real theory include \cite{NS}, \cite{Lof}, and \cite{Lab2}.

In the section below, if regularity is unspecified, we will assume that all the objects are regular enough for the operations to be well-defined.
\subsection{Complex affine immersions}
Let $\mathbb X$ be a complex manifold of complex dimension $n=m+1$ with almost complex structure $\mathbb{J}$ and torsion free holomorphic affine connection $D$. Starting in Section \ref{section: complex affine spheres}, $\mathbb X=\mathbb{C}^3,$ equipped with the Levi-Civita connection of its holomorphic Riemannian metric, so it might be helpful to have $\mathbb X=\C^n$ in mind as the typical target space. We denote by $\text{Aff}(\mathbb X, D)$ the space of biholomorphisms of $\mathbb X$ that preserve $D$.

Now, let $M$ be a (real) smooth manifold of dimension $m$, let $\widetilde M$ be its universal cover, and identify $\pi_1(M)$ with the deck group of the covering map. From now on until Section \ref{section: Gauss, Codazzi, integration}, fix a representation $\rho\colon \pi_1(M)\to \text{Aff}(\mathbb X, D)$ and let $\sigma\colon \widetilde M\to \mathbb X$ be a equivariant map. We refer to $\rho$ as the \textbf{holonomy} of $\sigma$. Note that one can see immersions $M\to \mathbb X$ as equivariant immersions $\widetilde M\to \mathbb X$ with the trivial holonomy $\rho\equiv id_{\mathbb X}$. 

The tangent map $\sigma_*\colon T\widetilde M \to \sigma^*(T\mathbb X)$ extends to a $\C$-linear map $\sigma_*\colon \C T\widetilde M \to \sigma^*(T\mathbb X)$ by defining $\sigma_*(v+iw):=\sigma_*(v)+\mathbb J(\sigma_*(w))$. In the following, with $\sigma_*$ we will always denote the extension to $\mathbb C T\widetilde M$.

\begin{defn}
    We say that an immersion $\sigma:\widetilde M\to \mathbb X$ is \textbf{admissible} if the induced map $\sigma_*: \C T\widetilde M\to \sigma^*T\mathbb X$ is injective. 
\end{defn}
Admissibility means that $\sigma_*$ maps $\R$-linearly independent vectors in $T_p \widetilde M$ to $\C$-linearly independent vectors in $\sigma^*(T\mathbb X)$. 
The admissibility condition allows the construction of several geometric tensors associated with the immersion, imitating the theory of hypersurfaces in smooth manifolds.

Since $\sigma$ is $\rho$-equivariant, $\pi_1(M)$ acts on the bundle $\sigma^*(T\mathbb X)$ through $\rho$ preserving both $\sigma^*D$ and $\sigma^*\mathbb J$: the quotient defines a bundle $E_\rho\to M$ that inherits both a connection, which we still denote with $\sigma^*D$, and a complex structure that we denote with $i$ to relax the notation.
Assuming $\sigma$ is admissible, $\mathbb C T M$ can be seen as a complex subbundle of $E_\rho$ with complex co-rank 1. We abuse notation slightly and write this subbundle as $\C T M\subset E_\rho$ when the context is clear.

\begin{defn}
A $\rho$-equivariant \textbf{complex affine immersion} of $\widetilde M$ in $\mathbb X$ is a pair $(\sigma, \xi)$ where $\sigma\colon M\to \mathbb X$ is a $\rho$-equivariant admissible immersion, and $\xi$ is a smooth section of $E_\rho$ called the \textbf{transversal vector field}, such that $\mathbb C T\widetilde M\oplus \mathrm{Span}_{\mathbb C }(\xi)= E_\rho$. 
\end{defn}
Compare with \cite[Definition II.1.1]{NS}.

A \textbf{complex affine connection} $\nabla$ on $M$ is a $\mathbb C$-bilinear affine connection on the bundle $\mathbb C TM$. Observe that all affine connections on $T M$ can be extended by linearity to complex affine connections on $\C TM$. Notice that this notion, defined on smooth manifolds, is very different from the notion of holomorphic connection on a complex manifold introduced above.

 According to the splitting $\mathbb C T M\oplus \mathrm{Span}_{\mathbb C} (\xi)$, we get that for all $\mathsf X \in \Gamma(T M)$, $Y\in \Gamma(\mathbb C T M)$
\begin{equation}
\label{eq: affinedecomposition}
    \begin{cases}
        (\sigma^*D)_{\mathsf X} Y&=: \nabla_{\mathsf X} Y + g(\mathsf X, Y)\cdot \xi\\
        (\sigma^*D)_{\mathsf X} \ \xi&=: -\mathsf S(\mathsf X) + \tau(\mathsf X) \cdot \xi.
    \end{cases}
\end{equation}
\begin{remark}
    One can see that:
\begin{itemize}[leftmargin=*]
        \item for all $X,Y\in \Gamma(\mathbb CT M)$, with $X=\mathsf X_1+i\mathsf X_2$, $\nabla_XY:=\nabla_{\mathsf X_1} Y+i\nabla_{\mathsf X_2} Y$ is a torsion-free complex affine connection on $\mathbb C T M$;
        \item $g$ extends to a symmetric $\mathbb C$-bilinear form on $\mathbb C T M$, which, by analogy with the theory of real affine spheres, we call the \textbf{affine fundamental form};
        \item $\mathsf S$ is linear, and it extends to a section of $\mathrm{End}(\mathbb C T M)$, called the \textbf{shape operator};
        \item $\tau$ is linear as well, and it extends to an element in $\Omega^1(\mathbb C T M)$, called \textbf{transversal connection form}.
    \end{itemize}
\end{remark}

\subsection{Blaschke affine immersions}
In this subsection we construct preferred transversal vector fields called affine normals. The theory here is roughly parallel to \cite[\S II.3]{NS}. {However, unlike the real case, we will see that there are topological obstructions to the existence of such an affine normal}.
First of all, let us see how the affine tensors defined in the previous subsection change according to a change in the transversal vector field.

\begin{lem}
\label{lemma: rescale affine normal}
Let $(\sigma, \xi)$ be a complex affine immersion, with $\nabla, g, \mathsf S, \tau$ as in Equation \eqref{eq: affinedecomposition}. 
Let $\hat{\xi}=\alpha\xi + \eta$ be a different transversal vector field for $\sigma$, with $\alpha\in C^\infty( M,\C^*)$ and $\eta\in \Gamma(\C T M)$. Then, the complex affine immersion $(\sigma, \hat \xi)$ induces $\hat \nabla, \hat g, \hat {\mathsf S}, \hat \tau$ according to: 
\begin{enumerate}
    \item  $\hat{g} =\frac 1 \alpha g $,
    \item $\hat{\nabla}_X Y = \nabla_X Y - \frac 1 \alpha g(X,Y)Z$,
      \item   $\hat{\mathsf S}(X)=\alpha \mathsf S(X)-\nabla_X \eta + \tau(X)\eta +\frac{\partial_X \alpha+ g(X,\eta)}{\alpha} \eta$,
    \item $\hat{\tau}(X)=\tau(X)+\frac{\partial_X \alpha+ g(X,\eta)}{\alpha}$.
\end{enumerate}
\end{lem}
\begin{proof}
For $X,Y, \eta\in \Gamma(TM)$, the proof follows by direct computation, as shown for the real case in \cite[Proposition II.2.5]{NS}. The equations for $X,Y\in \Gamma(\mathbb C  TM)$ follow by $\mathbb C$-linearity, while the result for $\eta\in \Gamma(\mathbb C TM)$ follows by applying transformations repeatedly, namely:
\begin{align*}
\xi& \quad \leadsto \quad \xi+ Re\left(\frac \eta \alpha\right) \quad \leadsto \quad -i \left(\xi+Re\left(\frac {\eta}{\alpha} \right) \right)+Im\left(\frac {\eta}{\alpha}\right) = -i \xi -i \frac{\eta}{\alpha}\quad \leadsto\\
&\leadsto \quad i\alpha \left( -i \xi -i \frac{\eta}{\alpha}\right) =\alpha\xi +\eta.
\end{align*}
\end{proof}

From now on, assume that $\mathbb X$ carries a \textbf{holomorphic $D$-flat volume form $\omega$}, i.e. a nowhere vanishing holomorphic $n$-form $\omega$ such that $D \omega=0$, and that our actions $\rho$ preserve this volume form. We denote by Aff$(\mathbb X, D, \omega)$ the space of biholomorphisms of $\mathbb X$ that preserve both $D$ and $\omega$.

Given $(\sigma, \xi)$ a complex affine immersion of $\widetilde M$ into $\mathbb X$, we define a nowhere vanishing section $\theta$ of $\bigwedge^{k} \C T^*\widetilde M$ on $\widetilde M$ by 
\begin{equation}
    \label{eq: def volume form}
    \theta(X_1,\dots, X_m) = (\sigma^*\omega)(X_1,\dots, X_m,\xi)\ .
\end{equation}
We call $\theta$ the \textbf{induced complex volume form}. By equivariance, the induced complex volume form $\theta$ descends to a volume form on $M$.

\begin{prop}\label{prop: nablatheta}
   For any $X\in \C TM,$ $\nabla_X \theta=\tau(X)\theta.$
\end{prop}
\begin{proof}
  The proof for $X\in \Gamma(TS)$ is identical to the proof of the analogous result for real affine immersions \cite[Proposition II.1.4]{NS}, then the statement follows by $\C$-linearity of $\nabla \theta$ and $\tau$.
\end{proof}

Proposition \ref{prop: nablatheta} and Equation \eqref{eq: affinedecomposition} tell us that $\theta$ is $\nabla$-parallel if and only if $(\sigma^*D)\xi$ preserves the sub-bundle $\mathbb C T\widetilde M$. 

We say that $\sigma$ is \textbf{non-degenerate} if, given a transversal vector field $\xi,$ the affine fundamental form $g$ is nowhere degenerate, i.e., it is a complex metric. By Lemma \ref{lemma: rescale affine normal}, all choices of affine transversal sections determine conformally equivalent fundamental forms, so this notion of non-degeneracy is independent of our choice of the transversal section.

\begin{remark}
Given a complex metric $g$ on a smooth manifold $\widetilde M$, on each simply connected chart with local coordinates $(x_i)_{i=1}^m$ one can define a local \textbf{metric volume form} $dV_g=\sqrt{det(g(\partial_{x_i}, \partial_{x_j}))_{i,j}}\cdot dx_1\wedge\dots\wedge dx_n$, which is well-defined up to a sign. Since $\widetilde M$ is simply connected, there always exists a global metric volume form on $\widetilde M$. 
\end{remark}

\begin{defn}
   Let $\rho\colon \pi_1(M)\to \text{Aff}(\mathbb X, D, \omega)$ be a homomorphism. A $\rho$-equivariant \textbf{Blaschke affine immersion} of $\widetilde M$ in $\mathbb X$ is an (equivariant) affine immersion $(\sigma, \xi)$ such that:
   \begin{enumerate}
       \item the affine fundamental form $g$ is nowhere degenerate, and there exists a global metric volume form $dV_g$ on $M$,
       \item $\nabla\theta=0$, hence $\tau=0$ by Proposition \ref{prop: nablatheta},
       \item $\theta=dV_g$.
   \end{enumerate} 
The transversal vector field $\xi$ is called \textbf{affine normal} and $g$ is called \textbf{Blaschke metric}.
\end{defn}

\begin{prop}
    \label{prop: blaschkenorm}
   Let $M=\widetilde M$ be simply connected. Given a non-degenerate affine immersion $(\sigma, \xi)$ of $\widetilde M$ into $\mathbb X$, there exists a transversal vector field $\hat\xi$, unique up to a sign, such that $(\sigma, \hat \xi)$ is a Blaschke affine immersion. $(\sigma, \hat \xi)$ is called the \textbf{Blaschke normalization} of $(\sigma, \xi)$.
\end{prop}
\begin{proof}
This proof is totally analogous to that of the corresponding result in the real case (see \cite[Theorem II.3.1]{NS}), but we include the full proof to show where the topological assumption is used.

Since $\widetilde M$ is simply connected, the nondegenerate affine fundamental form $g$ of $(\sigma, \xi)$ admits a global metric volume form $dV_{g}$. Given $\alpha\in C^\infty(M,\C^*)$ and $\eta\in \Gamma(\C TM)$, if we transform $\xi$ to $\hat{\xi}=\alpha\xi + \eta,$ the metric volume form $dV_{g}$ and the induced complex volume form $\theta$ of $(\sigma, \xi)$ change according to $dV_{\hat{g}}=\frac 1 \alpha dV_g$ and  $$\hat{\theta}(X,Y)=\omega(X,Y,\alpha\xi + \eta)=\omega(X,Y,\alpha \xi)=\alpha \theta(X,Y). \ $$ 
    Since $M$ is simply connected, there exists a function $\alpha\in C^\infty(\widetilde M,\C^*),$ unique up to multiplication by $-1$, such that $$\alpha^2 = \frac{dV_g}{\theta}.$$ Choosing such $\alpha$ guarantees that the condition (3) holds, so it remains to pick $\eta$ so that (2) holds. Since $g$ is non-degenerate, there exists a unique complex vector field $\eta$ such that, as $$g(\eta,\cdot)=-(\alpha\tau + d\alpha).\ $$ Using item (4) from Lemma \ref{lemma: rescale affine normal}, we see that $\hat{\xi}=\alpha \xi + \eta$ satisfies the requirements of the proposition. Observe that if we exchange $\alpha$ with $-\alpha,$ then $\eta$ becomes $-\eta,$ and hence $\xi$ is unique up to multiplication by $\pm 1$.
\end{proof}

\begin{remark}
Looking into the proof of Proposition \ref{prop: blaschkenorm}, one can see that, unlike the real case, for general $M$ and for a general representation $\rho\colon \pi_1(M)\to \mathrm{Aff}(\mathbb X, D, \omega)$, a $\rho$-equivariant non-degenerate affine immersion $(\sigma,\xi)$ might not admit a $\rho$-equivariant Blaschke normalization. 
    In fact, one needs to assume that $(M,g)$ admits a global volume form (equivalently, a $\pi_1(M)$-invariant metric volume form for $\widetilde g$ on $\widetilde M$), and - even assuming this - the function $\frac {dV_g}{\theta}$ might not admit a $\pi_1(M)$-invariant square root.

    Nevertheless, the obstructions only depend on the topological properties of $g$, $dV_g$, and $\theta$, so, if $\sigma$ is a deformation of a real affine immersion, then it admits a Blaschke normalization. 
\end{remark}
\begin{example}[Graph immersions] 
   Let $U$ a domain in $\R^n$ and $F:\R^n\to \C$ a $C^1$ function. A graph immersion is an admissible immersion $\sigma:U\to \C^{n+1}$ of the form $$\sigma(x_1,\dots, x_n) = (x_1,\dots, x_n, F(x_1,\dots, x_n)).$$ The choice of transversal vector field $\xi=(0,\dots, 0,1)$ yields $\mathsf S=0$ and $\tau=0$.
\end{example}
\begin{example}[Complex paraboloids]

    As a special example of a graph immersion, we have what we call complex paraboloids.   
   Consider, for instance, $F(x,y)=(x,y, \frac{1}{2}(x^2+e^{i\theta}y^2)).$ Taking $e^{i\theta}=1$ gives the well-known elliptic paraboloid, and taking $e^{i\theta}=-1$ gives a hyperbolic paraboloid. It is not hard to compute that $\xi=(0,0,e^{-i\theta})$ produces a transversal vector field with $\tau=0$ and $\theta = dV_g.$
\end{example}

\subsection{The fundamental theorem for complex affine immersions}
\label{section: Gauss, Codazzi, integration}
From now on, we will assume that $\mathbb X = \mathbb C^{n+1}$ endowed with the Levi-Civita connection of the holomorphic Riemannian metric, say $D$.
The standard equations from affine differential geometry admit complex extensions.
\begin{prop}
\label{prop: GCReqns}

   Consider an equivariant complex affine immersion $(\sigma, \xi)$ of $\widetilde M$ into $(\mathbb C^{m+1}, D)$ inducing the data $\nabla, g,\mathsf S, \tau$ on $M$ as in Equation \eqref{eq: affinedecomposition}.

  The following equations hold for all $X,Y,Z\in \Gamma(\mathbb C T M)$.
    \begin{align}
        \label{eq: GE generic} \text{Gauss equation:} &\qquad \mathrm R^{\nabla}(X,Y)Z=g(Y,Z)\mathsf S(X)-g(X,Z)\mathsf S(Y),\\
        \label{eq: Cod1 generic} \text{Codazzi equation 1:} &\qquad (\nabla_X g)(Y,Z)+\tau(X)g(Y,Z)=(\nabla_Yg)(X,Z)+\tau(Y)g(X,Z),\\
         \label{eq: Cod2 generic} \text{Codazzi equation 2:}& \qquad (\nabla_X \mathsf S)(Y)-\tau(X)\mathsf S(Y)=(\nabla_Y \mathsf S)(X)-\tau(Y)\mathsf S(X),\\
       \label{eq: Ricci generic} \text {Ricci equation:}&\qquad  g(X,\mathsf S(Y))-g(\mathsf S(X),Y)=d\tau(X,Y).
     \end{align}
\end{prop}
\begin{proof}
    The corresponding results hold for (real) affine immersions between real manifolds \cite[Theorem II.2.1]{NS}, as a result of the flatness of $D$. Indeed, one can see explicitly from Equation \eqref{eq: affinedecomposition} that for all $\mathsf X,{\mathsf Y},{\mathsf Z}\in \Gamma(T M)$

    \begin{equation}
    \label{eq: flatness-identities}
    \begin{split}
        \mathrm R^D(\mathsf X,{\mathsf Y}){\mathsf Z}=& \mathrm R^{\nabla}(\mathsf X,{\mathsf Y}){\mathsf Z}- g({\mathsf Y},{\mathsf Z})\mathsf S(\mathsf X)+g(\mathsf X,{\mathsf Z})\mathsf S({\mathsf Y}) \\
        &+(\nabla_{\mathsf X}g)({\mathsf Y},{\mathsf Z})\xi -(\nabla_{\mathsf Y}g)({\mathsf X},{\mathsf Z})\xi +g({\mathsf Y},{\mathsf Z})\tau({\mathsf X})-g({\mathsf X},{\mathsf Z})\tau({\mathsf Y}),\\
        \mathrm R^D({\mathsf X},{\mathsf Y})\xi&= -(\nabla_{\mathsf X}\mathsf S)({\mathsf Y})+(\nabla_{\mathsf Y}\mathsf S)({\mathsf X})- \tau({\mathsf X}) \mathsf S({\mathsf Y})+\tau({\mathsf Y})\mathsf S({\mathsf X}) \\
        &-g({\mathsf X}, \mathsf S({\mathsf Y}))\xi+g({\mathsf Y}, \mathsf S({\mathsf X}))\xi +d\tau({\mathsf X},{\mathsf Y})  \xi,
        \end{split}
    \end{equation}
    hence Equations \eqref{eq: GE generic}-\eqref{eq: Cod1 generic}-\eqref{eq: Cod2 generic}-\eqref{eq: Ricci generic} follow by $\mathrm R^D(\mathsf X,\mathsf Y)Z=0=\mathrm R^D(\mathsf X,\mathsf Y)\xi$. 
    Finally, the statement follows for all $X,Y,Z\in \Gamma(\C TS)$ by $\C$-linearity of $\nabla$ and of the tensors.
\end{proof}
Conversely, the immersion data $\nabla, g, \mathsf S,$ and $ \tau$ determine uniquely an affine immersion, as shown in the following statement. 
\begin{thm}
\label{thm: integration GC}
Let $M$ be a smooth real $m$-manifold.
Let $\nabla\colon \Gamma(\C TM)\to \Gamma(End_{\C} (\C TM))$ be a torsion-free affine connection, let $g$ be a complex metric, $\mathsf{S}$ an endomorphism of $\C TM$, and $\tau \in \Omega^1(\C TM)$ such that Equations \eqref{eq: GE generic}-\eqref{eq: Cod1 generic}-\eqref{eq: Cod2 generic}-\eqref{eq: Ricci generic} hold.

Then, there exists a $\pi_1(M)$-equivariant complex affine immersion $(\sigma, \xi)\colon \widetilde M\to \mathbb C^{m+1}$ with affine connection $\nabla$, affine fundamental form $g$, shape operator $\mathsf S$, and transversal connection form $\tau$.
Such affine immersions are unique up to composition with elements in $\mathrm{Aff}(\mathbb C^{m+1}, D)=\mathrm{GL}(m+1, \C)\ltimes \C^{m+1}$. 

{Moreover, if $\tau=0$ and $\nabla dV_g=0$, $(\sigma, \xi)$ can be chosen to be an equivariant Blaschke affine immersion for $\C^{m+1}$ endowed with the canonical volume form $\omega=dz_1\wedge \dots \wedge dz_{n+1}$, and its holonomy takes values in $\SL(m+1, \C)\ltimes \C^{m+1}$.}

\end{thm}
   Fo this statement, too, the computation follows as in the real case \cite[Theorem II.8.1]{NS}, up to a few adjustments. We recall the essential ideas and point out the main differences.
\begin{proof}
\emph{Step 1: Existence for $M=\widetilde M$.} Let $N\colon \C \times \widetilde M\to \widetilde M$ be the trivial bundle, and fix $\hat \xi\in \Gamma(N)$ a nowhere zero section. Consider the bundle $\mathrm{V}=\C T \widetilde M\oplus N$ and the connection $\hat D$ defined for all $\mathsf X\in \Gamma(T\widetilde M)$ and $Y\in \Gamma(\C T\widetilde M)$ as 
    \[
    \begin{cases}
        \hat D_{\mathsf X} Y &= \nabla_\mathsf X  Y+ g(\mathsf X, Y)\hat \xi\\
        \hat D_{\mathsf X} \hat \xi &=- \mathsf S(\mathsf X) +\tau(\mathsf X) \hat \xi\\
         \hat D_{\mathsf X} (i\hat \xi) &= -i\mathsf S(\mathsf X) +i\tau(\mathsf X) \hat \xi.
    \end{cases}
    \]
    Then, one gets the same equation as in \eqref{eq: flatness-identities} after replacing $D$ with $\hat D$, and $\mathrm R^{\hat D} (\mathsf X,Y)(i\xi)=i\mathrm R^{\hat D} (\mathsf X,Y)(\xi)$, hence the hypotheses of the theorem imply that $R^{\hat D}=0$. Moreover, $\hat D iY=i\hat D Y$, so there exists a bundle isomorphism 
    \[
  \Phi \colon  (\C T\widetilde M \oplus N, \hat D) \to (\C^{m+1}\times \widetilde M, D)
    \]
    with $D$ denoting the standard connection on the trivial bundle, and $\Phi$ being a $\C$-linear isomorphism fiberwise.

   For all $X\in \Gamma(\C T\widetilde M)$, we can see $\Phi(X_p)=:(p, \alpha(X_p))$ with $\alpha$ a $\C^{m+1}$-valued 1-form on $\widetilde M$. With a little abuse of notation, we will say that $\alpha=\Phi_{|\C T\widetilde M}$. Then, for all $\mathsf X, \mathsf Y\in \Gamma(T\widetilde M)$, $D_{\mathsf X}(\Phi(\mathsf Y)) = \alpha(\mathsf Y)$, and 
    \[d\alpha(\mathsf X, \mathsf Y)= \partial_\mathsf X(\alpha(\mathsf Y))-\partial_\mathsf Y(\alpha(\mathsf X))-\alpha([\mathsf X,\mathsf Y])=\Phi(\nabla_{\mathsf X} \mathsf Y-\nabla_{\mathsf Y}\mathsf X- [\mathsf X, \mathsf Y])=0
    \]
    as a result of $\nabla$ being torsion-free. Hence, there exists $\sigma \colon \widetilde M\to \C^{m+1}$ such that $d\sigma=\alpha$. As a consequence, $\sigma_*( X)=d\sigma( X)=  \Phi(X)$ so $\sigma_*(X)=0$ iff $X=0$, implying that $\sigma$ is an admissible immersion. 
    
    Moreover, if we consider the pull-back bundle $\sigma^*(\C^{m+1})$ with the pull-back connection $\sigma^*D$ and transversal vector field $\xi$ defined by $\sigma_*(\xi)= \Phi(\hat \xi)$, we get that, for all $ \mathsf X\in \Gamma(T\widetilde M)$ and $Y\in \Gamma(\C T\widetilde M)$, we have at each point that
    \[
    \begin{cases}
        \sigma_*\left((\sigma^*D)_{\mathsf X} Y\right)= D_\mathsf X(\sigma_*(Y))=D_\mathsf X(\Phi(Y))=\Phi(\hat D_\mathsf X Y)= \Phi(\nabla_\mathsf X Y+g(\mathsf X, Y)\hat \xi) =\sigma_*(\nabla_{\mathsf X} Y) +g(\mathsf X, Y) \sigma_*(\xi)\\
       \sigma_*\left( (\sigma^*D)_{\mathsf X}\xi\right)= D_X (\sigma_*\xi)= D_{\mathsf X}(\Phi(\hat \xi))= \Phi(\hat D_{\mathsf X} \hat \xi)= -\Phi(\mathsf S(\mathsf X)) +\tau(\mathsf X)\Phi(\hat \xi)= \sigma_*(-\mathsf S(\mathsf X) +\tau(\mathsf X)\xi),
    \end{cases}
    \]
    which implies that $(\sigma, \xi)$ has immersion data given by $\nabla, g, \mathsf S, \tau$.

    \vspace{8 pt}

    \emph{Step 2: Uniqueness and the equivariant case.}
    Let $(\sigma, \xi)$ and $(\sigma', \xi')$ be two affine immersions with the same immersion data $\nabla, g, \mathsf S, \tau$. We prove that there exists $A\in \mathrm{GL}(m+1, \C)\ltimes \C^{m+1}$ such that $\sigma'=A\circ\sigma$ and $\sigma'_*\xi'=A(\sigma_*\xi)$. 

Define the function $L\colon \widetilde M\to \mathrm{GL}(m+1, \C)$ as the unique matrix such that, for all $X\in \Gamma(\C T\widetilde M)$, $\sigma'_*(X)=L\circ\sigma_*(X)$ and $\sigma'_*(\xi')=L(\sigma_*(\xi))$. 

 Then (recalling that $\sigma_*((\sigma^*D)_\mathsf X Y)= \partial_\mathsf X (d\sigma(Y))= \partial_\mathsf X (\sigma_*(Y))$), we have that for all $\mathsf X\in \Gamma(T\widetilde M)$, $Y\in \Gamma(\C T\widetilde M)$,
 \begin{align*}
 (\partial_\mathsf X L) (\sigma_* Y)&= \partial_\mathsf X( (L\circ \sigma_*)(Y))- L(\partial_\mathsf X(\sigma_* Y))= \partial_\mathsf X(\sigma'_*(Y))- L(\partial_\mathsf X(\sigma_* Y))\\
 &=\sigma'_*(\nabla_\mathsf X Y)+g(\mathsf X,Y) \xi' - L(\sigma_*(\nabla_\mathsf X Y))-g(\mathsf X, Y) L(\xi)=0,\\
 (\partial_\mathsf X L) (\sigma_* \xi)&= \partial_{\mathsf X}(L(\sigma_*\xi)) -L(\partial_{\mathsf X} (\sigma_*\xi))= \partial_{\mathsf X}(L(\sigma_*\xi)) -L(\partial_{\mathsf X} (\sigma_*\xi)) -L(\partial_{\mathsf X}(\sigma_*\xi))\\
 &=-\sigma'_*(\mathsf S({\mathsf X}))+ \tau({\mathsf X})\sigma'_*(\xi)+L(\sigma_*(\mathsf S({\mathsf X})))+ \tau({\mathsf X})L(\sigma_*(\xi)) =0,
 \end{align*}
hence $dL=0$, implying that $L$ is a constant. Moreover, $d(\sigma'-L\sigma)=0$, hence there exists $b\in \C^{m+1}$ such that $\sigma'=L\cdot \sigma +b$, proving uniqueness. The equivariant case follows from uniqueness in the usual way.

\vspace{8 pt}

\emph{Step 3: Blaschke immersions.} Let $\omega=dz_1\wedge\dots \wedge dz_{m+1}$ be the standard $D$-flat volume form on $\C^{m+1}$. Assume $\tau=0$, and that $\nabla dV_g=0$ and construct a corresponding equivariant affine immersion $(\sigma, \xi)\colon \widetilde M\to \C^{m+1}$, and let $\theta=(\sigma^*\omega)(\cdot, \xi)$ be the pull-back form. Then, by Proposition \ref{prop: nablatheta}, $\nabla \theta=0$. Assuming $\nabla (dV_g)=0$ as well, we get that $dV_g= C \theta$ for some constant $C$ as a consequence of the fact that $\bigwedge^n \C T^*\widetilde M$ is rank-1. 
{Up to replacing now $(\sigma, \xi)$ with the affinely equivalent affine immersion $(\frac 1 C \sigma, \frac 1 C \xi)$, we can assume that $C=1$, therefore $(\sigma, \xi)$ is a Blaschke affine immersion in $(\C^{m+1}, D, \omega)$.}
If $(\sigma, \xi)$ is equivariant, its holonomy must have image in $\SL(m+1, \C)\ltimes\C^{m+1}$.

\end{proof}

\subsection{Complex affine spheres}
\label{section: complex affine spheres}
We now restrict to the case where $M$ is a closed oriented surface, and accordingly now write $M=S,$ and $\mathbb X= \mathbb C^3$ with the standard affine connection $D$ and volume form $\omega=dz_1\wedge dz_2\wedge dz_3$. 
We have that $\mathrm{Aff}(\mathbb C^3, D, \omega)=\SL(3,\mathbb C)\ltimes \mathbb C^3$.

\begin{defn}\label{def: CASdef}
Let $\rho\colon \pi_1(M)\to \mathrm{Aff}(\mathbb C^3, D, \omega)$, and let $(\sigma,\xi)$ be a $\rho$-equivariant Blaschke affine immersion.  We say that $(\sigma,\xi)$ is a (proper) \textbf{complex affine sphere} if the shape operator has the form $\mathsf S=\lambda \cdot \mathrm{id},$ where $\lambda\in \mathbb C^*$. A complex affine sphere $(\sigma, \xi)$ is \textbf{positive and hyperbolic} if $\lambda=-1$ and the Blaschke metric $g$ is a positive complex metric.
\end{defn}

\begin{remark}
    As shown in \cite[Theorem 6.5]{BEE}, every positive complex metric on an oriented surface $S$ admits a metric volume form compatible with the orientation. We will assume that the affine normal is compatible with the induced volume form. 
\end{remark}

\begin{remark}
    If an affine sphere $(\sigma, \xi)$ has shape operator in the form $\mathsf S=\lambda \cdot \mathrm{id}$ with $\lambda$ a function, by Codazzi Equation in Proposition \ref{prop: GCReqns} one gets that $d\lambda=0$, hence $\lambda$ is a constant. 
\end{remark}

\begin{remark}
\label{rmk: meeting point}
    If $(\sigma,\xi)$ is a proper complex affine sphere, then all of the affine normals meet at a single point. Indeed, one can see that the map $p\mapsto \sigma(p) +\frac 1 \lambda \xi(p)$ is a constant. 
\end{remark}

\begin{convention}
By Remark \ref{rmk: meeting point}, up to composing $\sigma$ with a translation, from now on {we will assume that the meeting point of the affine normals is} $\underline 0\in \C^3$, hence \begin{equation}
\label{eq: sigma multiple of xi}
\sigma=-\frac 1 \lambda \xi\ 
\end{equation}
and the data of $\sigma$ and $\lambda$ include the data of $\xi$. In this setting, if $\sigma$ is an equivariant affine sphere, then its holonomy fixes $\underline 0$, and hence takes values in $\SL(3,\mathbb C)$.
\end{convention}

Pick tensors play a prominent role in the real theory, and we turn to them now. Given $g,$ we have both the induced connection $\nabla$ and the Levi-Civita connection of $g$, say $\nabla^g$. The difference tensor on $\C TS$ is $$K(X,Y) = K_X(Y) = \nabla_XY - \nabla_X^g Y.$$ Since both connections have no torsion, $K$ is symmetric. The \textbf{Pick tensor} is the complex $3$-form defined by lowering indices, $$C(X,Y,Z) = -2g(K_X Y,Z).$$ 
The proposition below can be proved as in the real case for real vector fields (see \cite[Section II.4]{NS}), and by $\C$-linearity we get the general case.

\begin{prop}
\label{prop: formulas affine spheres}
Let $(\sigma, \xi)$ be an equivariant proper complex affine sphere. Then, for all $X, Y, Z, W \in \Gamma(\C TS),$ we have that 
\begin{enumerate}
\item $C=\nabla g$,
    \item $\mathrm{tr}_g K_X=0$,
\item $C,\nabla C,$ $\nabla K$, $\nabla^g C$, and $\nabla^g K$ are symmetric in all of their entries,
    \item $g((\nabla^g_X K)(Y,Z), W)=  g((\nabla^g_X K)(Y,W), Z)$,
    \item $\mathrm R^\nabla(X,Y)Z = \mathrm R^g(X,Y)Z +[K_X,K_Y](Z).$
\end{enumerate}
\end{prop}
We use only items (2) and (3), but the other results above help motivate some aspects of the proof of Theorem \ref{thm: integrating GE for CAS} below. Using that $\nabla^g C$ is symmetric, we show that positive hyperbolic affine spheres come with holomorphic cubic differentials, and thus we see the link to Theorem A.

\begin{prop}
\label{prop: Pick tensor is sum}
    Let $(\sigma, \xi)$ be an equivariant positive hyperbolic complex affine sphere in $\mathbb C^3$. Let $g$ be the Blaschke metric, with conformal classes given by $\ccpair\in \CSCS$. Then, 
    \[
    C=\Q_1 +\overline{\Q_2},
    \]
where $\Q_1$ is a $c_1$-holomorphic cubic differential and $\overline \Q_2$ is a $\overline{c_2}$-holomorphic cubic differential.
\end{prop}
\begin{proof}
Let $z$ and $\overline w$ denote local coordinates for $c_1$ and $\overline{c_2}$ respectively. 
    By Proposition \ref{prop: formulas affine spheres}, $C$ is symmetric, so it decomposes as
    \[
    C=pdz^3+ q dz^2 \cdot d\overline w + r dz \cdot d\overline w^2 + s d\overline{w}^3\ ,
    \]
    where $p,q,r,s$ are local $\mathbb C$-valued functions. By $(2)$ in Proposition \ref{prop: formulas affine spheres}, $tr_g(K_{\partial_{\overline z}})=0$, hence 
    \[
    0= tr_g( C(\partial_{\overline z}, \cdot, \cdot))= 2\frac{1}{g(\partial_{\overline z}, \partial_w)} C(\partial_{\overline z}, \partial_{\overline z}, \partial_w)= \frac{\partial_{\overline z}\overline w \cdot \partial_w z}{h(\partial_{\overline z}, \partial_w)} r \ ,
    \]
    hence $r=0$. Similarly, from $tr_g(K_{\partial_{w}})=0$ we deduce that $q=0$.
    
    We are left to prove that $p$ is $c_1$-holomorphic and $s$ is $\overline{c_2}$-holomorphic. First, by Lemma \ref{parallellem}, for all $X\in \Gamma (\C TS)$, $\nabla^g _X \partial_{\overline z}\in \mathrm{Span}(\partial_{\overline z})$ and $\nabla^g _X \partial_{w}\in \Gamma(\mathrm{Span}(\partial_{w}))$. 
As a result, 
        \[(\nabla^g dz)(\partial_{\overline z})= d( dz(\partial_{\overline z}) )- dz(\nabla^g \partial_{\overline z})=0-0=0.\ 
        \]
        Moreover, by Cartan's formula, for all $X,Y\in \Gamma(\C TS)$,
        \begin{align*}
        (\nabla^g_X dz) (Y) - (\nabla^g_Y dz) (X)&=\partial_X( dz(Y))-dz(\nabla^g_X Y) - \partial_Y(dz(X)) +dz(\nabla^g_Y X)= \\
        &=\partial_X(dz(Y))-\partial_Y(dz(X))-dz([X,Y])= (ddz)(X,Y)=0,
        \end{align*}
        hence $\nabla^g dz$ is commutative and $\nabla^g dz\in \Gamma(\mathrm{Span}(dz^2))$. 
        In the same fashion, $\nabla^g d\overline w\in \Gamma(\mathrm{Span}(d\overline w^2))$.
       Iterating the derivative, one can see that for all $m\in \mathbb Z_+$
       \[
       \nabla^g (dz^m) \in \Gamma(\mathrm{Span}(dz^{m+1}))\qquad \qquad \qquad   \nabla^g (d\overline w^m) \in \Gamma(\mathrm{Span}(d\overline w^{m+1})).\ 
       \]
 Now, by $(3)$ in Proposition \ref{prop: formulas affine spheres}, we have that $\nabla^g C$ is symmetric. Hence, from
        \[
        \nabla^g C= dp \otimes dz^3 + p\nabla^g (dz^3) + ds\otimes d\overline w ^3 + s\nabla^g (d\overline w^3) 
        \]
        we finally observe that
        \begin{align*}
            \partial_{\overline z} p \cdot (\partial_w z)^3= (\nabla^g_{\partial_{\overline z}}C)(\partial_{w}, \partial_w, \partial_w)&= (\nabla^g_{\partial_{w}}C)(\partial_{w}, \partial_w, \partial_{\overline z}) =0,  \\
             \partial_{w} s \cdot (\partial_{\overline z} \overline w)^3= (\nabla^g_{\partial_{w}}C)(\partial_{\overline z}, \partial_{\overline z}, \partial_{\overline z})&= (\nabla^g_{\partial_{\overline z}}C)(\partial_{\overline z}, \partial_{\overline z}, \partial_{w}) =0,  
        \end{align*}
        implying that $  \partial_{\overline z} p=\partial_{w} s\equiv 0 $.
\end{proof}
Conversely, the data of $g$ and $C$ completely characterize the complex affine immersion, as described by the following theorem.
\begin{thm}
    \label{thm: integrating GE for CAS}
    Let $g$ be a positive complex metric on a closed oriented surface $S$, with conformal class $\ccpair\in  \CSCS$, and let $\Q_1$ and $\overline \Q_2$ be holomorphic cubic differentials for $c_1$ and $\overline{c_2}$ respectively. 
    
    There exists an equivariant positive hyperbolic complex affine sphere $\sigma\colon \widetilde S\to (\C^3, D,\omega)$ with Blaschke metric $g$, Pick tensor $C=\Q_1+\overline{\Q_2}$, and with holonomy in $\SL(3,\C)$ if and only if $g$ and $C$ satisfy 
    \begin{equation} 
        \label{eq: Gauss CAS}
        \mathrm K_g=-1+ \frac{1}{4}g(\Q_1,\overline{\Q_2}).
    \end{equation}
Moreover, such a complex affine sphere is unique up to composition with elements in $\SL(3,\C)$.
\end{thm}

In local holomorphic coordinates $z$ and $\overline{w}$ for $c_1$ and $\overline{c_2}$ respectively, if $g=\lambda dz d\overline{w}$ and $\Q_1 =\phee dz^3,$ $\overline{\Q_2}=\overline{\psi}d\overline{w}^3,$ then $$\frac{1}{4}g(\Q_1,\overline{\Q_2})=2\frac{\phee\overline \psi}{\lambda^3}.$$ As well, take note that, in terms of the Pick tensor, $g(\Q_1,\overline{\Q_2})=\frac{1}{2}g(C,C),$ since $\Q_1$ and $\overline{\Q_2}$ are $g$-isotropic.
\begin{proof}
As above, let $z$ and $\overline{w}$ be coordinates for $c_1$ and $\overline{c_2}$ respectively, and $g=\lambda dzd\overline{w}$, $\Q_1 =\phee dz^3,$ and $\overline{\Q_2}=\overline{\psi}d\overline{w}^3.$

Let $\nabla$ be the affine connection on $\C TS$ defined by $C=-2g(\nabla-\nabla^g, \cdot)$. It suffices to show that $g$, $\nabla$, $\mathsf S =-\mathrm{id}$, and $\tau=0$ satisfy the hypothesis of Theorem \ref{thm: integration GC}, as well as $\nabla dV_g=0$, to produce an equivariant complex affine sphere. As in Remark \ref{rmk: meeting point}, by translating we can ensure that the holonomy lies in $\SL(3,\C)$.

Observe that Equations \eqref{eq: Cod2 generic} and \eqref{eq: Ricci generic} hold trivially. We prove that Equation \eqref{eq: Cod1 generic} holds. Since $\nabla^g g=0$, for all $X,Y,Z\in \Gamma(\C TS)$,
    \begin{align*}
        (\nabla_X g)(Y,Z)-(\nabla_Y g)(X,Z)&=(\nabla_X g)(Y,Z)-(\nabla_Y g)(X,Z) - (\nabla^g_X g)(Y,Z)-(\nabla^g _Y g)(X,Z)\\
        &=-  g(\nabla_X Y, Z)-   g(Y,\nabla_X Z)+  g(\nabla_Y X, Z)+  g(X,\nabla_Y Z) \\
        &=\frac 1 2 (C(X,Y,Z)+ C(X,Z,Y)-C(Y,X,Z)-C(Y,Z,X))=0,
    \end{align*}
   where in the last line we used the symmetry of $C$.
    
    We next prove that $\nabla dV_g=0$. Suggestively defining the symmetric tensor $K$ by $K_X Y= \nabla_X Y-\nabla^g_X Y$, we observe that for all $X,$ $$0=C(\cdot , \partial_{\overline z}, \partial_w)=-2g(K_{\cdot} \partial_{\overline z}, \partial_w),$$ and hence $K_{X} \partial_{\overline z}=K_{\partial_{\overline z}} X\in \Gamma(
\mathrm{Span}(\partial_w))$. Similarly, $K_{X} \partial_{w}= K_{\partial_w} X\in \Gamma(\mathrm{Span}(\partial_{\overline z}))$. One can compute the tensor $K$ explicitly:
\begin{equation}
\label{eq: K in a basis}
    \begin{split}
        &K_{\partial_w}\partial_w= -\frac 1 2\frac{(\partial_w z)^3}{g(\partial_{\overline z},\partial_ w) }\phee \partial_{\overline z},\\
        & K_{\partial_{\overline z}}\partial_{\overline z}=-\frac 12  \frac{(\partial_{\overline z} \overline w)^3}{g(\partial_{\overline z},\partial_w )}\overline \psi \partial_{w},\\
       &  
       K_{\partial_{\overline z}} \partial_w= K_{\partial_w}\partial_{\overline z}=0.
    \end{split}
\end{equation}

Observe that, for all 1-forms $\theta$, one has $K_X(\theta):= \nabla_X \theta- \nabla_X^g \theta=-\theta( K_X)$. Hence, $K_{\partial_w} dz=-dz(K_{\partial_w})=0$, implying that $K dz \in \mathrm{Span}(d\overline w^2)$ and $K d\overline w\in \mathrm{Span}(dz^2) $. 

As a result, writing $dV_g$ in the form $dV_g= \delta dz \wedge d\overline w$, we have that for all $X\in \Gamma(\C TS)$
\[
\nabla_X dV_g= \nabla_X dV_g- \nabla^g_X dV_g= \delta  K_X(dz)\wedge d\overline w+  \delta dz\wedge K_X(d\overline w)=0,\
\]
as desired.

Now we prove that Equation \eqref{eq: GE generic} holds.
First of all, observe that $\nabla^g C$ is symmetric in all of its entries. In fact, 
    \[
    \nabla^g C= \partial_z \phee dz^4 +\phee \nabla^g (dz^3) + \partial_{\overline w} \overline \psi d\overline w^4 +\overline \psi \nabla^g(d\overline w^3)
    \]
    and, as seen in the proof of Proposition \ref{prop: Pick tensor is sum}, $\nabla^g (dz^3)\in \Gamma(\mathrm{Span}(dz^4))$ and $\nabla^g (d\overline w^3)\in \Gamma(\mathrm{Span}(d\overline w^4))$. As a consequence, $\nabla^g K$ is symmetric in all of its entries because one can easily check that \[(\nabla^g_X C) (Y,Z,W)-(\nabla^g_Y C)(X,Z, W)= -2 g((\nabla^g_X K)(Y,Z),W)+ 2g((\nabla^g_YK)(X,Z), W)\ \]
for all $X,Y,Z\in \Gamma(\C TS).$ Now, from $\nabla =\nabla^g+K$, we have that \begin{align*}
\mathrm R^\nabla(X,Y)Z&=\mathrm R^g(X,Y)Z+ (\nabla^g_X K)(Y,Z)- (\nabla^g_Y K)(X,Z) +[K_X,K_Y]Z\\
&=\mathrm R^g(X,Y)Z+[K_X,K_Y]Z, \ 
\end{align*}
so the Gauss Equation \eqref{eq: GE generic} for $\mathsf S=-\mathrm{id}$ is equivalent to 
\[
\mathrm R^g(X,Y)Z= -[K_X, K_Y]Z-g(Y,Z)X+g(X,Z)Y,
\]
hence to
\begin{equation}
    \label{eq: Gauss'}
    g(\mathrm R^g(X,Y)Z,W)= -g([K_X, K_Y]Z,W)- g(Y,Z)g(X,W)+g(X,Z)g(Y,Z)
\end{equation}
for all $X,Y,Z,W\in\Gamma(\C TS)$.
It is now easy to check, using that $K$ is $g$-self-adjoint, that in Equation \eqref{eq: Gauss'} the right-hand side has the same symmetries as the left-hand side, so it is sufficient to check that the equality holds for $X=W=\partial_{\overline z}$ and $Y=Z=\partial_w$. Finally writing out the Gauss curvature as $$\mathrm K_g= \frac{g(\mathrm R^g(\partial_{\overline z},\partial_{ w})\partial_{ w},\partial_{\overline z})}{-(g(\partial_{\overline z}, \partial_w)^2)}= \frac{g(K_{\partial_{\overline{z}}}K_{\partial_w}\partial_w, \partial_{\overline z})-
g(K_{\partial_w}K_{\partial_{\overline{z}}}\partial_w, \partial_{\overline z})-(g(\partial_{\overline z}, \partial_w))^2 }{(g(\partial_{\overline z}, \partial_w))^2},$$ Equation \eqref{eq: K in a basis} returns that the Gauss equation \eqref{eq: GE generic} is equivalent to
$$\mathrm K_g= \frac {g(K_{\partial_w}\partial_w, K_{\partial_{\overline{z}}}\partial_{\overline z})} {(g(\partial_{\overline z}, \partial_w))^2} -1 =\frac 1 4 \phee \overline \psi \cdot \frac {(\partial_w z)^3(\partial_{\overline z}\overline w)^3}{(g(\partial_{\overline z}, \partial_w))^3}-1= 2\frac{\phee \overline \psi}{\lambda^3}-1.$$
\end{proof}
From now on, we will be interested only in Blaschke metrics $g$ of the form $g=e^{2u}h$, with $h$ being a Bers metric and $u:S\to \C$ a function. Any Blaschke metric that's connected to a Riemannian metric can be written in this way. 

\begin{example}
We are not providing a full construction, but there are in fact positive complex metrics $g$ that have constant curvature $-1$ and that are conformal multiples of Bers metrics by a conformal multiple that does not admit a logarithm; by \ref{thm: integrating GE for CAS}, $g$ and $C=0$ define a positive hyperbolic complex affine sphere. To construct such a metric, one can imagine starting from a Fuchsian type immersion $(f,\overline f)\colon \widetilde S\to \mathbb G$ and applying to $f$ a $2\pi$-grafting around a closed non-separating curve $\gamma$: precomposing with a suitable isotopy before grafting, we get a new developing map $\hat f\colon \widetilde S\to \CP^1$ with the same holonomy and such that $\hat f(p)\ne \overline f(p)$ for all $p\in \widetilde S$, hence $(\hat f, \overline f)\colon \widetilde S\to \mathbb G$ is an equivariant immersion. The induced complex metric $g$ now has curvature $-1$ by Theorem 6.7 in \cite{BEE}, and it is such that, assuming $\dot \gamma\ne 0$, $g(\dot \gamma, \dot \gamma)$ defines a homotopically nontrivial loop on $\mathbb C^*$, defining an example as above.
\end{example}

As a consequence of Theorem \ref{thm: integrating GE for CAS} and of the formula for the curvature of conformal metrics in Section \ref{subsection: complex metrics}, we have the following corollary.

\begin{cor}
\label{cor: PDE of Gauss eq}
    Let $h=\hpair$ be a Bers metric, with $\ccpair\in \CSCS$. Let $\Q_1$ and $\overline{\Q_2}$ be holomorphic cubic differentials for $c_1$ and $\overline{c_2}$ respectively. Then $g=e^{2u}\cdot h$, with $u\colon S\to \C$, and $C=\Q_1+\overline{\Q_2}$ are the Blaschke metric and the Pick tensor of an equivariant positive hyperbolic complex affine sphere in $\C^3$ if and only if 
\begin{equation}
    \label{eq: PDE of Gauss eq}
    G(c_1, \overline{c_2}, \Q_1, \overline{\Q_2},u):=\Delta_h u-e^{2u}+\frac 1 4 {h(\Q_1, \overline{\Q_2})}\cdot   e^{-4u}+1=0.
\end{equation}
\end{cor}

\subsection{Motivating examples}
Our first two examples show that complex affine spheres generalize both real affine spheres and the quasiconformal maps from the Bers theorem.
\begin{example}[Real affine spheres]\label{ex: ras}
    Affine differential geometry is a special case of the theory outlined here.

  By embedding $\R^3$ inside $\C^3$, any immersion $\sigma:\widetilde{S}\to \R^3$ can be viewed as an admissible immersion $\sigma: \widetilde{S}\to \C^3$. Moreover, if we can pick a section $\xi$ of $\sigma^* T\R^3$ that is transverse to $TS$, in which case $\sigma$ is an affine immersion in the traditional sense \cite[Definition II.1.1]{NS}, then $\sigma: \widetilde{S}\to \C^3$ is a complex affine immersion. If we demand that the image of $\sigma$ is locally strictly convex, then one can find a global real transverse normal. Choosing the transverse normal to be the complexification of a real transverse normal, all of the tensors that come into play are complexifications of real tensors. Moreover, the Blaschke metric is Riemannian, the Pick tensor is real, and the immersion data from Theorem \ref{thm: integrating GE for CAS} is always of the form $(c,\overline{c},\Q,\overline{\Q})$. Equation (\ref{eq: PDE of Gauss eq}) is a real semilinear elliptic PDE, $$\Delta_h u = e^{2u}-\frac{1}{4}h(\Q,\overline{\Q}) e^{-4u}-1,$$ which is the ordinary Tzitz{\'e}ica equation and is known to admit a unique real solution \cite[Proposition 4.0.2]{Lof}, \cite[Lemma 4.1.2]{Lab2}. If one takes that real solution and integrates it out to an affine sphere using Theorem \ref{thm: integrating GE for CAS}, then by uniqueness that affine sphere lives in an isometrically and totally geodesically embedded copy of $\R^3$. 
\end{example}

\begin{example}[Bers maps]\label{bersmaps}
Consider $\C^{2,1}$, namely $\C^3$ with the complex bilinear form $\inners$ defined by $\langle \underline z, \underline w \rangle= z_1w_1 + z_2 w_2 - z_3w_3$, and define $\mathbb{ \hat X}_2=\{\underline z\in \mathbb C^3 \ |\ \langle \underline z, \underline z \rangle =-1\}\  $
with the inherited holomorphic Riemannian metric. With reference to Section \ref{section: holo Riem mflds}, one has $\hat {\mathbb X}_2\cong \mathbb X _2 \cong \mathbb G$. 

We show that positive hyperbolic complex affine spheres with Pick tensor $C=0$ correspond to equivariant immersions into ${\mathbb {\hat X}}_2$ with pull-back metric being the Blaschke metric (which, by Theorem \ref{thm: integrating GE for CAS}, has curvature -1). Moreover, in the identification of $\PSL(2,\C)$ with $Isom_0(\mathbb {\hat X}_2)= Isom_0(\C^{2,1})= \mathrm{SO}(2,1,\C)$, if the Blaschke metric is a Bers metric, then the holonomy is quasi-Fuchsian, which shows that the construction is compatible with Bers' theorem.

As a consequence of Theorem 6.7 in \cite{BEE}, for all positive complex metric $g$ of curvature $-1$ on $S$, there exists a  $(\pi_1(S), \mathrm{SO}(3,\mathbb C))$-equivariant admissible immersion $\sigma\colon \widetilde S\to \mathbb {\hat X}_2$ with pull-back metric $g$. Moreover, such an immersion is totally geodesic, namely $(\sigma^*D^{\mathbb {\hat X}_2})_X Y= \nabla^g_X Y$ where $D^{\mathbb {\hat X}_2}$ denotes the Levi-Civita connection of $\mathbb {\hat X}_2$. On the other hand, $N_{\underline z}= i  \underline z$ is a normal vector field for $\mathbb {\hat X}_2$ inside $\mathbb C^{2,1}$ (namely, at each point it is the unique vector - up to a sign - with norm 1 orthogonal to $T\mathbb {\hat X}_2$). As a result, one can easily see that $(\sigma^*D)_{\mathsf X} N=  i\mathsf X$ and the second fundamental form of the immersion inside $\mathbb C^3$ is $-i g \otimes N= g\otimes \underline {\sigma(p)}$, where $\underline {\sigma(p)}$ denotes $\sigma(p)$ seen as a vector of $\mathbb C^3$.
Hence, denoting $\xi_{p}= \sigma^*(\underline {\sigma(p)})$, we have that  
\[
\begin{cases}  (\sigma^*D)_\mathsf{X} Y&= \nabla^g_{\mathsf X} Y + g(\mathsf X, Y)\xi   \\
(\sigma^*D)_{\mathsf X} \xi& = \mathsf X.
\end{cases}\
\]
Fix the canonical volume form $\omega$ on $\C^{2,1}$ given by the determinant. One can easily see that $\omega(\cdot, \cdot, \underline z)$ on $\mathbb {\hat X}_2$ is a volume, i.e. it is compatible with the holomorphic Riemannian metric. Up to composing $\sigma$ with the map $-\mathrm{id}$ on $\mathbb {\hat X}_2$ (which does not lie in $\mathrm{SO}(2,1,\C)$), we have that $dV_g= \sigma^*(\omega(\cdot, \cdot, z))$. We therefore conclude that $\sigma$ is a positive hyperbolic complex affine sphere. Since the affine connection coincides with $\nabla^g$, the Pick tensor is zero, and by uniqueness, we conclude that this is the unique positive hyperbolic complex affine sphere with Blaschke metric $g$ and Pick tensor $C=0$.
Finally, in the identification of $\PSL(2,\C)\cong Isom_0(\mathbb G)\cong Isom_0(\mathbb {\hat X}_2)\cong \mathrm {SO}(2,1,\C)$, one has trivially that the holonomy $\pi_1(S)\to \mathrm{SO}(2,1,\C)$ of the positive hyperbolic complex affine sphere with Blaschke metric being the Bers metric $\hpair$ and Pick tensor $C=0$ is the quasi-Fuchsian representation $\rho\ccpair$.
\end{example}

\begin{example}[Complex Tzitz{\'e}ica surfaces] 
The Tzitz{\'e}ica surface plays an important role in studying high energy harmonic maps to symmetric spaces \cite{DW}. We can explicitly construct a complex analogue. We won't refer to this example again, but we hope it conveys that one can make a lot of examples of positive hyperbolic complex affine spheres by hand.

 Let $z=z(x,y),w=w(x,y):\R^2\to\C$ be diffeomorphisms, and $\zeta=e^{2\pi \frac i 3}$ the third root of unity. For each $c\in\C^*,$ we consider the surface given by, $$f(x,y)=c(e^{z+\overline{w}}, e^{\zeta^2z+\zeta\overline{w}},e^{\zeta z+\zeta^2\overline{w}}).$$ 
Note that the image lies inside the complex surface $$\{z=(z_1,z_2,z_3)\in\C^3: z_1z_2z_3=c\}.$$ When $z=w$ and $c\in \R,$ this is the ordinary Tzitz{\'e}ica surface, describing a connected component of the surface in $\R^3$ defined implicitly by $xyz=c.$ For this reason, when $z\neq w$, we call our surfaces complex Tzitz{\'e}ica surfaces. We leave it to the reader to compute that $f$ defines a positive hyperbolic complex affine sphere, and that the Blaschke metric is $2dzd\overline{w}$ and the Pick tensor is $C=dz^3 + d\overline{w}^3$.
\end{example}

\section{Complex Lie derivatives and transport}

The space of complex structures on a surface is an infinite-dimensional Banach manifold. Still, its quotient up to isotopy, namely Teichm\"uller space, is finite-dimensional and the holonomy map factors through it. Generally speaking, taking the quotient up to isotopy is a powerful tool for constructing finite-dimensional configuration spaces while preserving their geometric significance. This is also the case for the moduli space of real hyperbolic affine spheres (discussed in Section \ref{sec: LL parametrization}).

Since $\mathcal{C}(S)\times \mathcal{C}(\overline{S})$ is the complexification of $\mathcal{C}(S),$ the natural group acting on $\mathcal{C}(S)\times \mathcal{C}(\overline{S})$ is the product $\mathrm{Diff}_0(S)\times \mathrm{Diff}_0(S)$, which returns a finite dimensional quotient, namely $\mathcal{T}(S)\times \mathcal{T}(\overline{S})$ (note that if we naively act by $\mathrm{Diff}_0(S)$ on the left or right or diagonally, we get an infinite dimensional quotient). The map $(c_1,\overline{c_2})\mapsto \vb*h_{(c_1,\overline{c_2})}$ pushes forward the action of $\mathrm{Diff}_0(S)\times \mathrm{Diff}_0(S)$ onto the space of Bers metrics, but the meaning of the induced quotient is not clear. For instance, if $\phi\in \mathrm{Diff}_0(S),$ it is not straightforward to see in what sense or why $-\frac 4 {(z-\overline w)^2}dz d\overline w$ and $-\frac 4 {(z-\phi^*(\overline w))^2}dz \phi^*(d\overline w)$ should be identified. 

In this section, we define the notions of complex Lie derivative and of transport of tensors through paths of complex vector fields. As we will see, these tools allow us to study the $\mathrm{Diff}_0(S)\times \mathrm{Diff}_0(S)$ action and will motivate the construction of the finite-dimensional space $\Sph(S)$ in Section \ref{sec: CAS space}. Moreover, this approach will be a key tool in proving Theorem D.

\subsection{Complex Lie derivatives}

Let $M$ be a smooth manifold. Denote by $$\mathrm{Tens}^{p,q}(M)=\Gamma\left( (TM)^{p_ \otimes} \otimes (T^*M)^{q_\otimes}\right)$$ the space of $(p,q)$-type tensors on $M$. The Lie derivative along a vector field can be seen as a bilinear map
\begin{equation*}
\begin{split}
    \Gamma(TM)\times \mathrm{Tens}^{p,q}(M)&\to \mathrm{Tens}^{p,q}(M),\\
    (X,\alpha)&\mapsto \mathscr L_X \alpha. 
\end{split}\ 
\end{equation*}
We extend it $\C$-bilinearly to a map
\[
\Gamma(\mathbb C TM)\times (\mathrm{Tens}^{p,q} (M)\otimes \C) \to \mathrm{Tens}^{p,q} (M)\otimes \C. \ 
\]
That is, we extend by the rule $$\mathscr L_Z \alpha= \mathscr L_{Re Z} (Re (\alpha)) -\mathscr L_{Im Z}(Im(\alpha))+ i\mathscr L_{Re Z}(Im (\alpha)) +i\mathscr L_{Im Z}(Re (\alpha)).$$
We call $\mathscr L_Z \alpha$ the \textbf{complex Lie derivative of $\alpha$ along the complex vector field $Z$.} $\mathscr L_Z$ is characterized by the following algebraic properties:
\begin{enumerate}
    \item for all $C^1$ function $f\colon S\to \mathbb C$, $\mathscr L_Z(f)=  \partial_Z f,$ 
    \item for all complex vector fields $W\in \Gamma(\mathbb C TM)$, 
    $
    \mathscr L_Z W= [Z,W], 
    $
    \item for all complex differential forms $\omega\in \Omega^r(M)\otimes \C$,
$
\mathscr L_Z \omega= d\omega(Z, \cdot)+ d( \omega(Z, \cdot)),$ 
    \item for all $\alpha,\beta$ smooth complex tensors,
     $\mathscr L_Z (\alpha\otimes \beta)= (\mathscr L_Z \alpha)\otimes \beta+ \alpha \otimes (\mathscr L_Z \beta),$
     \item one has the recursive formula for every $\alpha\in \mathrm{Tens}^{p,q} (M)\otimes \mathbb C$ given by
     \begin{align*}
     (\mathscr L_Z \alpha)(\omega_1, \dots, \omega_p, W_1,\dots, W_q)&= \partial_Z(\alpha(\omega_1, \dots, \omega_p, W_1,\dots, W_q)) - \alpha ( \mathscr L_Z \omega_1, \dots, \omega_p, W_1,\dots, W_q))- \dots \\
     &-\alpha (  \omega_1, \dots, \mathscr L_Z\omega_p, W_1,\dots, W_q))-\alpha (  \omega_1, \dots, \mathcal \omega_p, \mathscr L_ZW_1,\dots, W_q))-\dots \\
     &-\alpha (  \omega_1, \dots, \mathcal \omega_p, W_1,\dots, \mathscr L_ZW_q)), 
     \end{align*}
     where $\omega_1, \dots,\omega_p\in \Omega^1(\mathbb C TM)$, and $W_1, \dots, W_q\in \Gamma(\mathbb C TM)$.
\end{enumerate}

\begin{remark}
\label{rmk: complex derivative of a complex structure}
We show a way to define the \textbf{complex Lie derivative of a complex structure} $c_1\in \mathcal C(S)$ for which $\mathscr L_Z {c_1}\in T_{c_1}\mathcal C(S)$ lies in the kernel of the projection $\mathcal C(S)\to \mathcal T(S)$.

We define it as follows. Let $Z\in \Gamma(\C TS)$ and let $z$ be any local holomorphic coordinate for $c_1$, and decompose $Z$ with respect to the basis $\{\partial_z, \partial_{\overline z}\}$, hence $Z=pr_{\partial_z} (Z) + pr_{\partial_{\overline z}} (Z)$. We define 
\[
\mathscr L_{Z} c_1:= \mathscr L_{pr_{\partial_z} (Z)+\overline{pr_{\partial_z} (Z)}} c_1
\]
where the RHS is a standard real Lie derivative, corresponding to the infinitesimal deformation through the isotopy $\phi_t$ with $\dot \phi_0= pr_{\partial_z} (Z)+\overline{pr_{\partial_z} (Z)}$.

A motivation for this choice is that 
\[
\mathscr L_{Z} z= \mathscr L_{pr_{\partial_z}(Z)} z= \mathscr L_{pr_{\partial_z} (Z)+\overline{pr_{\partial_z} (Z)}} z  = \frac d {dt}\big|_{t=0}(z\circ \phi_t),
\]
and also  $\mathscr L_Z dz= \mathscr L_{\dot \phi_0} dz$. 
Moreover, by the properties of the complex Lie derivative, one has that for all $c_1$-holomorphic $n$-th differential $\alpha$, $\mathscr L_Z \alpha= \mathscr L_{\dot \phi_0} \alpha$.

Using the same approach on $\overline S$, one defines the complex Lie derivative of any element in $\mathcal C(\overline S)$.
\end{remark}

\begin{example}
\label{ex: Lie derivative of Bers metrics}
Here is a key example to motivate the complex Lie derivative. Let $S$ be a closed surface, let $\phi_t:[0,1]\times S\to S$ be an isotopy, and let $\ccpair\in \CSCS$ with Bers metric $\vb*h_{(c_1, \overline{c_2})}$. Let $z$ and $\overline{w}$ be local coordinates for $c_1$ and $\overline{c_2}$ respectively, and consider the path of Bers metrics given locally by 
\[
h_t= \vb*h_{(c_1, \phi_t^*(\overline{c_2}))}= -\frac{4}{(z-\phi_t^*(\overline w))^2}dz\cdot  (\phi_t^*(d\overline{w})).\ 
\]
The variation $\dot h_t=\frac{d}{d t}\big|_{0} h_t$ defines a bilinear form on $\C TS$. We claim that this variation can be seen as a complex Lie derivative of $h_t$. Indeed, let $\overline {w_t}:=\phi_t^*(\overline w)$ be a local coordinate for $\phi_t^*(\overline{c_2})$, denote $\dot \phi_t=: \gamma_t \partial_z+\overline{\gamma_t}\partial_{\overline z}$, and consider the complex vector field 
\[
Z_t:= \left(\gamma_t \frac{\partial_z \overline w_t}{\partial_{\overline z}\overline w_t}+\overline{\gamma_t}\right)\partial_{\overline z}. \ 
\]
One can easily check that this description of $Z_t$ does not depend on the choice of the local holomorphic coordinates $z$ and $\overline w_t$, so $Z_t\in \Gamma(\mathbb CTS)$. 

By construction, one has that: \begin{align*}
&\mathscr L_{Z_t} c_1=0,\\
&\mathscr L_{Z_t} \phi_t^*(\overline {c_2})= \gamma_t \partial_z \overline{w_t} +\overline{\gamma_t} \partial_{\overline z} \overline{c_2}= \mathscr L_{\dot{\phi_t}}\phi_t^*( \overline {c_2} ).
\end{align*}
In particular, $\mathscr L_{Z_t} z=0$, $\mathscr L_{Z_t} dz=0$, $\mathscr L_{Z_t} \overline w_t=\mathscr L_{\dot \phi_t}w_t$ and $\mathscr L_{Z_t} d\overline {w_t}=\mathscr L_{\dot \phi_t}d\overline{w_t}$.
In conclusion, by applying the properties of the complex Lie derivative, one gets that 
\[
\dot h_t=\mathscr L_{Z_t} h_t.\ 
\]
\end{example}

Recall the quotient map $\mathcal{C}(S)\to \mathcal{T}(\overline{S})$, which we label as $\Psi.$
As Example \ref{ex: Lie derivative of Bers metrics} suggests, Lie derivatives through complex vector fields allow to deform simultaneously pairs of complex structures. This is made clearer by the following Proposition which can be seen as a complex extension of the natural description for $T\mathcal C(S)$ through the bundle $\Psi$.

\begin{prop}
    \label{prop: tangent to C(S)C(S)}
    Let $\ccpair\in \CSCS$ and consider the map $$d_{c_1}\Psi\oplus d_{\overline{c_2}}\overline{\Psi}: T_{c_1}\mathcal{C}(S)\times T_{\overline{c_2}}\mathcal{C}(\overline{S})\to T_{[c_1]}\mathcal{T}(S)\times T_{[\overline{c_2}]}\mathcal{T}(\overline{S}).$$
    The map
    \begin{align*}
        \Gamma(\C TS)&\to T_{c_1}\mathcal C(S)\times T_{\overline {c_2}}\mathcal C(\overline S)\\
        Z &\mapsto (\mathscr L_Z c_1, \mathscr L_Z \overline{c_2})
    \end{align*}
    defines a $\C$-linear isomorphism onto the kernel of $d\Psi_{c_1}\oplus d\overline{\Psi}_{\overline{c_2}}$ .
\end{prop}
    As a result, one can see
    \[
    T_{\ccpair} \mathcal C(S)\times \mathcal C(\overline S)\cong T_{[c_1]}\mathcal T(S)\times T_{[\overline{c_2}]} \mathcal T(\overline S)\times \Gamma(\C TS).
    \]
\begin{proof}
Use the notation of Remark \ref{rmk: complex derivative of a complex structure} and Example \ref{ex: Lie derivative of Bers metrics}.
    The map is injective: since $\mathrm{Span}(\partial_{\overline z})\cap \mathrm{Span}(\partial_w)=\{0\}$, $pr_{\partial_z} Z=0=pr_{\partial_{\overline w}}Z$ implies $Z=0$.

By Remark \ref{rmk: complex derivative of a complex structure}, we see that the image of the map is in fact contained in the kernel.

    Let $\dot \phi_1, \dot \phi_2\in T_{\mathrm{id}} \mathrm{Diff}(S)$, hence $\mathscr L_{\dot \phi_1} c_1\in T_{c_1} \mathcal C(S)$ and $\mathscr L_{\dot \phi_2} \overline {c_2}\in T_{\overline{c_2}} \mathcal C(\overline S)$. As in Example \ref{ex: Lie derivative of Bers metrics}, one can find $Z_1, Z_2\in \Gamma(\C TS)$ such that 
    \begin{align*}
        \mathscr L_{Z_1} c_1= \mathscr L_{\dot \phi_1} c_1 \qquad &\text{and} \qquad   \mathscr L_{Z_1} \overline{c_2}=0,\ \\
        \mathscr L_{Z_1} c_1= 0\qquad &\text{and} \qquad   \mathscr L_{Z_1} \overline{c_2}=\mathscr L_{\dot \phi_2} \overline{c_2} ,
    \end{align*}
    hence \[
    (\mathscr L_{Z_1+Z_2} c_1, \mathscr L_{Z_1+Z_2} \overline{c_2})= (\mathscr L_{\dot \phi_1} c_1, \mathscr L_{\dot \phi_2} \overline{c_2}),
    \]
    proving surjectivity.

    We are left to prove that the map is $\C$-linear, namely that $\mathscr L_{iZ}c_1=i\mathscr L_{Z}c_1$ and that $\mathscr L_{iZ} \overline{c_2}=-i\mathscr L_Z \overline{c_2}$. To see this, let $\phi_t$ be an isotopy, and let $Z\in \Gamma(\C TS)$ be any vector field such that $\mathscr L_{Z}c_1= \frac{d}{dt}|_0 \phi^*_t(c_1)$. Denoting $z_t:=\phi_t^*z$ and seeing $\mathcal C(S)$ as the space of $C^{\infty}$ Beltrami forms for $c_1$, we have that 
    \[
    \mathscr L_Z{c_1} =\frac d {dt}|_0 \left( \frac{dz_t(\partial_{\overline z})}{d{z_t}(\partial_{ z})} \frac{d\overline z}{dz}\right)= (d\dot z)(\partial_{\overline z}) \frac{d\overline z}{dz}= (\mathscr L_Z dz)(\partial_{\overline z}).
    \]
    Since $\mathscr L_{iZ} dz=i\mathscr L_Z dz$ we conclude that $ \mathscr L_{iZ}{c_1}= i \mathscr L_Z{c_1}$. Finally, using that $\mathscr L_Z \overline{c_2}=\overline{\mathscr L_{\overline{Z}} (c_2)}$, we get that $\mathscr L_{iZ}{\overline{c_2}}=-i\mathscr L_Z(\overline{c_2})$.
\end{proof}

\begin{cor}
\label{cor: lie derivative Bers}
    Let $h_t=\vb*h_{(c_1^t, \overline{c_2}^t)}$ be a $C^1$ path of Bers metrics with constant holonomy, then $\dot h_t=\mathscr L_{Z_t} h_t$, where $Z_t$ is such that $(\dot c_1^t, \dot{\overline{c_2}}^t)=(\mathscr L_{Z_t} c_1^t, \mathscr L_{Z_t} \overline{c_2}^t)$.
\end{cor}
\begin{proof}
   Applying the same method as in example \ref{ex: Lie derivative of Bers metrics}, denoting by $z_t$ and $\overline{ w_t}$ complex coordinates of $c_1^t$ and $\overline{c_2}^t$, we get
    \[
    \frac d{dt} h_t=  \frac{8(\partial_{Z_t} z_t- \partial_{Z_t} \overline{w_t})}{(z_t-\overline w_t)^3} dz_t d\overline{w_t} - \frac 4{(z_t-\overline w_t)^2} ( (\mathscr L_{Z_t}dz)\cdot  d\overline w_t + dz \cdot ( \mathscr L_{Z_t} d\overline w_t) )= \mathscr L_{Z_t} h_t.\ 
    \]
\end{proof}

\subsection{Transport through complex vector fields}\label{subsec: trans}

Unlike real Lie derivatives, realizing complex Lie derivatives as infinitesimal deformations is difficult. More precisely, given a path of complex vector fields $Z_t\in \Gamma(\C TM)$, $t\in [0, T]$ and a tensor $\alpha$, we are interested in solutions $\alpha_t$ to the system
\begin{equation}
    \label{eq: transport problem}
    \begin{cases}
        \mathscr L_{Z_t} \alpha_t = \dot \alpha_t\\
        \alpha_0=\alpha,
    \end{cases}
\end{equation}
and understanding the existence and uniqueness for this Cauchy problem \eqref{eq: transport problem} is not in general straightforward. A simple case is as follows: on the unit disk, for $g(z,t)\partial_{\overline{z}}$ a time-dependent vector field and $f$ a given function, find $f_t$ such that
\begin{equation}
    \label{eq: transport problem2}
    \begin{cases}
        g(z,t)\partial_{\overline{z}} f_t(z) = \frac{d}{dt}f_t(z)\\
       f_0(z)=f(z).
    \end{cases}
\end{equation}
As we will see below, by the Cauchy-Kovalevskaya and Holmgren theorems, when $g$ and $f$ are real analytic, we can find a unique short time solution. But in fact, as can be verified through H{\"o}rmander's classical criterion, the problem (\ref{eq: transport problem2}) is not locally solvable, which means one can find $C_0^\infty$ input data for which no short time solution even exists. For more on local solvability, see \cite{Fol} and \cite{Treves}.

\begin{defn}
\label{def: analytic transport}
    We will call a solution to $\eqref{eq: transport problem}$ a \textbf{transport of $\alpha$ through $Z_{\bullet}$} and we denote it by $\alpha_t =: \mathscr T_{Z_t} \alpha$.
\end{defn}
We will also use the notation $\mathscr T_{Z_t} c$ to denote paths of complex structures such that $\frac d {dt} \mathscr T_{Z_t} c= \mathscr L_{Z_t} (\mathscr T_{Z_t} c)$, in the sense of Remark \ref{rmk: complex derivative of a complex structure}.

For the rest of this subsection, we assume that one can solve the problem \eqref{eq: transport problem}, and in the next subsection we discuss existence in the real analytic case.

\begin{example}
  In fact, the theory of Bers metrics suggests that we shouldn't hope for long time existence of the transport either. Using the unit disk as a toy model, consider a map $(f_1,\overline{f_2}):\mathbb{D}\to \mathbb{CP}^1\times \mathbb{CP}^1\backslash \Delta$ for which there exists points $p,q\in\mathbb{D}$ such that $f_1(p)=\overline{f_2}(q).$ Then let $\phi_t$ be any isotopy of the disk taking $p$ to $q$ at a time $t_0.$ For $t<t_0$, we have a family of positive complex metrics given by $$-\frac{4}{(f_1-\phi_t^*\overline{f_2})^2}df_1\cdot (\phi_t^*d(\overline{f_2})),$$ which is tangent to a complex Lie derivative at time zero (see Example \ref{ex: Lie derivative of Bers metrics}), but which also blows up at $p$ as $t$ increases to $t_0.$
\end{example}

\begin{remark}
\label{rmk: properties of veet}
    By the definition of complex Lie derivative, we can deduce some natural properties of the transport.
    \begin{enumerate}[leftmargin=*]
        \item If $\omega\in\Omega^p(S)\otimes \C$ is transportable through $Z_{\bullet}$, then
        \[
        \veet (d\omega) = d (\veet \omega ).
        \]
        \item For all complex tensors $\alpha, \beta$ which can be transported through $Z_\bullet$ one has 
        \[
        \mathscr T_{Z_t} (\alpha \otimes \beta)= (\mathscr T_{Z_t} \alpha) \otimes (\mathscr T_{Z_t} \beta).
        \]
        \item If $\alpha$ in $\mathrm{Tens}^{p,q}(S)\otimes \mathbb C$, $W_1,\dots W_q\in \Gamma(\mathbb C TS)$, and $\omega_1, \dots \omega_p\in \Omega^1(\mathbb C TS)$ can all be transported through $Z_\bullet$, then
        \[
        \veet \left( \alpha(\omega_1, \dots, \omega_p, W_1, \dots, W_q )\right) = (\veet \alpha)( \veet(\omega_1), \dots, \veet(\omega_p), \veet(W_1), \dots, \veet (W_q)).
        \]
    \end{enumerate}
\end{remark}

\begin{remark}
\label{rmk: transport Bers metrics}
Let $\ccpair\in \CSCS$, let $\phi^1_t,\phi^2_t$ be isotopies. Let $Z_t\in \Gamma(\C TS)$ be such that $\frac d{dt}((\phi^1_t) ^*c_1, (\phi^2_t)^* \overline{c_2})= (\mathscr L_{Z_t} c_1^t, \mathscr L_{Z_t}\overline{c_2}^t)$ as in Proposition \ref{prop: tangent to C(S)C(S)}. Then, for all $c_1$-holomorphic (resp. $\overline{c_2}$-holomorphic) tensor $\alpha$, one has $(\phi_1^t)^*\alpha= \veet \alpha$ (resp. $(\phi_2^t)^*\alpha= \veet \alpha $). 
Moreover, by Corollary \ref{cor: lie derivative Bers}, $\vb*h_{((\phi_1^t)^*{c_1}, (\phi_2^t)^*\overline{c_2})}= \veet \vb*h_{(c_1, \overline{c_2})}$.
\end{remark}

Here are a few more remarks about the transport through complex vector fields. Observe that the hypothesis of the following proposition holds for $Z_t\in \Gamma(\C TS)$ corresponding to $\frac d {dt} ((\phi_1^t)^* c_1, (\phi_2^t)^*{\overline{c_2}})$ and for $\veeet (\hpair)= \vb*h_{((\phi_1^t)^* c_1, (\phi_2^t)^*{\overline{c_2}})}$.

\begin{prop}
\label{prop: transport connection and laplacian}
Denote by $\veeet= \veet$ the transport of tensors through the path of vector fields $t\mapsto Z_t\in \Gamma(\C TS)$, $t\in [0,T]$. Assume that $g$ is a positive complex metric on $S$ for which $\veeet g$ exists and is a positive complex metric, and that $(\vb*c_+, \vb* c_-)(\veeet g)= (\veeet (\vb*c_+(g)), \veeet( \vb* c_-(g)))$.

\begin{enumerate}[leftmargin=*]
\item Up to reducing $T$, there exists a local orthonormal frame $(e_1, e_2)$ for $g$ which admits a transport through $Z_\bullet$ and such that $(\veeet e_1, \veeet e_2)$ is orthonormal for $\veeet g$.
\item If $\alpha\in \mathrm{Tens}^{p,q}(S)\otimes \mathbb C$ can be transported through $Z_\bullet$,
\[
\veeet (g(\alpha, \alpha))= (\veeet g)(\veeet \alpha,\veeet \alpha). 
\]
    \item If $X,Y\in \Gamma(\mathbb C TS)$ can be transported through $Z_\bullet$,  
    \[\veeet (\nabla^g_X Y)=\nabla^{\veeet g}_{\veeet X} (\veeet Y), \]
    where $\nabla^{\veeet(g)}$ denotes the Levi-Civita connection of $\veeet (g)$.
    \item The curvature changes according to $\veeet( \mathrm R^g)= \mathrm R^{\veeet(g)}$ and $\veeet (\mathrm K_g)= \mathrm K_{\veeet g}$.
    \item If $\alpha$ can be transported through $Z_\bullet$, 
    \[\veeet (\nabla^g \alpha)= \nabla^{\veeet g} (\veeet \alpha).\  \]
    \item If $f\in C^r(S, \mathbb C)$ can be transported through $Z_\bullet$, higher-order derivatives change according to
    \[
    \veeet \left((\nabla^g)^r f\right)= \left(\nabla^{\veeet g} \right)^r (\veeet f).\ 
    \]
    In particular, the Hessian and the Laplacian change according to 
    \[\veeet(\mathrm{Hess}_g(f))= \mathrm{Hess}_{\veeet (g)} (\veeet f),\]\[
    \veeet (\Delta_g (f))=\Delta_{\veeet (g)} (\veeet f).\ 
    \]

\end{enumerate}
\end{prop}
\begin{proof}
In the proof, we denote $g_t= \veeet g$ and $\ccpair=(\vb* c_+, \vb* c_-)(g)$, with $z$ and $\overline w$ local holomorphic coordinates for $c_1$ and $\overline{c_2}$ respectively. Observe that since $\veeet c_1$ and $\veeet \overline{c_2}$ correspond, by Proposition \ref{prop: tangent to C(S)C(S)}, to deformation by isotopies, up to shrinking the charts and reducing $T$, we can take $z_t=\veeet z$ and $\overline{w_t}= \veeet \overline w$ local holomorphic coordinates for $\veeet c_1$ and $\veeet \overline{c_2}$. 

\begin{enumerate}[leftmargin=*]
\item We first determine a transportable orthonormal coframe. Use $g_t$ to denote also the dual bilinear forms on $\C T^*S$. By hypothesis, $ g_t ( dz_t,  dz_t)= g_t( d\overline w_t,  d \overline w_t)=0$. It is therefore easy to check that
\[
\theta^1_t= \frac{dz_t+ d\overline w_t}{\sqrt{2g_t(dz_t, d\overline w_t)}}\qquad \quad \theta^2_t= i\frac{dz_t- d\overline w_t}{\sqrt{2 g_t (dz_t, d\overline w_t)}}
\]
determine an orthonormal coframe for $g_t$, and satisfy $\frac d {d} \theta^j_t= \mathscr L_{Z_t} \theta^j_t$, hence $\theta^j_t= \veeet \theta^j_0$. (Here the square roots exist up to shrinking to a contractible neighborhood.)

Denote with $(e_1^t, e_2^t)$ the dual frame to $(\theta^1_t, \theta^2_t)$, which is therefore $g_t$-orthonormal. Then, applying $\frac d {dt}- \mathscr L_{Z_t}$ to the equations $\theta^j_t(e_k^t)=\delta_k^j$, we get that $\frac d {dt} e_k^t- \mathscr L_{Z_t} e_k^t$ lies in the kernel of both $\theta^1_t$ and $\theta^2_t$, therefore it is zero and $e_k^t= \veeet e_k^0$.

\item Let $(e_1, e_2)$ be a $g$-orthonormal frame as in item 1. Denote the dual frame $(\theta^1, \theta^2)$, which is therefore such that there exist transports for which $(\veeet \theta^1_t, \veeet \theta^2_t)$ is an orthonormal frame for $g_t$. Then, 
\begin{align*}
&\veeet \left( g\left( \theta^{m_1}\otimes \theta^{m_p} \otimes e_{n_1}\otimes e_{n_q},   \theta^{m'_1}\otimes \theta^{m'_p}\otimes  e_{n'_1}\otimes e_{n'_q}\right) \right)\\
&=  \delta_{m_1, m'_1}\cdot \dots \cdot  \delta_{m_p, m'_p} \cdot \delta_{n_1, n'_1}\cdot  \dots \cdot \delta_{n_q, n'_q}\\
&=(\veeet g)\left(\veeet \big(\theta^{m_1}\otimes\dots \otimes \theta^{m_p} \otimes e_{n_1}\otimes \dots \otimes e_{n_q}\big), \veeet \left(  \theta^{m'_1}\otimes \dots \otimes\theta^{m'_p}\otimes  e_{n'_1}\otimes \dots \otimes e_{n'_q} \right)\right).
\end{align*}
To conclude the proof, observe that applying $\frac d {dt}- \mathscr L_{Z_t}$ to the basis decomposition $\veeet\alpha=\sum_{I,J} a^t_{I,J} \veeet(\theta^I) \otimes \veeet(e_J)$, with $I,J$ multi-indices, one gets that $a^t_{I,J}= \veeet a_{I,J}^0$, hence the statement follows by linearity.
\item 
Use the notation as in the previous steps. As shown in \cite{BEE}, for complex metrics one has a natural notion of Levi-Civita connection forms that extend the Riemannian notion. So the metric $g_t$ and the orthonormal frame $(\veeet e_1, \veeet e_2)$ coframe $(\veeet\theta^1, \veeet \theta^2)$ induce the forms $(\omega_i^j)_t$ defined by
\[\nabla^{g_t} \veeet(e_i)= \sum _{j} (\omega_i^j)_t\otimes \veeet(e_j),
\] and which
are characterized as the unique 1-forms such that 
\[
\begin{cases}
\veeet(d\theta^i) =-\sum_j (\omega_j^i)_t \wedge \veeet( \theta^j)\\
(\omega_j^i)_t=-(\omega_i^j)_t .
\end{cases}
\]
Applying $\frac d {dt}-\mathscr L_{Z_t}$ to the equation on top, one gets by Remark \ref{rmk: properties of veet} that \[\left(\frac d {dt}-\mathscr L_{Z_t}\right)(\omega^i_j)_t\wedge (\veeet \theta^j)=0=-\left(\frac d {dt}-\mathscr L_{Z_t}\right)(\omega_j^i)_t \wedge (\veeet \theta^i),\] hence $(\frac d {dt}-\mathscr L_{Z_t})(\omega_j^i)_t=0$, thus $\veeet(\nabla^g { e_j})=\nabla^{\veeet g} (\veeet e_j)$ for all $j$. Finally if  $X\in \Gamma(\mathbb C T S)$ admits a transport through $Z_{\bullet}$, and $X=\sum_j a_j e_j$, we obtain once again that $\veeet X= \sum_j \veeet a_j \veeet e_j$ and
\begin{align*}
\veeet( \nabla^g X)&= \veeet\left(\sum_j d a_j \otimes e_j + \sum_j a_j \otimes \nabla^g e_j\right) \\
&=\sum_j d(\veeet{a_j}) \otimes (\veeet {e_j}) + \sum_j \veeet(a_j) \otimes \nabla^{\veeet (g)} (\veeet e_j)= \nabla^{\veeet(g)}\veeet(X).\ 
\end{align*}

which concludes the proof.

\item The proof follows easily by applying iteratively the previous step to the tensor \[
\mathrm R^g(X,Y) Z= \nabla^g_X\nabla^g_Y Z- \nabla^g_Y\nabla^g_X Z- \nabla^g_{\nabla^g_X Y-\nabla^g_Y X} Z. \ 
\]

With the notation as in the previous steps, we also see that
\[
\veeet(\mathrm K_g)= \veeet ( g(\mathrm R_g ( e_1, e_2) e_2, e_1) )=  (\veeet g)(\mathrm R_{\veeet g} (\veeet e_1, \veeet e_2)\veeet e_2, \veeet e_1)= \mathrm K_{(\veeet g)}.
\]

\item By the first item and Remark \ref{rmk: properties of veet}, the statement is true if $\alpha$ is a 1-form since
\[
(\veeet (\nabla^g \alpha) )(\veeet e_j)= \veeet ((\nabla^g \alpha) e_j)= \veeet ( d( \alpha(e_j)) )- \veeet(\alpha(\nabla^g e_j))= (\nabla^{\veeet(g)} \veeet(\alpha) ) (\veeet e_j).  \ 
\]
The proof then follows for all $(p,q)$-type real analytic tensors $\alpha$ using the formula
   \begin{align*}
      (\nabla^g \alpha)(\omega_1, \dots, \omega_p, Y_1,\dots, Y_q)&= d(\alpha(\omega_1, \dots, \omega_p, W_1,\dots, W_q)) - \alpha ( \nabla^g \omega_1, \dots, \omega_p, W_1,\dots, W_q))- \dots -\\
     &-\alpha (  \omega_1, \dots, \nabla^g \omega_p, W_1,\dots, W_q))-\alpha (  \omega_1, \dots, \omega_p, \nabla^g W_1,\dots, W_q))-\dots -\\
     &-\alpha (  \omega_1, \dots, \mathcal \omega_p, W_1,\dots, \mathcal \nabla^g W_q)) 
\end{align*}
for $\omega_k\in \{\theta^1, \theta^2\}$ and $W_k\in \{e_1, e_2\}$.

\item The statement is true for $r=1$:
\[
g_t\left(\veeet\left( \nabla^g f \right), \veeet e_k\right)=\veeet( g(\nabla^g f, \veeet e_k))= \veeet(df (Y))= d f_t(\veeet e_k)= g_t(\nabla^{g_t} f_t,\veeet e_k). \ 
\]
where $g_t=\veeet (g)$, $f_t=\veeet (f)$, $Y_t=\veeet (Y)$. The proof then follows directly by induction on $r$.

To see the equality for the Laplace-Beltrami operator, one can just compute that
\[
\veeet(\Delta_g f) =\sum_j \veeet (\mathrm{Hess}_g (f)(e_j, e_j))= \sum_j \mathrm{Hess}_{\veeet g} (\veeet f)(\veeet e_j, \veeet e_j)= \Delta_{(\veeet g)} (\veeet f). \ 
\]
\end{enumerate}
\end{proof}

\begin{remark}
   The condition on the conformal type of the complex metric is used only in the proof of the first item of the Proposition. In fact, one can prove the other items for complex metrics on manifolds provided that an analog of step 1 is assumed as a hypothesis, namely that there exist local orthonormal frames that admit transports that are orthonormal for the transported metric. 
\end{remark}

\subsection{Transport of real analytic tensors}
\label{subsec: transport analytic}
A useful class of cases for which the problem \eqref{eq: transport problem} admits a unique solution is for the transport of real analytic complex tensors through real analytic complex vector fields. Real analytic objects on $M$ have some nice density properties, which extend some results also more generally in the smooth setting. Here we recall a few results: 
\begin{itemize}[leftmargin=*]
    \item The space of real analytic sections of a real analytic bundle over $M$ is dense in the space of sections
(see \cite[\S 30.12]{Bookanalytic}). In particular, real analytic complex vector fields are dense in the space of complex vector fields.
\item The space $C^{\omega}(M,M)$ of real analytic mappings from $M$ to $M$ is locally modeled on the space of real analytic vector fields by locally applying the inverse of the exponential map for any real analytic Riemannian metric on $M$ (see \cite[\S 37]{Bookanalytic}). In particular, it is locally path-connected through paths of the form $t\mapsto \phi_t$ such that $t\mapsto \phi_t(p)$ is piecewise analytic for all $p\in M$.
\item  The space $\mathrm{Diff}^{\omega}(M)$ of analytic diffeomorphisms of $M$ is an open subset of $C^\omega(M,M)$, and it is therefore locally path-connected through paths of the same type.
\item  $\mathrm{Diff}^\omega(M)$ is dense in $\mathrm{Diff}(M)$ and the inclusion is a homotopy equivalence (see \cite{Tsuboi}). In particular, $\mathrm{Diff}_0^\omega(M)$ is dense in $\mathrm{Diff}_0(M)$. 
\end{itemize}

As we will see, the following theorem will be very useful also in the non-analytic setting because of the nice density properties above.

\begin{thm}
\label{thm: complex transport of tensor}
    Let $M$ be a real analytic manifold. Let $t\mapsto Z_t$ be a path of (local) real analytic complex vector fields on $M$ that is piecewise real analytic with respect to $t\in [0,T]$.
Let $\alpha$ be a real analytic (local) complex tensor on $M$.

Then, for all $p\in M$ there exists a neighborhood $U_p\subset M$, $T_0\in (0, T]$, and a unique path of tensors $t\mapsto \alpha_t$, $t\in [0,T_0]$, such that on $U_p$ 
\[\begin{cases}
    \dot \alpha_t=\mathscr L_{Z_t} \alpha_t\\
    \alpha_0=\alpha.
\end{cases}\ 
\]
A posteriori the path $t\mapsto \alpha_t$ is piecewise real analytic with respect to $t$ and each $\alpha_t$ is a real analytic tensor.
  
  Moreover, if $M$ is compact, one can take $M=U_p$. 
\end{thm}

\begin{proof}
This theorem is mainly an application of the Cauchy-Kovalevskaya Theorem which allows to find the unique real analytic solution. Uniqueness then follows by Holmgren's theorem, which says that any solution to the problem must be real analytic.

    \emph{Step 1}. We first prove the statement for a real analytic function $\alpha$. 
    In a real analytic chart $(x_1, \dots, x_n)$ centered in $p$, solving the equation $\dot \alpha_t=\mathscr L_{Z_t} \alpha_t$ with prescribed $\alpha_0$ is equivalent to solving the Cauchy problem (with real coefficients)
    \[
    \begin{cases}
    \dot u_1^t= \sum_j a^t_j \partial_{x_j} u_1^t -b^t_j\partial_{x_j}u_2^t + a^t_j \partial_{x_j} Re(\alpha_0) -b^tj\partial_{x_j}Im(\alpha_0)\\
    \dot u_2^t= \sum_j b^t_j \partial_{x_j} u_1^t+ a^t_j \partial_{x_j} u_2^t + b^t_j \partial_{x_j} Re(\alpha_0)+ a^t_j \partial_{x_j} Im(\alpha_0)\\
    u_1^0=u_2^0=0,
    \end{cases}\ 
    \]
    where $u_1^t+iu_2^t:= \alpha_t- \alpha_0$, and 
    $Z_t= \sum_j (a_j^t +i b_j^t)\partial_{x_j}.$

    By the Cauchy-Kovalevskaya Theorem, there exists a unique real analytic solution $(u_1^t, u_2^t)$ to this problem in a neighborhood $U_0$ of the point and for small enough time. 
    
\vspace{8pt}

\emph{Step 2.} We prove existence for tensors. Consider again a local analytic chart $(x_1, \dots, x_n)$ around $p$. 
By the previous step, there exists a neighborhood $U_p$ contained in the chart and $T_0>0$ such that for all $t\in [0,T_0]$: 
\begin{itemize}
    \item there exists a unique $x_j^t\colon U_p\to \mathbb C$ such that $\frac d{dt} x_j^t=\partial_{Z_t} x_j^t$ and $x_j^0=x_j$,
    \item $dx^t_1, \dots, dx^t_n$ are $\mathbb C$-linearly independent $1$-forms.
\end{itemize}
By the properties of the complex Lie derivative, $\mathscr L_{Z_t} dx_j^t= \frac d {dt} (dx_j^t)$. 
Now, denote by $e^t_1, \dots, e^t_n\in \Gamma(\mathbb C T U_p)$ the dual frame corresponding to $dx^t_1,\dots, dx^t_n$. Then, $e^0_j=\partial_{x_j}$, and 
    \begin{align*}
        &0= \mathscr L_{Z_t} (dx_j^t(e^t_k) )= (\mathscr L_{Z_t} dx_j^t)(e^t_k) + dx_j^t(\mathscr L_{Z_t}e^t_k), \\
        &  0= \frac d {dt} (dx_j^t(e^k_t) )= d\dot{x}_j^t(e^t_k) + dx_j^t( \dot{e}^t_k)= (\mathscr L_{Z_t} dx_j^t)(e^t_k) + dx_j^t( \dot{e}^t_k),
    \end{align*}
so $\frac d {dt} (e^t_k)=\mathscr L_{Z_t}e^t_k.$

Finally, let $I=(i_1, \dots, i_p), J=(j_1, \dots, j_q)$ be multindices.
    Denote $e^t_I= e^t_{i_1}\otimes \dots \otimes e^t_{i_N}$ and $ dx^t_J= dx^t_{j_1}\otimes dx^t_{j_{N'}}$. One can easily check that $\frac d {dt} e^t_I= \mathscr L_{Z_t} e^t_I$ and $\frac d {dt} dx^t_J= \mathscr L_{Z_t} dx^t_J$. 
    
    Finally, for any analytic tensor $\alpha_0=\sum_{I,J} \alpha_{I,J} e_I\otimes dx_J$ tensor, let $\alpha^t_{I,J}$ be the unique function defined for small times in a neighborhood of $p_0$ such that $\dot \alpha^t_{I,J}=\partial_{Z_t} \alpha^t_{I,J}$; then it is straightforward to see, using the properties of the complex Lie derivative, that \[
    \frac d {dt} \alpha_t= \mathscr L_{Z_t} \alpha_t,
    \]
    with $\alpha_t=\sum_{I,J} \alpha^t_{I,J} e^t_I\otimes dx^t_J$.

\vspace{8pt}

    \emph{Step 3.} We prove uniqueness. Assume that $\alpha_t$ and $\beta_t$ satisfy $\dot \alpha_t=\mathscr L_{Z_t} \alpha_t$, $\dot \beta_t=\mathscr L_{Z_t}\beta_t$ and $\alpha_0=\beta_0$. 
With respect to the bases above we can write $  \alpha_t=\sum_{I,J} \alpha_{I,J} ^t e_I^t \otimes dx_J^t$ and $\beta_t=\sum_{I,J} \beta_{I,J}^t e_I^t\otimes dx_J^t$ for some functions $\alpha_{I,J}^t$ and $\beta_{I,J}^t$. Using the properties of the complex Lie derivative, we deduce that $\frac d {dt} \alpha_{I,J}^t= \partial_{Z_t} \alpha_{I,J}^t$ and $\frac d {dt} \beta_{I,J}^t= \partial_{Z_t} \beta_{I,J}^t$ with $\alpha_{I,J}^0=\beta_{I,J}^0$, hence $\alpha_{I,J}^t=\beta_{I,J}^t$ by Step 1 and by the Cauchy-Kovalevskaya Theorem.

\vspace{8pt}

\emph{Step 4.} If $M$ is compact, one obtains a solution for small time in a neighborhood of any given point, and the solutions coincide on the intersections of such neighborhoods by uniqueness. Extracting a finite subcover from such neighborhoods, one can find analytic solutions on the whole $M$ for all $t\in [0, T']$ for some $T'>0$ small enough.
\end{proof}

\begin{cor}
     Let $M$ be a real analytic manifold. Let $t\mapsto Z_t$ be a path of (local) real analytic complex vector fields on $M$ that is piecewise real analytic with respect to $t\in [0,T]$.

Then, for all $C^1$ tensors $\alpha$, the problem \eqref{eq: transport problem} has at most one solution.
\end{cor}
\begin{proof}
    It follows by the linearity of the equation and by the fact that by Theorem \ref{thm: complex transport of tensor}, $\alpha_t\equiv 0$ is the unique solution for $\alpha=0$.
\end{proof}

\section{Complex elliptic operators}
Here we derive the local expression for the Laplacian of a complex metric, investigate the analytic properties of our complex elliptic operators, and prove Theorems D and E.

We keep the following notation. For a constant $C$ depending on quantities $a_1,\dots, a_n$, we write $C(a_1,\dots, a_n)$. In this notation, the value of such a constant may change in the course of a proof. Analogous to Section 2, for $U$ a domain inside $\C,$ or the surface $S,$ we denote $L^p$ spaces and Sobolev spaces by $L^p(U)$ and $W^{k,p}(U)$ respectively, and local spaces by $L_{loc}^p(U)$ and $W_{loc}^{k,p}(U).$ We use $||\cdot||_{L^p(U)}$, $||\cdot||_{W^{k,p}(U)}$ to denote the norm. When $U=\C$, we'll just write $||\cdot||_{p}$ or $||\cdot||_{k,p}$. We view Beltrami forms defined on domains in $\C$ as functions with respect to the standard coordinate $z$ and use the terms function and form interchangeably.
We'll often work locally in disks inside $\mathbb{C}$ so for this we set the notation, for $r>0,$ $\mathbb{D}_r=\{z\in\mathbb{C}: |z|<r\},$ and $\mathbb{D}_1=\mathbb{D}.$

\subsection{The Laplacian of a complex metric} Throughout this subsection, let $z$ and $w$ be coordinates on the unit disk $\mathbb{D}$. We view $z$ as the standard coordinate and $w$ as a $C^2$ reparametrization of $\mathbb{D},$ $w=w(z).$ Let $\mu$ denote the Beltrami form, defined by $\mu =\frac{\partial_{\overline z} w}{\partial_z w},$ which satisfies $|\mu|<1.$ Let $g$ be a $C^2$ positive complex metric on $\mathbb{D}$ that is locally of the form $g(z)=\lambda(z)dzd\overline{w}$, where $\lambda\in C^2(\mathbb{D},\C^*).$ We compute the Laplacian as a function of $z$. First we record a lemma. Let $\nabla^g$ be the Levi-Civita connection of $g$.
\begin{lem}\label{LCbelt}
  $\nabla_{\partial_{w}}^g\partial_{\overline z}=  \partial_w z\cdot \partial_{\overline z}\overline \mu\cdot \partial_{\overline z}$.
\end{lem}
\begin{proof}
Recall from Lemma \ref{projectionlem} that if $\Pi_1$ is the projections of $\C T\mathbb{D}$ onto the the line spanned by $\partial_{\overline{z}}$, with kernel spanned by $\partial_w$, then $\nabla_{\partial_w}^g \partial_{\overline{z}}=\Pi_1([\partial_w,\partial_{\overline{z}}])$. As well, note that $\frac{\partial_w \overline{z}}{\partial_w z}=-\overline{\mu}$, as can be seen from the expression $$0=\partial_{\overline{w}}w=\partial_{\overline{w}}\overline{z}\partial_{\overline{z}}w+\partial_{\overline{w}}z\partial_z w.$$ Thus, writing out
\begin{equation}
    \label{eq: change basis}
    \partial_w= \partial_w z \cdot\partial_z + \partial_w \overline z \cdot \partial_{\overline z}= \partial_w z ( \partial_z - \overline \mu \partial_{\overline z}), 
\end{equation}
we get that 
\begin{align*}
    [\partial_w,\partial_{\overline z}] =&\partial_{\overline z} (\partial_w z\overline \mu ) \cdot \partial_{\overline z}-\partial_{\overline z} \partial_w z \cdot \partial_z  \\
    &= \partial_{\overline z} (\partial_w z\overline \mu ) \cdot \partial_{\overline z}-\frac{ \partial_{\overline z} \partial_w z }{\partial_w z}\partial_w - \partial_{\overline z} \partial_w z  \overline \mu \partial_{\overline z} \\
    &= \partial_w z \partial_{\overline z}\overline \mu \partial_{\overline z}-\frac{ \partial_{\overline z} \partial_w z }{\partial_w z}\partial_w,
\end{align*}
so $\nabla_{\partial_{w}}^g\partial_{\overline z}=  \partial_w z \partial_{\overline z}\overline \mu \partial_{\overline z}$.
\end{proof}

\begin{prop}\label{Deltaformula}
    In the holomorphic coordinate $z$, the Laplacian $\Delta_g$ is expressed as $$\Delta_g = \frac{4}{\lambda\overline{\partial_z w}}\partial_{\overline{z}}(\partial_z-\overline{\mu}\partial_{\overline{z}}).$$
\end{prop}
\begin{proof}
In the $g$-isotropic basis $\{\partial_{\overline{z}},\partial_w\}$ for the complexified tangent bundle $\C T\mathbb{D},$ for all $C^2$ functions $f\colon \mathbb D \to \C$, 
$$\Delta_g f= \frac{2}{g( \partial_{ w}, \partial_{\overline z})} \mathrm{Hess}_g(f)(\partial_{w},\partial_{\overline z})= \frac{2}{g(\partial_{w}, \partial_{\overline z})}\left(\partial_w \partial_{\overline{z}}f - (\nabla_{\partial_w}^g\partial_{\overline z})f\right)\ .$$
Hence, by Lemma \ref{LCbelt} and Equation \eqref{eq: change basis},
\begin{align*}
    \Delta_g=& \frac{2}{g( \partial_{ w}, \partial_{\overline z})} \left( 
\partial_w\partial_{\overline z} - (\partial_w z) (\partial_{\overline z} \overline \mu) \partial_{\overline z}\right) 
 \\
 &=
 \frac{2 \partial_w z}{g( \partial_{ w}, \partial_{\overline z})} \left(\partial_z \partial_{\overline z}  - \overline \mu \partial_{\overline z}\partial_{\overline z}  - \partial_{\overline z} \overline \mu \partial_{\overline z} 
 \right)\\
 &= \frac{ 4}{\lambda \overline{\partial_z w}} (\partial_{\overline z } (\partial_z- \overline \mu \partial_{\overline z})). 
\end{align*}
\end{proof}

\subsection{Interior elliptic estimates}
Here we deduce interior elliptic estimates for complex Laplacians via estimates for the Beltrami operator. Note that we need elliptic estimates in particular over, say, Schauder estimates, for the proof of Lemma \ref{lem: analyticity}. We also need to apply the holomorphicity results of \cite{ESbelt} in Section 6, which are framed using Sobolev spaces, so it's more convenient to establish elliptic estimates and Fredholm properties in Sobolev spaces. 

We make reference to the Beurling transform, which is a singular intergal operator $T$ defined a priori on $C_0^\infty(\C)$, and which, for all $1<p<\infty$, extends to a continuous operator $T:L^p(\C)\to L^p(\C)$, which is an isometry for $p=2$. The operator norm, which we denote by $C_p,$ varies continuously with $p$. For a reference, see \cite[Chapter 4]{IT}. The operator $T$ is appearing because, for a Beltrami form $\mu,$ the Beltrami operator $\Big (\partial_{\overline{z}}-\mu\partial_z \Big )$ has elliptic estimates on $W^{k,p}$ spaces (i.e., is Fredholm) only for $p$ such that $||\mu||_\infty C_p<1$ (see \cite[Theorem 2.3]{ESbelt}).


As in the previous subsection, let $g$ be a positive complex metric on $\mathbb{D}$ of the form $g=\lambda dz d\overline{w},$ and let $\mu$ be the Beltrami form of $w$. We now choose $k>2$ and $l>k+2$. We assume that $\mu$ is $W^{l,\infty}$. We use $W^{k+2,p}$ for the Sobolev spaces on which our linear maps act. For a linear differential operator $L,$ we define $||L||_{W^{k,p}(U)}$ to be the sum of the $W^{k,p}(U)$ norms of the coefficients, and we say that $L$ is $W^{k,p}$ (resp. $W_{loc}^{k,p}$) if all of the coefficients are $W^{k,p}$ (resp. $W_{loc}^{k,p}$). The main estimate is as follows.

\begin{prop}\label{interiorestimate}
Let $L$ be a $W_{loc}^{k+1,\infty}$ linear second order operator with the same principle symbol as $\Delta_g$. Let $u\in L_{\mathrm{loc}}^{p}(\C)$ and $v\in W_{\mathrm{loc}}^{k,p}(\C)$, with $u$ satisfying $Lu=v$ in the weak sense. Then $u\in W_{\mathrm{loc}}^{k+2,p}(\C)$ and for all $r<R<\infty$ and $p$ such that $||\mu||_\infty C_p<1$, there exists $C=C(||L||_{W^{k+1,\infty}(\mathbb{D}_R)},||\mu||_{W^{k+1,\infty}(\mathbb{D}_R)},r,R,k,p)$ such that 
$$
    ||u||_{W^{k+2,p}(\mathbb{D}_r)}\leq C(||u||_{W^{k,p}(\mathbb{D}_R)}+||v||_{W^{k,p}(\mathbb{D}_R)}).$$
    \end{prop}

\begin{remark}
If $(c_1,\overline{c_2})$ are $W^{l,\infty}$ complex structures, then the associated Bers metric $h$ is $W^{l,2}$ (see \cite[Theorem B]{ESbelt}). For $l>k+2$, by Morrey's inequality, $h$ is $k+1$ times differentiable with H{\"o}lder continuous $(k+1)^{th}$-derivatives, hence $\Delta_h$ is $W^{k+1,\infty},$ and the hypothesis of Proposition \ref{interiorestimate} applies to $L=\Delta_h.$
\end{remark}
Since the Laplacian factors as a product of first order complex elliptic operators, Proposition \ref{interiorestimate} essentially follows from the interior elliptic estimates for the Beltrami operator. 

\begin{proof}[Proof of Proposition \ref{interiorestimate}]
It suffices to assume that $u$ is $C^\infty$ and has compact support contained in $\mathbb{D}_r$, for then the result follows from standard approximation arguments, using the density of $C_0^\infty(\C)$ in all of the Sobolev spaces for disks. One can follow the proof of Theorem 10.3.6 and Corollary 10.3.10 in \cite{Ni}, as long as one remembers to restrict the range of $p$ so that $||\mu||_\infty C_p<1$. 

Under this assumption, we first treat the case where $L=\Delta_g$. The hypothesis $\Delta_g u = v$ means that
$u':=\Big ({\partial_z} - \overline{\mu} {\partial_{\overline{z}}}\Big)u $ satisfies ${\partial}_{\overline{z}}u' = \lambda \overline{\partial_z w}v$. By elliptic regularity for ${\partial}_{\overline{z}}$, we obtain 
\begin{align*}
       ||u'||_{k+1,p}&\leq C(r,k)(||u'||_{k,p}+||\lambda \overline{\partial_z w}v||_{k,p}) \\
       &\leq C(||L||_{W^{k,\infty}(\mathbb{D}_r)},r,k)(||u'||_{k,p}+||v||_{k,p}).
\end{align*}
Note that $||u'||_{k,p}\leq C(||\mu||_{W^{k,\infty}(\mathbb{D}_r)})||u||_{k+1,p},$ so we can turn the inequality above into 
 \begin{equation}\label{secondestimate}
     ||u'||_{k+1,p}\leq C(||L||_{W^{k,\infty}(\mathbb{D}_r)},r,k)(||u||_{k+1,p}+||v||_{k,p}).
 \end{equation}
Now, the elliptic estimate for the Beltrami equation, Theorem 2.3 from \cite{ESbelt}, applied to the equation $\Big ({\partial}_{ z} - \overline{\mu} {\partial}_{ \overline{z}}\Big)u=u',$ returns
\begin{equation}\label{thirdestimate}
    ||u||_{k+1,p}\leq C(||\mu||_{W^{k+1,\infty}(\mathbb{D}_r)},r,k,p)(||u||_{k+1,p}+||u'||_{k+1,p}).
\end{equation}
Here we stress that we could only apply the above estimate because $||\mu||_\infty C_p<1.$ Note also that \cite[Theorem 2.3]{ESbelt} applies to the ordinary Beltrami operator, rather than the conjugate Beltrami operator $\Big ({\partial}_{ z} - \overline{\mu} {\partial}_{ \overline{z}}\Big)$, but the two operators are related by a complex conjugation and hence the same estimate holds.
Composing the estimates (\ref{secondestimate}) and (\ref{thirdestimate}), we arrive at 
\begin{equation}\label{fourthestimate}
    ||u||_{k+2,p}\leq C(||L||_{W^{k+1,\infty}(\mathbb{D}_r)},r,k,p)(||u||_{k+1,p}+||v||_{k,p}).
\end{equation}
We want to adjust the constants in order to replace the $||u||_{k+1,p}$ on the right with $||u||_{k,p}$. By the interpolation inequality \cite[Theorem 7.28]{GT}, there exists $C(r,k)>0$ such that for every $\varepsilon>0$, 
\begin{equation}\label{interpolationestimate}
    ||u||_{k+1,p}\leq C(r,k)(\varepsilon||u||_{k+2,p}+\varepsilon^{-(k-1)}||u||_{k,p}).
\end{equation}
Inserting (\ref{interpolationestimate}) into (\ref{fourthestimate}), for $\varepsilon>0$ small enough we can rearrange to obtain 
\begin{equation}\label{fifthestimate}
    ||u||_{k+2,p}\leq C(||L||_{W^{k+1,\infty}(\mathbb{D}_r)},r,k,p)(||u||_{k,p}+||v||_{k,p}),
\end{equation}
which is the desired result for $L=\Delta_g.$

Next, we establish the result for arbitrary $L$. Toward this, we just need to control $||\Delta_g u||_{k,p}$ in terms of $||L u||_{k,p}$ and $||u||_{k,p}$. The only needed observation is that our hypothesis guarantees that $(L-\Delta_g)$ is an operator of order at most $1$, and therefore
\begin{align*}
    ||\Delta_g u||_{k,p}&\leq ||L u||_{k,p} +||(L-\Delta_g)u||_{k,p}\\
    &\leq ||L u||_{k,p}+C(||L||_{W^{k,\infty}(\mathbb{D}_r)}, ||\mu||_{W^{k+1,\infty}(\mathbb{D}_r)})||u||_{k+1,p}.
\end{align*}
We then insert into the last inequality into the $v$ slot of (\ref{fifthestimate}) and use the interpolation inequality as above to turn the $||u||_{k+1,p}$ term into a $||u||_{k,p}$ term. This establishes the proposition in the $C_0^\infty$ case, and as we discussed above, the main proposition follows from arguments of \cite{Ni}.
\end{proof}

Having finished Proposition \ref{interiorestimate}, we deduce the expected consequences that should come along with elliptic estimates. From bootstrapping the estimate of Proposition \ref{interiorestimate}, we have the following.
\begin{prop}[Regularity]\label{nonlinearreg}
 Let $U\subset \C$ and $V\subset \C\times \C^2$ be open domains, let $F:U\times V \to \C$ be a $W^{k+1,\infty}$ function, and let $L$ be a $W_{loc}^{k+1,\infty}$ semi-linear second order complex elliptic operator of the form $$Lu(z) = \Delta_g u +F(z,u,Du).$$
If $u\in W_{\mathrm{loc}}^{k,p}(U)$ satisfies $Lu=0$,  then for $p$ such that $||\mu||_\infty C_p<1$, $u$ is in $W_{loc}^{k+2,p}(U).$
\end{prop}
    It of course follows that, if $L$ is $C^\infty,$ then $u$ is $C^\infty.$
\begin{proof}
Since $u\in W_{\mathrm{loc}}^{k,p}(U)$ the function $G(z)=F(z,u(z),Du(z))$ is $W_{\mathrm{loc}}^{k-1,p}(U)$. Therefore, $\Delta_g u \in W_{\mathrm{loc}}^{k-1,p}(U)$. Applying the interior elliptic estimate from Proposition \ref{interiorestimate} with $\Delta_g$, we obtain that $u$ is $W_{\mathrm{loc}}^{k+1,p}.$ This implies that $G$ and hence $\Delta_g u$ are $W_{\mathrm{loc}}^{k,p}.$ Elliptic estimates then give that $u$ is $W_{\mathrm{loc}}^{k+2,p}.$
\end{proof}
On closed surfaces, we have the expected Fredholm property. Given a positive complex metric $g$, locally of the form $g=\lambda dz d\overline{w}$, with $\partial_{\overline{z}} w=\mu \partial_z w$, we say that $g$ has  Beltrami form $\mu$.
\begin{prop}[Fredholmness]\label{fredholmness}
    Let $S$ be a closed surface, $g$ a positive complex metric on $S$ with induced $W^{l,\infty}$ complex structures $(c_1,\overline{c_2})$, and with Beltrami form $\mu \in W^{l,\infty}$. For $l>k+2,$ let $$L:W^{k+2,p}(S)\to W^{k,p}(S)$$ be a linear operator whose principle symbol agrees with that of $\Delta_h$. For $p$ such that $||\mu||_\infty C_p<1,$ $L$ is a Fredholm operator. If $g$ is deformable to a Bers metric, the Fredholm index is zero. 
\end{prop}
\begin{proof}
    Once one has the estimate of Proposition \ref{interiorestimate}, one can carry out the usual ``real" proof of Fredholmness verbatim (for example, see \cite[Theorem 10.4.7]{Ni}). If $g$ is deformable to a Bers metric, then $L$ is homotopic through Fredholm operators to the Euclidean Laplacian, so its index is zero.
\end{proof}
\begin{remark}\label{moregeneral}
We believe it should be possible to prove analogous results for more general second order complex linear elliptic operators, for example those of the form $$a \partial_z^2+b\partial_z\partial_{\overline{z}}+c \partial_{\overline{z}}^2,$$ where $a,b,c$ are complex-valued functions satisfying $a
\zeta^2 + b\zeta\overline{\zeta}+c\overline{\zeta}^2\neq 0$ for all $\zeta\in\C^*,$ using more general factorizations (see \cite[\S 7.8]{Ast}), or higher order singular integral operators, as developed by Begehr (see \cite[\S 5]{Beg}). We worked with the Laplacian rather than a more general operator because doing so allows for more expeditious proofs that use fewer tools and also highlights special features of the Laplacian.

The same remark applies to Lemma \ref{lem: analyticity}.
\end{remark}

\subsection{Proof of Lemma \ref{lem: analyticity}}\label{sec: thmE}
We prove Lemma \ref{lem: analyticity}. Let $g$ be a real analytic positive complex metric on an open domain $U\subset \C$, locally of the form $g=\lambda dzd\overline{w}.$ Recalling the setup, let $V\subset \C\times \C^2$ be an open domain, let $F:U\times V \to \C$ be a real analytic function that's linear in the second two coordinates. Throughout this subsection, suppose we are given $u\in C^2(U,\C)$ with $(u,Du)(U)\subset V$ that satisfies the inhomogeneous second order complex elliptic linear equation $$Lu(z) := \Delta_g u +F(z,u,Du)=0.$$ The statement of Lemma \ref{lem: analyticity} is that $u$ is real analytic.

To prove the theorem, we will use Cauchy's method of majorants and borrow an idea from Kato's work on real analyticity for smooth solutions to certain real analytic (real) elliptic second order equations that are non-linear \cite{Kato}. The idea of Kato from \cite{Kato} has been used by Hashimoto \cite{Hash} and Blatt \cite{Blatt} to study analyticity for more general fully non-linear elliptic equations.
 \begin{remark}\label{rem: differences}
     We have to control a few extra terms compared to the real case. The reason is as follows. For any real semi-linear elliptic second order PDE, around any point one can find a small neighbourhood in which we have a change of coordinates that transforms the principle symbol to the Euclidean Laplacian. For complex elliptic PDE, such local transformations do not exist. For example, the reader can verify that no such transformation exists for $$\partial_z\partial_{\overline{z}}+c\partial_{\overline{z}}^2,$$ where $c$ is a constant of norm strictly less than $1$.
 \end{remark}

 \begin{remark}
 In \cite{Blatt}, Blatt proves the result analogous to Lemma \ref{lem: analyticity} for fully non-linear second-order elliptic operators. Lemma \ref{lem: analyticity} could surely be pushed through to non-linear settings.
 \end{remark}

We now begin the preparations toward the proof.  A multi-index is a tuple of positive integers $\alpha=(m_1,\dots, m_N),$ and we write $|\alpha|=\sum_{i=1}^N m_i$. Given (real or complex) coordinates $x_1,\dots, x_n$ and a multi-index $\alpha=(m_1,\dots, m_N),$ we set $\partial^\alpha = \partial_{x_1}^{m_1}\dots \partial_{x_N}^{m_N}$. By Proposition \ref{nonlinearreg}, we know that any solution $u$ to $Lu=0$ is $C^\infty$, so we can apply Cauchy's method of majorants to $u$.
\begin{thm}[Cauchy's method of majorants]\label{Cauchyestimates}
Given a domain $U\subset \R^n,$ a $C^\infty$ function $f:U\to \C$ is real analytic if and only if for any relatively compact sub-domain $V$ and all multi-indices $\alpha,$ there exist constants $A,C>0$ such that
$$||\partial^\alpha f||_{L^\infty(V)}\leq CA^{|\alpha|}|\alpha|!.$$  
\end{thm}
Of course, one can absorb the $C$ into the $A^{|\alpha|}$ term, so we'll often replace the condition with $||\partial^\alpha f||_{L^\infty(V)}\leq A^{|\alpha|}|\alpha|!.$
Since real analyticity is a local property, we're welcome to work locally and assume $U=\mathbb{D}$, and furthermore we can assume that our function $u$ admits uniform $C^\infty$ bounds in $\mathbb{D}.$
By Sobolev embedding theorems, we can work in a space such as $W^{k,2}$ (for any $k$) rather than $L^\infty$: namely, when working on the unit disk $\mathbb{D}$, one need only prove that for some $k$ and any $0<r<1$, 
\begin{equation}\label{RAL2}
||\partial^\alpha u||_{W^{k,2}(\mathbb{D}_r)}\leq CA^{|\alpha|}|\alpha|!,
\end{equation}
where again $A,C>0$ are fixed constants, which now depend on $r$. 
Working with $W^{k,2}$ norms is more convenient because of the elliptic theory. We will also repeatedly use that $W^{k,2}$ spaces are Banach algebras under multiplication when $k>1$.
\begin{prop}\label{Balgebra}
    Let $k>1$. There exists $C=C(r,k)>0$ such that for any two compactly supported $W^{k,2}$ functions $f_1,f_2$ $$||f_1f_2||_{W^{k,2}(\mathbb{D}_r)}\leq C||f_1||_{W^{k,2}(\mathbb{D}_r)}||f_2||_{W^{k,2}(\mathbb{D}_r)}.$$
\end{prop}
For convenience, we replace $||\cdot||_{W^{k,2}(\mathbb{D}_r)}$ with $C(r,k)||\cdot||_{W^{k,2}(\mathbb{D}_r)}$ so that multiplication in $W^{k,2}(\mathbb{D}_r)$ has norm at most $1$. 

Preliminary results aside, fix $R>r>0$ and $k>2$. Similar to \cite{Kato}, \cite{Hash}, and \cite{Blatt}, we define a cutoff function $\eta$ in $C_0^\infty(\mathbb{D},[0,1])$ that is equal to a fixed constant $c>0$ in $\mathbb{D}_{r}$ and with support in $\mathbb{D}_R,$ with $c$ chosen to be small enough to ensure that $||\eta||_{W^{k,2}(\mathbb{D}_R)}\leq 1$ (note that in previous works such as \cite{Kato}, \cite{Hash}, and \cite{Blatt}, the authors choose $c=1$ while not imposing a $W^{k,2}$ bound). Rather than estimate terms such as $\partial^\alpha u$, we will actually estimate things like $\eta^N\partial^{\alpha} u$ in the subdisk $\mathbb{D}_R$ and then simply restrict to $\mathbb{D}_{r}$. Since $\mathbb{D}_R$ is specified to contain the support, we will abuse notation slightly and write $||\cdot||_{k}$ for $||\cdot ||_{W^{k,2}(\mathbb{D}_R)}.$ The point of multiplying by $\eta^N$ is that we don't have to deal with passing to subdisks: if $v$ is any regular enough function, with an $s<r$ such that $\mathbb{D}_s$ also contains the support of $\eta$, then, $$||\eta v||_{k+2}\leq C(||L_0(\eta v)||_{k}+||\eta v||_{k}).$$
The following notation will be helpful. For $N=0$ and $1$, set $\widetilde{L}_N=\max_{|\alpha|=N} ||\partial^\alpha u||_{k}$ and for $N\geq 2$ put
$$\widetilde{L}_N=\max_{|\alpha|=N} ||\eta^{N-2}\partial^\alpha u||_{k}.$$ Further set $$L_N = \widetilde{L}_{N+1}+\widetilde{L}_N+1.$$ 
Since $\eta=c$ on $\mathbb{D}_r,$ $$\max_{|\alpha|=N+1}||\partial^\alpha u||_{W^{k,2}(\mathbb{D}_r)}\leq c^{-(N-1)}\Tilde{L}_{N+1}\leq c^{-(N-1)}L_N.$$ Hence, to prove Lemma \ref{lem: analyticity}, it suffices to show that $L_N\leq B^{N+1} (N+1)!$ for some $B>0,$ since then one has the bound  $$\max_{|\alpha|=N+1}||\partial^\alpha u||_{W^{k,2}(\mathbb{D}_r)}\leq c^2\big (\frac{B}{c}\big )^{N+1}(N+1)!.$$

In the proof below, keep in mind that at the beginning of the section we declared that, without comment, uniform constants $C$ can change in the course of the proofs. Below, without stating explicitly, we will use the Banach algebra property to make estimates with the same idea as $\max_{|\alpha|=N}||\eta^N\partial^\alpha u||_k\leq \max_{|\alpha|=N}||\eta^2||_k\Tilde{L}_N\leq \Tilde{L}_N.$ 
\begin{proof}[Proof of Lemma \ref{lem: analyticity}]
Let $u$ be a solution to $Lu=0,$ which we know is $C^\infty.$ Let's write the equation $\Delta_g u + F(x,u,Du)=0$ as $L_0 + \lambda\overline{\partial_z w}\hat{F}(x,u,Du)=0,$ where now $L_0=\partial_z\partial_{\overline{z}}-\overline{\mu}\partial_{\overline{z}}^2,$ and $\hat{F}$ contains all of the first order and zero order information of $L$. To simplify things, let's just abuse notation and redefine $F$ to be $\lambda\overline{\partial_z w}\hat{F}$. Getting rid of the conformal factor and the first order terms from $\Delta_g$ makes for a slightly cleaner computation below. With these modifications done, we write $F$ as $$F(z,u,Du) = f_1(z)\partial_z u(z) + f_2(z)\partial_{\overline{z}}u(z) + f_3(z)u(z)+f_4(z),$$
with the $f_i$'s real analytic. We fix $A>0$ such that for every $i$ and all multi-indices $\alpha$ of length $N$, $||\partial^\alpha f_i||_{k}\leq A^NN!$

We prove by induction on the length of the multi-index $N$ that there is a $B>0$ such that $L_N\leq B^{N+1}(N+1)!$. In the work below, we will put specifications on $B$ when they are needed. The base case is trivial and we assume the result holds for a given $N$. Let $\alpha$ be a multi-index of length $N,$ and $\beta$ a multi-index of length $2$.  Write $$\eta^N\partial^{\alpha+\beta}u = \partial^\beta (\eta^N \partial^\alpha u) + [\eta^N,\partial^\beta]\partial^\alpha u.$$
The elliptic estimate of Proposition \ref{interiorestimate} then returns 
\begin{align*}
    ||\eta^N \partial^{\alpha+\beta}u||_{k}&\leq ||\partial^\beta\eta^N \partial^{\alpha}u||_{k}+||[\eta^N,\partial^\beta] \partial^{\alpha}u||_{k} \\
&\leq ||\eta^N \partial^{\alpha}u||_{k+2}+||[\eta^N,\partial^\beta] \partial^{\alpha}u||_{k} \\
    &\leq C(||L_0(\eta^N \partial^\alpha u)||_{k}+||\eta^N \partial^\alpha u||_{k})+||[\eta^N,\partial^\beta] \partial^{\alpha}u||_{k}.
\end{align*}
Similar to above, we express $L_0(\eta^N \partial^\alpha u)$ as
$$L_0(\eta^N \partial^\alpha u)=\eta^N L_0(\partial^\alpha u)+[L_0,\eta^N]\partial^\alpha u.$$ We further unfoil, using that $L_0(u)+F(x,u,Du)=0$, $$L_0(\partial^\alpha u)=\partial^\alpha L_0 u + [L_0,\partial^\alpha]u  =\partial^\alpha F+\partial^\alpha(\overline{\mu})\partial_{\overline{z}}^2u.$$
Moreover, $$||L_0(\eta^N\partial^\alpha u)||_{k}\leq ||\eta^N\partial^\alpha F||_{k}+||\partial^\alpha(\overline{\mu})||_{k}||\partial_{\overline{z}}^2u||_{k} + ||[L_0,\eta^N]\partial^\alpha u||_{k}.$$ 
Returning to the main estimate, we have
\begin{align*}
    ||\eta^N \partial^{\alpha+\beta}u||_{k}
    &\leq C(||\eta^N \partial^\alpha F||_{k}+||[L_0,\eta^N]\partial^\alpha u||_{k}+||\eta^N \partial^\alpha u||_{k}+||\eta^N \partial_{\overline{z}}^2 u||_{k})+||[\eta^N,\partial^\beta] \partial^{\alpha}u||_{k} \\
    &=: J_1 + J_2 + J_3 + J_4+J_5.
\end{align*}
We find a suitable bound for each $J_i$. The terms $J_3$ and $J_4$ are both straightforwardly bounded by $L_{N-1}$ and $L_1$ respectively, and hence, making $B\geq \max\{L_1,1\}$, both bounded above by $B^{N}N!$. To handle $J_2$ and $J_5,$ first take note of the formula, for any sufficiently regular function $v,$ for coordinates $x_1,\dots, x_n,$ writing $\partial_{ij}=\partial_{x_i}\partial_{x_j},$
$$\partial_{ij}(\eta^N v) = (N(N-1)\eta^{N-2}\partial_i \eta\partial_j \eta+N\eta^{N-1}\partial_{ij}\eta)v + N\eta^{N-1}(\partial_i \eta\partial_j v+\partial_j\eta \partial_i v)+\eta^N \partial_{ij}v,$$ from which we derive the identity 
\begin{equation}\label{thingtoestimate}
    [\eta^N,\partial_{ij}]\partial^\alpha u = (N(N-1)\eta^{N-2}\partial_i \eta\partial_j \eta+N\eta^{N-1}\partial_{ij}\eta)\partial^\alpha u + N\eta^{N-1}(\partial_i \eta\partial_j \partial^\alpha u+\partial_j\eta \partial_i \partial^\alpha u).
\end{equation}
Splitting $L_0$ as $L_0=\partial_z\partial_{\overline{z}}-\overline{\mu}\partial_{\overline{z}}^2$,
$$J_2 + J_5\leq C\left(\left\|\left[\partial_z\partial_{\overline{z}},\eta^N\right]\partial^\alpha u\right\|_{k}+ \left\|\left[\partial_{\overline{z}}^2,\eta^N\right]\partial^\alpha u\right\|_{k}+ \left|\left[\partial^\beta,\eta^N\right]\partial^\alpha u\right\|_{k}\right).$$
Via the formula (\ref{thingtoestimate}), and the Banach algebra property, all three of the terms above can be controlled by $CN(N-1)L_{N-1}+CNL_N,$ and therefore $$J_2+J_5\leq CN(N-1)L_{N-1}+CNL_N\leq C(N+1)!(B^N+B^{N+1}).$$
To bound $J_1,$ we use the higher order product rule. For notation's sake, if $\partial^\alpha = \partial_{x_{i_N}}\dots \partial_{x_{i_1}}$, write $\partial_1^k = \partial_{x_{i_N}}\dots \partial_{x_{i_{N-k+1}}},$ $\partial_2^k = \partial_{x_{i_{N-k}}}\dots\partial_{x_{i_1}}.$ Then, $$\partial^\alpha F = \Big (\sum_{j+k=N}{N \choose j} (\partial_1^j f_1\partial_2^k \partial_z u +\partial_1^j f_2\partial_2^k \partial_{\overline{z}} u +\partial_1^j f_3\partial_2^k u)\Big )+\partial^\alpha f_4.$$ Multiplying by $\eta^N$ and taking the norm, and using $||\partial^\alpha f_i||_k\leq A^NN!,$ we get $$||\eta^N\partial^\alpha F||_k\leq C \sum_{j+k=N}{N \choose j} (A^jj!\Tilde{L}_{k+1}+A^jj!\Tilde{L}_k) + A^N N!.$$
We add the condition that $B\geq A$. Inserting $\Tilde{L}_k\leq L_{k-1}\leq B^{k}k!$, we obtain further bounds:
\begin{align*}
    ||\eta^N\partial^\alpha F||_k&\leq C \sum_{j+k=N}N! (A^jB^{k+1}(k+1)+A^jB^k) + A^N N!\\
    &\leq C\sum_{j+k=N}N!(B^{N+1}(k+1)+B^{N})+B^NN! \\
    &\leq C N(N!)(B^{N+1}(N+1)+B^{N}) +B^N N!.
\end{align*}
Using trivial bounds such as $B^k\leq B^{k+1}$, we conclude that $||\eta^N\partial^\alpha F||_k\leq C(N+2)!B^{N+1}.$
Putting everything together,
\begin{align*}
    L_{N+1}&\leq C((N+2)!B^{N+1}+ (N+1)!B^{N+1}+B^N N!)
    \leq C (N+2)! B^{N+1}
\end{align*}
Imposing now that $B$ is greater than or equal to our constant $C$, which does not depend on $N$, we have the desired bounded $L_{N+1}\leq (N+2)!B^{N+2}$. As discussed right before the proof, this bound implies the Cauchy estimate for $u$, and Theorem \ref{Cauchyestimates} then yields that $u$ is real analytic.
\end{proof}

\subsection{Proof of Theorem D}
The aim of this section is to prove Theorem \ref{thm F extended}, which directly implies Theorem D. Toward this goal, we first prove some general results. We define more generally, for any positive complex metric $g,$ $L_{g,\lambda}: W^{k+2,p}(S)\to W^{k,p}(S)$ by $L_{g,\lambda} u = \Delta_gu - \lambda u$. By Proposition \ref{fredholmness}, when $g$ is deformable to a Bers metric, $L_{g, \lambda}$ has index zero and is invertible if and only if $\lambda$ is not an eigenvalue for $\Delta_g$.

\begin{lem}\label{openness}
    Let $g$ be a positive complex metric such that $\Delta_g$ has Fredholm index zero, $\lambda\in \C,$ and suppose that $L_{g,\lambda}$ is an isomorphism. Then there exists a neighbourhood $U$ of $g$ inside $\mathcal{CM}(S)$ such that for all $g'\in U$, $L_{g',\lambda}$ is an isomorphism. 
\end{lem}
\begin{proof}
 Suppose that the conclusion of the lemma fails, so that we can find positive metrics $g_n$ converging to $g$ such that $L_{g_n,\lambda}$ is not an isomorphism. Since for $n$ large, each $L_{g_n,\lambda}$ is Fredholm of index zero, this means that for every $n$ we can find a non-zero $u_n\in W^{k+2,p}(S)$ such that $$L_{g_n,\lambda}u_n=0.$$   Since each $L_{g_n}^\lambda$ is linear, we can rescale each $u_n$ so that $||u_n||_{W^{k+2,p}}=1$. By the Rellich-Kondrachov theorem, $u_n$ subconverges in $W^{k,p}$ to some non-zero $u_\infty$ that satisfies $L_{g,\lambda} u=0$ weakly. By Proposition \ref{interiorestimate}, $u\in W^{k+2,p}(S),$ and we've contradicted the fact that $L_{g,\lambda}$ is an isomorphism.
\end{proof}

We will say that $c_1\in \mathcal C(S)$ and $\overline {c_2}\in \mathcal C(\overline S)$ are \textbf{co-real analytic} (or that $\ccpair$ is a co-real analytic pair) if they induce the same real analytic structure on $S$. Equivalently, $c_1$ and $\overline{c_2}$ are co-real analytic if the Beltrami form of one with respect to the other is real analytic (by \cite[Theorem 3]{AB}). Observe that the diagonal action of $\mathrm{Diff}_0(S)$ preserves this condition. 

By the density of $\Diff_0^\omega(S)$ in $\Diff_0(S)$, the projection $\CSCS\to \TSTS$ is such that co-real analytic pairs in each fiber are dense in the fiber.

We will use the following result about the spectrum of the Laplacian of Bers metrics from co-real analytic complex structures, which is independently interesting.
\begin{prop}\label{isotopypreserve}
 Let $(c_1, \overline{c_2}),(c_1', \overline{c_2}')\in \mathcal C(S)\times \mathcal C(\overline S)$ be two pairs of co-real analytic complex structures on $S$ that project to the same point in $\TSTS$. 

 Then $L_{\hpair,\lambda}$ is an isomorphism if and only if $L_{\vb*{h}_{(c'_1, \overline{c_2}')},\lambda}$ is an isomorphism. 
\end{prop}
\begin{proof}
We will actually prove the contrapositive statements, namely $\lambda$ lies in the spectrum of $\Delta_{\hpair}$ if and only if $\lambda$ lies in the spectrum of $\Delta_{\vb*{h}_{(c'_1, \overline{c_2}')},\lambda}$.

Fix on $S$ the real analytic structure induced by $c_1$ and $\overline{c_2}$ and denote by $\mathrm{Diff}^\omega(S)$ the corresponding subgroup of real analytic diffeomorphisms. 

Since, for all $ \phi\in$ $\mathrm{Diff}(S)$, $L_{\hpair,\lambda}$ is an isomorphism if and only if $L_{\vb*{h}_{( \phi^*c_1,  \phi^*\overline{c_2})},\lambda}$ is an isomorphism, we can assume that $c_1=c_1'$ and that $\overline{c_2}'$ and $\overline{c_2}$ lie in the same $\mathrm{Diff}_0^{\omega}(S)$-orbit.

As we mentioned in the introduction to Section \ref{subsec: transport analytic}, $\mathrm{Diff}_0^{\omega}(S)$ is locally analytically path connected, so we are left to show the following: if $\lambda$ is an eigenvalue for $\Delta_{\hpair}$, for all real analytic path $\phi_t\in \mathrm{Diff}^\omega_0(S)$, $t\in [0,1]$, $\lambda$ is an eigenvalue for $\Delta_{\vb*h_{(c_1, \phi_t^*{\overline {c_2}})}}$ for all $t\in [0,1]$. 

Denote $h_t=\vb*h_{(c_1, \phi_t^*{\overline {c_2}}) }$ and 
$$A^\lambda = \{t\in [0,1]: \lambda \text{ lies in the spectrum of }\Delta_{h_t} \}.$$ $A_\lambda$ contains $0$ and, by Lemma \ref{openness}, $A^{\lambda}$ is closed. Let $t_0$ be the maximum of the connected component of $A^{\lambda}$ containing $0$, and assume for the sake of contradiction that $t_0<1$.  Denoting by $z$ and $\overline w$ local coordinates of $c_1$ and $\overline{c_2}$ respectively, as well as $\dot \phi_t= \xi_t \partial_{z}+\overline \xi_t \partial_{\overline z}$, define the vector field $Z_t:= \left(\xi_{t+t_0} \frac{\partial_z \overline w_{t+t_0}}{\partial_{\overline z}\overline w_t}+\overline{\xi_{t+t_0}}\right)\partial_{\overline z}$ as in Example \ref{ex: Lie derivative of Bers metrics}, so that, in the notation of Definition \ref{def: analytic transport}, $\veet (\vb*h_{(c_1, \phi_{t_0}^*(\overline {c_2})) })= \vb*h_{(c_1, \phi_{t_0+t}^*(\overline {c_2}))} $. Now, let $u$ be in the kernel of $L_{h_t,\lambda}.$ By Lemma \ref{lem: analyticity}, $u$ is a real analytic function. By Theorem \ref{thm: complex transport of tensor}, there exists $t_1$ with $t_0<t_1\le 1$ such that the analytic transport $u_t=\veee_{Z_t} u$ is defined for all $t\in [0, t_1-t_0]$ (with $u_t$ being non-zero), and, by Proposition \ref{prop: transport connection and laplacian}, for all $t\in [0,t_0-t_1]$
\[
\Delta_{\vb*h_{(c_1, \phi_{t+t_0}^*{\overline{c_2}})}} u_t=\Delta_{\vb*h_{(c_1, \phi_{t_0}^*{\overline{c_2}})}}\left(\veee_{Z_{t}} u\right)=\veee_{Z_{t}}\left(\Delta_{\vb*h_{(c_1, \phi_{t_0}^*({\overline{c_2}}))}} u\right)= \lambda \veee_{Z_{t}} u= \lambda u_t,
\]
hence $[t_0,t_1]\subset A^{\lambda}$, contradicting the maximality of $t_0$.
\end{proof}

For the case $\lambda=2$, we can say more. Recall we're setting $L_h=L_{h,2}.$

\begin{prop}\label{perturb}
   Let $h=\hpair$ be a smooth Bers metric and $q$ a small-enough $c_1$-holomorphic quadratic differential such that $h+q$ is a Bers metric.
   Then, $L_{h}$ is an isomorphism if and only if $L_{h+q}$ is an isomorphism.
\end{prop}
\begin{proof}
Notice that it is sufficient to prove that if $L_{h}$ is an isomorphism, then $L_{h+q}$ is an isomorphism. Indeed, the converse implication follows by seeing $h= (h+q)-q$.

   Observe that there exists $\eps=\eps(h, q)$ such that, for all smooth $\varrho\colon S\to \mathbb C^*$ for which $\|1-\varrho\|_{\infty}<\eps$, $\varrho h+q$ is a positive complex metric. 
In fact, one can see explicitly from the computation of Lemma 6.1 in \cite{ElE} that the condition for $\varrho h +q$ to be positive is 
\[
\left|\overline{\mu}+ \frac 1 \varrho \frac{q(\partial_z, \partial_z)}{ h(\partial_z, \partial_z)}\right|<1,
\]
where $\mu=\frac{\partial_{\overline z} w}{\partial_z \overline w }$, and $z$ and $\overline w$ are local holomorphic coordinates for $c_1$ and $\overline{c_2}$ respectively.

We prove that the image of $L_{h+q}$ contains $C^{\infty}(S,\overline \C)$. It will follow from Fredholmness  that $L_{h+q}$ is surjective, and hence an isomorphism because it has index zero.

Since $L_h$ is an isomorphism, for all $f\in  C^\infty (S,\C)$ there exists $\dot u$ such that $L_h(\dot u)=f$. By Proposition \ref{nonlinearreg}, $\dot u\in C^\infty(S, \C)$.
Now, assume $t$ is small enough so that $(1+t\dot u) h+q$ is a positive complex metric, and set $h'_t$ to be the Bers metric in the conformal class of $(1+t\dot u) h +q$, so there exists a path of functions $\sigma_t,$ with $\sigma_0=1$, such that
\[
(1+t\dot u) h + q= \sigma_t h'_t.
\]
One can see explicitly, with computations as in \cite[Remark 3.9]{ElE}, that $\vb*{c_+}(h'_t)= c_1$ and that the Beltrami form of $\vb* {c_-}(h'_t)$ with respect to $c_1$ is $ \mu+\frac{\overline{q(\partial_z, \partial_z)}}{\ \overline{(1+t\dot u)h(\partial_z, \partial_z)}\ }$, so $\vb* {c_-}(h'_t)$ depends smoothly on $t$, hence, by Proposition \ref{appendixholomorphicity}, $h'_t$ depends smoothly on $t$ and so does $\sigma_t$. 

By Proposition \ref{prop: Kg and Kg+q}, we have that $\mathrm K_{(1+t\dot u) h}=\mathrm K_{ \sigma_t h'_t}$, hence, using  Equation \eqref{eq: conf curvature},
\[
\frac{1}{(1+t\dot u)} \left(-1-\frac 12 \Delta_{h} (\log  (1+t\dot u))\right)= \frac{1}{\sigma_t} \left(-1-\frac 12 \Delta_{h'_t} (\log  \sigma_t)\right), \ 
\]
and, taking the time derivative at $0,$
\begin{equation}
    \label{eq: transport Delta-2}
\Delta_{h}\dot u -2 \dot u = \Delta_{h+q}\dot \sigma -2 \dot \sigma, 
\end{equation}
where we used that the coefficients of the operators $\Delta_{h'_t}$ change smoothly with respect to $t$, so that $\left(\frac {d}{dt}_{|0}\Delta_{h_t}\right) (0)=0.$ As a result, we get that $f=\Delta_{h+q}\dot \sigma -2 \dot \sigma$, hence $\dot \sigma$ is smooth by Proposition \ref{nonlinearreg}, proving the surjectivity of $L_{h+q}$.

\end{proof}

\begin{lem} 
\label{lemma: h+q co real analytic}
Let $h=\hpair$ be a Bers metric and $q$ a small-enough $c_1$-holomorphic quadratic differential such that $h+q=\vb*h_{(c_1, \overline{c_2}')}$ is a Bers metric. If $c_1, \overline{c_2}$ are co-real analytic, then $c_1, \overline{c_2}'$ are co-real analytic.
\end{lem}
\begin{proof}
    As usual, denote $h=\lambda dz d\overline w$, $h+q= \lambda'dz d\overline \eta$ and $q=\phee dz^2$ in local coordinates $z, \overline w, \overline \eta$ for $c_1, \overline{c_2},\overline{c_2}'$ respectively.

    From $\lambda dz d\overline w+q= \lambda'dz d\overline \eta$ we get that 
    \[
    \frac{\partial_z \overline \eta}{\partial_{\overline z}\overline \eta} \frac{dz}{\overline{dz}} = \frac{\lambda \partial_z \overline w + \phee}{\lambda \partial_{\overline z}\overline w}\frac{dz}{d\overline z},
    \]
    which is real analytic.
\end{proof}

Theorem D now follows from the following statement. 
\begin{thm}
\label{thm F extended}
    The set
    \[
    \mathcal B^{reg}=\{\ccpair \in \CSCS\ |\ L_{\hpair}\colon W^{k+2, p}(S)\to W^{k,p}(S) \text{ is invertible}\}
    \]
    is an open dense subset of $\mathcal C(S)\times\mathcal C(\overline S)$ which contains all the pairs of co-real analytic complex structures. The intersection of $\mathcal B^{reg}$ with any fiber of the projection $\CSCS\to \TSTS$ is connected.
\end{thm}
\begin{proof}
We first prove that $\mathcal B^{reg}$ contains all the co-real analytic pairs.

Fix $[c_1]\in \mathcal T(S)$. Define $Y\subseteq \{[c_1]\}\times \mathcal T(\overline S)$ to be the subset of elements admitting a co-real analytic pair of representatives $\ccpair$ such that $L_{\hpair}$ is invertible. Conversely, define $N\subseteq \{[c_1]\}\times \mathcal T(\overline S)$ to be the subset of elements admitting a co-real analytic pair of representatives $\ccpair$ such that $L_{\hpair}$ is \emph{not} invertible. 

By Proposition \ref{isotopypreserve}, $Y$ and $N$ are complementary. $Y$ contains $([c_1], [\overline{c_1}])$: in fact, for the Riemannian hyperbolic metric $h=\vb*h_{(c_1, \overline {c_1})}$, $L_h$ is the complexification of a real operator $W^{k+2, p}(S, \mathbb R)\to W^{k, p}(S, \mathbb R)$, which is invertible by a simple maximum-principle argument.

 Moreover, both $Y$ and $N$ are open: by Theorem \ref{thm: ElE deform}, Proposition \ref{perturb}, and Lemma \ref{lemma: h+q co real analytic}, if $c_1$ and $\overline{c_2}$ are co-real analytic and $([c_1], [\overline{c_2}])$ lies in $Y$ (resp $N$), then one can parametrize a local neighborhood of it in $Y$ (resp. $N$)  whose elements are of the form $([c_1], [\overline{c_2}^q])$ where $\hpair +q = :\vb*h_{(c_1, \overline{c_2}^q)}$ and $q$ is a small-enough $c_1$-holomorphic quadratic differential. We conclude that $Y= \{[c_1]\}\times \mathcal T(\overline S)$ so, by the generality of $c_1$, we get that all the co-real analytic pairs lie in $\mathcal B^{reg}$.

By Lemma \ref{openness}, $\mathcal B^{reg}$ is open.

Finally, let $F$ be a fiber of $\CSCS\to \TSTS$. 
 By looking at any smooth trivialization of $\CSCS\cong \mathrm{Diff}_0(S)\times \mathrm{Diff}_0(S)\times \TSTS$, the density of $\mathrm{Diff}^{\omega}_0(S)$ in $\mathrm{Diff}_0(S)$ implies that any element of $F\cap \mathcal B^{reg}$ admits a connected neighborhood inside it containing a co-real analytic pair; the statement follows from the fact that co-real analytic pairs in $F\cap \mathcal B^{reg}$ determine a connected subset.
\end{proof}

\section{The holomorphic extension of the Labourie-Loftin parametrization}

\subsection{The Labourie-Loftin parametrization}\label{sec: LL parametrization}
We recall the definition of Hitchin representations from the introduction. Firstly, embedding $\iota: \mathrm{PSL}(2,\R)\to \mathrm{PSL}(3,\R)=\mathrm{SL}(3,\R)$ using an irreducible representation to $\textrm{SL}(3,\R)$, we can turn a Fuchsian representation $\rho_0:\pi_1(S)\to\mathrm{PSL}(2,\R)$ into a representation to $\mathrm{SL}(3,\R)$. A representation $\rho:\pi_1(S)\to \mathrm{SL}(3,\R)$ is \textbf{Hitchin} if it can be deformed to a representation of the form $\iota\circ \rho_0$ through representations to $\mathrm{SL}(3,\R)$. The space of Hitchin representations inside $\chi(\pi_1(S),\mathrm{SL}(3,\R))$ determines a connected component, called a Hitchin component and labelled $\mathrm{Hit}_3(S).$ 

Recall that a domain $D\subset \mathbb{RP}^2$ is convex if can be seen as a convex domain inside an affine chart. By work of Goldman \cite{Gol} and Choi-Goldman \cite{ChG}, every Hitchin representation $\rho:\pi_1(S)\to \mathrm{SL}(3,\R)$ acts properly discontinuously on a bounded convex domain $D\subset \mathbb{RP}^2$, identifying $S$ with $\faktor{D}{\rho(\pi_1(S))}.$ Labourie and Loftin both observed that, by a result conjectured by Calabi \cite{Calabi} and attributed to Cheng-Yau \cite{CY1} \cite{CY2} and Calabi-Nirenberg (unpublished), there exists a unique $\rho$-equivariant hyperbolic affine sphere centered at zero and asymptotic to the positive cone of $D$ in $\R^3.$

Labourie and Loftin both prove that the Gauss equation for real affine spheres 
 $$\Delta_h u = e^{2u}-\frac{1}{4}h(\Q,\overline{\Q})e^{-4u}-1$$
 has a unique (real) solution, and thus one can parametrize a real analytic moduli space of real affine spheres by the bundle $\mathcal{M}_3(S)$ from the introduction (which we define precisely below). By taking the holonomies of affine spheres, one gets the Labourie-Loftin parametrization $$\mathcal{L}:\mathcal{M}_3(S)\xrightarrow{\sim} \mathrm{Hit}_3(S).$$
 In the same spirit, we now return to complex affine spheres and our Equation (\ref{eq: PDE of Gauss eq}) from Section 3 in order to build an open subset $\Omega\subset \mathcal{M}_3(S)\times \mathcal{M}_3(S)$ containing the diagonal and the map ${\mathcal{L}}_{\C}:\Omega\to \chi(\pi(S),\mathrm{SL}(3,\C))$ that extends $\mathcal{L}$.

\subsection{Global analysis and the Gauss equation}\label{sec: BM and Gauss}
In this subsection we parametrize a class of $C^\infty$ positive hyperbolic affine spheres as a complex Fréchet manifold. In order to make the global analysis more convenient, we first work over a family of Banach manifolds, and then we take a projective limit.

Recall $\mathcal{C}^l(S),\mathcal{C}(S),$ etc., defined in Section 2. Let $\mathcal{CD}^l(S)$ be the bundle over $\mathcal C^l(S)$ whose fiber over $c\in \mathcal C^l(S)$ is the space $H^0(S, \mathcal K_c^3)$ of $c-$holomorphic cubic differentials. Similarly, we have the analogous bundle $\mathcal{CD}(S)\to \mathcal{C}(S)$. It follows from work of Bers \cite[Theorem 2]{Bhd} that $\mathcal{CD}^l(S)$ (resp. $\mathcal{CD}(S)$) is a holomorphic vector bundle over $\mathcal{C}^l(S)$ (resp. $\mathcal{C}(S)$). We explain how to see these bundles precisely. Bers gave a biholomorphic embedding from $\mathcal{T}(S)$ onto a bounded domain $\mathcal{B}$ in $\C^{3g-3}$. In \cite[Appendix A]{Nolte}, Nolte explains carefully how to use \cite{Bhd} to see our bundle $\mathcal M_3(S)$ from the introduction as a holomorphic vector bundle on $\mathcal{B}$ (and thus on $\mathcal{T}(S)$). For every $l,$ Bers' construction defines a holomorphic map from $\mathcal{C}^l(S)\to \mathcal{B}$. Topologically pulling back the vector bundle structure of $\mathcal{M}_3(S)\to \mathcal{B}$ gives the bundle $\mathcal{CD}^l(S).$ Picking trivializations of the form $U\times V,$ with $U\in \mathcal{C}^l(S)$ and $V$ a finite dimensional complex vector space, we can direct sum the norm on $\mathcal{C}^l(S)$ with any norm on $V$. Patching things together, $\mathcal{CD}^l(S)$ obtains the structure of a Banach manifold and a holomorphic vector bundle over $\mathcal{C}^l(S).$ The procedure is the same for $\mathcal{CD}(S).$ Note that $\mathrm{Diff}_0(S)$ acts on $\mathcal{CD}(S)$ via the ordinary action on the base $\mathcal{C}(S)$ and by pull-back on the fibers. Using \cite{EE} to split trivializations for $\mathcal{CD}(S)$ as $U\times V =(U_1\times U_2)\times V,$ where $U_1\subset \mathcal{T}(S)$ and $U_2\subset \textrm{Diff}_0(S),$ we see that the quotient under this action is, as expected, $\mathcal{M}_3(S)$, and the projection map is holomorphic. We will denote by $\CD^l(\overline S)$, $\CD(\overline S)$, and $\mathcal M_3(\overline S)$ the analogous spaces for the oppositely oriented surface $\overline S$. We point out in passing that these constructions are not restricted to cubic differentials, but work just the same for differentials of any degree.

For $k>2$ and $l>k+2$, we consider the function
\begin{equation}
\label{Gauss equation section 6}
    \begin{split}
        G\colon \mathcal{CD}^l(S)\times \mathcal{CD}^l&(\overline S)\times W^{k+2, 2}(S) \to W^{k, 2}(S)
        \\
        G(c_1, \overline{c_2}, \Q_1, \overline{\Q_2},u):&=\Delta_{\vb*h_{\ccpair}} u-e^{2u}+\frac 1 4 \vb*h_{\ccpair}(\Q_1,\overline{\Q_2})\cdot   e^{-4u}+1
    \end{split}
\end{equation}
whose zeroes, by Corollary \ref{cor: PDE of Gauss eq}, correspond to positive hyperbolic complex affine spheres with Blaschke metric of the form $g=e^{2u} \hpair$. Moreover, we define the space
\begin{equation}
        \label{eq: def of CAS}
    \mathrm{CAS}(S):=\{(c_1,\overline{c_2},\Q_1,\overline{\Q_2},u) \in \CD(S)\times \CD(\overline S)\times C^\infty(S,\C) : G(c_1,\overline{c_2},\Q_1,\overline{\Q_2},u)=0\},\ 
\end{equation}
with the induced subspace topology.

Note that by Proposition \ref{nonlinearreg}, if $c_1,\overline{c_2},\Q_1,\overline{\Q_2}$ are $C^{\infty}$, then any $W^{k+2,2}$ solution $u$ to $G(c_1,\overline{c_2},\Q_1,\overline{\Q_2},u)=0$ is actually $C^\infty$, so in the definition of $\CAS(S)$ one can replace the $C^\infty$ above by $W^{k+2,2}$.

 By Example \ref{ex: ras}, the space of real hyperbolic affine spheres corresponds to the subspace
\[
\mathrm{RAS}(S):= \{ (c, \overline c, \Q, \overline{\Q}, u)\in \mathrm{CAS}(S)\ |\ u \text{ is a real function} \} .
\]
 
The map   $G\colon \mathcal{CD}^l(S)\times \mathcal{CD}^l(\overline S)\times W^{k+2, 2}(S) \to W^{k, 2}(S)$ is holomorphic: by Hartogs' theorem on separate holomorphicity, we're permitted to check on each coordinate separately. Since holomorphicity in $u$ is clear, we are left to prove it in the direction of $\mathcal{CD}^l(S)\times \mathcal{CD}^l(\overline S)$. By Proposition \ref{appendixholomorphicity}, the association $\ccpair\to \hpair$ is holomorphic. It follows from the Koszul formula for the Levi-Civita connection that for fixed $u$, $(c_1,\overline{c_2})\mapsto \Delta_{\hpair}u$ is holomorphic as well. Similarly, $(c_1,\overline{c_2}, \Q_1, \overline{\Q_2})\mapsto \hpair(\Q_1,\overline{\Q_2})$, being a composition of holomorphic functions, is holomorphic.

Fix a point $\sigma=(c_1,\overline{c_2},\Q_1,\overline{\Q_2},u)\in \CAS(S)$ and consider the linearization $L_{\sigma}$ of $G$ at $\sigma$ in the direction of $u$, which is a linear map from $T_uW^{k+2,2}(S,\C)=W^{k+2,2}(S,\C)$ to $T_{0}W^{k,2}(S,\C)=W^{k,2}(S,\C)$, where the equal sign is an implicit identification that we will always make. The evaluation of the map $L_\sigma$ at a tangent vector $v\in W^{k+2,2}(S,\C)$ is obtained by computing 
$$L_\sigma(v)=\frac{d}{dt}\Big|_{t=0}G(c_1,\overline{c_2},\Q_1,\overline{\Q_2},u+tv),$$ which returns 
\begin{equation}
\label{eq: linearized GE}
L_\sigma(v)=\Delta_{\hpair} v-2ve^{2u}- ve^{-4u}\hpair(\Q_1,\overline{\Q_2}) .
\end{equation}
By Propositions \ref{nonlinearreg} and \ref{fredholmness}, $L_\sigma$ defines a Fredholm operator of index zero on each $W^{k+2,2},$ and it is an isomorphism on $C^\infty$ if and only if it defines an isomorphism for every $W^{k+2,2}.$

\begin{defn}
    We say that $\sigma\in \CAS(S)$ is \textbf{infinitesimally rigid} if $L_\sigma$ is an isomorphism. We denote the subspace of infinitesimally rigid elements by $\CAS^*(S)$.
\end{defn}

First of all, we observe that $\mathrm{CAS}^*(S)$ is non-empty and contains the real affine spheres.
\begin{prop}
    $\mathrm{RAS}(S)$ is contained in $\mathrm{CAS}^*(S)$.
\end{prop}
\begin{proof}
    For all $\sigma\in \mathrm{RAS}(S)$, the linear map  $L_{\sigma}$ restricts to a map $W^{k+2,2}(S, \mathbb R)\to W^{k, 2}(S, \mathbb R)$, and Labourie showed in his study of real affine spheres that such a restriction is bijective \cite[Lemma 4.1.2]{Lab2}. $L_\sigma$ is the complex bilinear extension of the map from $W^{k+2,2}(S, \mathbb R)\to W^{k, 2}(S, \mathbb R)$, and hence for any $v\in W^{k+2,2}(S)$, $L_{\sigma}(v)=L_{\sigma}(Re(v))+ i L_{\sigma}(Im(v))$, and it follows that $L_\sigma$ is injective on $W^{k, \infty}(S, \mathbb C)$. Since $L_{\sigma}$ is Fredholm of degree zero, it is bijective.
\end{proof}
Denote by \begin{align*}
    \hat\pi\colon \CAS^*(S)&\to \CD(S)\times\CD(\overline S)\\
    \sigma=(c_1, \overline{c_2}, \Q_1, \overline{\Q_2}, u)&\mapsto (c_1,\overline{c_2},\Q_1 \overline{\Q_2})
\end{align*}
the natural projection. By transversality theory, we have the following. 

\begin{prop}
\label{prop: CAS* local biholo}
    The space $\CAS^*(S)$ is a complex Fr{\'e}chet manifold and $ \hat\pi\colon \CAS^*(S)\to \CD(S)\times\CD(\overline S)$ is a local biholomorphism.
\end{prop}
\begin{proof}
For $k>2$ and $l>k+2$, define 
$\CAS^*_{l,k}(S)$ to be the subset of elements $\sigma\in \CD^l(S)\times \CD^l(\overline S)\times W^{k+2,2}(S,\C)$ such that $G(\sigma)=0$ and its linearization $L_\sigma: W^{k+2,2}(S,\C)\to W^{k,2}(S,\C)$ is an isomorphism. Also denote by $\hat \pi\colon \CAS_{l,k}^*(S)\to \CD^l(S)\times\CD^l(\overline S) $ the projection that extends the projection from the smooth setting. The map \begin{align*}
\Phi\colon \CD^l(S)\times \CD^l(\overline S)\times W^{k+2,2}(S,\C)&\to \CD^l(S)\times \CD^l(\overline S) \times W^{k,2}(S,\C)\\
(c_1, \overline{c_2}, \Q_1, \overline{\Q_2}, u)&\mapsto (c_1, \overline{c_2}, \Q_1,  \overline{\Q_2}, G(c_1,  \overline{c_2}, \Q_1, \overline{\Q_2}, u) )
\end{align*} is holomorphic and maps onto the submanifold of the target described by $\mathcal{CD}^l(S)\times \mathcal{CD}^l(\overline{S})\times \{0\}$. By assumption, for all $\sigma\in \CAS_{l,k}^*(S)$, $L_\sigma$ is an isomorphism, and hence so is $d_\sigma \Phi$. By the Transversality Theorem for Complex Banach manifolds, $\mathrm{CAS}_{l,k}^*(S)$ is a complex Banach submanifold of $\CD^l(S)\times \CD^l(\overline S)\times W^{k+2,2}(S,\C)$.

Since $G|_{\CAS_{k,l}^*(S)}\equiv 0$, $d_\sigma\Phi= (d_\sigma \hat\pi, 0)$, and hence $d_\sigma \hat\pi\colon \CAS_{l,k}^*(S)\to  \CD^l(S)\times \CD^l(\overline S)$ is bijective, which implies that $\hat\pi$ is a local biholomorphism.

From now on, set $l= k+3$, and $\mathrm{CAS}_k^*(S)=\mathrm{CAS}_{k,k+3}^*$. Then, $\mathrm{CAS}_k^*(S)$  parametrizes complex affine spheres with lower regularity, and the intersection is $\cap_{k'\geq k}\mathrm{CAS}_k^*(S)$ is $\mathrm{CAS}^*(S).$ To put a Fréchet structure on $\mathrm{CAS}^*(S),$ fix an open cover of $\CAS_{k}^*(S)$ by open subsets $U_k$ over which $\hat \pi$ projects biholomorphically to open subsets $V_k\subset \CD^{k}(S)\times \CD^{k}(\overline S)$,  thus giving a compatible atlas for the complex Banach structure on $\CAS_{k}^*(S)$. For all $k'>k$, the intersection with $\CAS^*_{k'}(S)$ gives complex charts $U_{k'}= U_k\cap \CAS^*_{k'}\to V_{k'}= V_k \cap (\CD^{k'+3}(S)\times \CD^{k'+3}(\overline S))$. Finally, the restriction to 
\[
U_k\cap \CAS^*(S)= \bigcap_{k'\ge k} U_{k'} \to  \bigcap_{k'\ge k} V_{k'}= V_k\cap (\CD(S)\times \CD(\overline S))
\]
defines Frechét charts for $\CAS^*(S)$ for which $\hat \pi$ is holomorphic. 

The construction above can be formalized through a projective limit construction for complex Banach manifolds. As explained in \cite{Dodson}, in the smooth category, one can take a projective limit of Banach manifolds and the limit is a Fréchet manifold, and moreover a projective family of locally diffeomorphic mappings between projective systems gives rise to a local diffeomorphism between the two limiting Fréchet manifolds (see Propositions 3.13, 3.18, and Corollary 3.29 in \cite{Dodson}). The proofs go through verbatim in the holomorphic category, although we don't need to refer to \cite{Dodson} for the whole thing since the construction above is explicit.
\end{proof}

\begin{remark}\label{rem: character variety}
    We need to make a brief digression on the character variety and holomorphicity, which concerns Theorem \textrm{A}, Proposition \ref{prop: hol of CAS*} and Theorem \ref{thm: holonomy from the quotient of CAS}. 
 The space $\mathrm{Hom}(\pi_1(S),\mathrm{SL}(3,\C))$ embeds inside $\mathrm{SL}(3,\C)^{2g}$ as a complex manifold, and the $\mathrm{SL}(3,\C)$-action by conjugation is by biholomorphisms. We set $\chi(\pi_1(S), \SL(3,\C))$ to be the Hausdorffification of the quotient with respect to this action. In this paper, we take a naive standpoint and say that a map $F: M\to \chi(\pi_1(S),\mathrm{SL}(3,\C))$, with $M$ a complex (Fréchet) manifold, is holomorphic if, around every point of $M$, it locally lifts to holomorphic maps $M\supset U\to \mathrm{Hom}(\pi_1(S),\mathrm{SL}(3,\C))$. As a result, if $N\to M$ is a holomorphic map between complex (Fréchet) manifolds, then the composition $N\to \chi(\pi_1(S),\mathrm{SL}(3,\C))$ is holomorphic.

The quotient of the subset of irreducible representations in $\textrm{Hom}(\pi_1(S),\mathrm{SL}(3,\C))$ carries the structure of a complex orbifold \cite[Proposition 49]{Sik} (the smooth locus consists of classes of representations that are also simple, which means the centralizer is the center of $\textrm{SL}(3,\C)$ \cite[Corollary 50]{Sik}). If $F$ has image in this subset, then it is holomorphic as a map between complex orbifolds.

We don't know whether the holonomy of a positive hyperbolic complex affine sphere is irreducible. Since the irreducible condition is open, we certainly have it in a neighbourhood of the diagonal. In a future work, we will confirm that every point in the domain $\Omega$ of Theorem A corresponds to an irreducible representation.

.
 \end{remark}

\begin{prop}
\label{prop: hol of CAS*}
    The holonomy map $hol
\colon\mathrm{CAS}^*(S)\to \chi(\pi_1(S), \mathrm{SL}(3, \C))$ is holomorphic.
\end{prop}

\begin{proof}
The complex affine sphere $\sigma$ parametrized by $(c_1, \overline{c_2}, \Q_1, \overline{\Q_2}, u)\in \mathrm{CAS}^*(S)$ 
gives rise to structural data $g=e^{2u}\hpair$, $\nabla= -\frac 1 2 g^{-1}(\Q_1+\overline{\Q_2})+\nabla^g$, $\mathrm S=-\mathrm{id}$, and $\tau=0$.  Going through Theorem \ref{thm: integration GC}, out of this data we obtain a rank $3$ complex vector bundle over $S$ that carries a flat $\mathrm{SL}(3,\C)$-connection $D^\sigma$ that satisfies the structural equation given at the beginning of the proof of Theorem \ref{thm: integrating GE for CAS}.

Using Proposition \ref{appendixholomorphicity}, the correspondence $(c_1, \overline{c_2}, \Q_1, \overline{\Q_2}, u)\mapsto e^{2u}\hpair$ is holomorphic, and it is routinely checked that $(c_1, \overline{c_2}, \Q_1, \overline{\Q_2}, u)\mapsto -\frac 1 2 g^{-1}(\Q_1+\overline{\Q_2})+\nabla^g$ is holomorphic as well.  Using the formula from Theorem \ref{thm: integrating GE for CAS}, it is clear that $\sigma=(c_1, \overline{c_2}, \Q_1, \overline{\Q_2}, u)\mapsto D_\sigma$ defines a holomorphic map to the complex affine space that models flat $\mathrm{SL}(3,\C)$ connections.

It is then well known that the holonomy map from the space of flat $\mathrm{SL}(3,\C)$-connections to the character variety is holomorphic \cite[\S 9]{Gun}. Intuitively, if $D$ is a holomorphically varying flat connection, the $D$-parallel section passing by a fixed element in the bundle varies holomorphically as well: by fixing a point $p$ on $\widetilde S$ and a basis on one fiber, one gets that the holonomy homomorphism $\pi_1(S)\to \SL(3,\C)$ induced by the corresponding basis of $D$-parallel sections varies holomorphically inside $\mathrm{Hom}(\pi_1(S), \SL(3,\C))$, hence $hol$ is holomorphic in the sense of Remark \ref{rem: character variety}.
\end{proof}

\subsection{The space $\Sph(S)$}\label{sec: CAS space}

We are now ready to define a finite-dimensional quotient of $\CAS^*(S)$ that keeps track of the geometric information of the corresponding complex affine spheres.

For the rest of the section, we will take local trivializations of the bundle $\mathcal C(S)\to \mathcal T(S)$ obtained from local holomorphic sections $s\colon \mathcal T(S)\supset U\to \mathcal C(S)$ such that each element in the image is a real analytic complex structure compatible with a fixed real analytic structure on $S$. {We show a way to construct such a section in Remark \ref{rmk: nice sections of C(S)} below.} In this way, we can trivialize open subsets of $\mathcal C(S)$ as $
U\times \mathrm{Diff}_0(S)$, with $U$ projecting injectively to $\mathcal T(S)$, in such a way that the projection $U\times \mathrm{Diff}_0(S)\to U$ is holomorphic (recall Theorem \ref{thm: earle eells}), and the trivialization restricts to  $\mathcal C^{\omega}(S)\cap (U\times \mathrm{Diff}_0(S))=U\times \mathrm{Diff}_0^{\omega}(S) $. By trivializing $\mathcal{CD}(S)\to \mathcal{C}(S)$ over such an open subset, we also obtain local trivializations for the bundle $\CD(S)\to \mathcal M_3(S)$.

\begin{remark}
    \label{rmk: nice sections of C(S)}
 We can construct such sections using Theorem \ref{thm: ElE deform}: starting from a real analytic representative $c$, let $h=\vb*h_{(c,c)}$ be the (real analytic) hyperbolic metric in its conformal class, then define a local section by taking complex structures of the form $\vb*{c_+}(h+\overline q)$ with $q\in \mathcal K_c^2(S)$ and $\|q\|_h<\frac 12$ (the corresponding Beltrami form for $\vb*{c_+}(h+\overline q)$ is in fact $\frac {\overline q}h$, see \cite[Remark 2.9]{ElE}). 
\end{remark}

Now, define the map 
\begin{align*}    
[\hat \pi]\colon \CAS^*(S)&\to \mathcal M_3(S)\times\mathcal M_3(\overline S)\\
(c_1, \overline{c_2}, \Q_1, \overline{\Q_2}, u)&\mapsto\left( [c_1, \Q_1], [\overline{c_2}, \overline{\Q_2}]\right). \ 
\end{align*}

{Continuous deformations within the fibers of $[\hat \pi]$ are best understood using the  transport formalism.}

\begin{prop}
\label{prop: fibers are transports}
    Let $\sigma_t=(c_1^t, \overline{c_2}^t, \Q_1^t,  \overline{\Q_2}^t, u_t)\in \CAS^*(S)$ be a $C^1$ path contained in a fiber of $[\hat \pi]$. Then, for all $t$, there exists $Z_t\in \Gamma(\C TS)$ such that
    \[
    \dot \sigma_t= (\mathscr L_{Z_t} c^t_1,\mathscr L_{Z_t} \overline{c_2}^t, \mathscr L_{Z_t} \Q_1^t, \mathscr L_{Z_t} \overline{\Q_2}^t, \mathscr L_{Z_t} u_t). 
    \]
\end{prop}
\begin{proof}
  To simplify the notation, we do the computation in $t=0$, then the statement follows for the whole path by reparametrizing. Denote $\sigma_0=(c_1, \overline{c_2}, \Q_1, \overline{\Q_2}, u)$. Since $\sigma_t$ is contained in a fiber of $[\hat \pi]$, there exist $\phi_1^t, \phi_2^t\in \mathrm{Diff}_0(S)$ such that 
  \[
  (c_1^t, \Q_1^t)=((\phi_1^t)^*c_1, (\phi_1^t)^*\Q_1), \qquad (\overline{c_2}^t, \overline{\Q_2^t})=((\phi_2^t)^*\overline{c_2}, (\phi_2^t)^*\overline{\Q_2}). \ 
  \]
  By Proposition \ref{prop: tangent to C(S)C(S)} and Remark \ref{rmk: complex derivative of a complex structure}, there exist $Z\in \Gamma(\C TS)$ such that 
$$\frac d{dt}\Big|_0 ((\phi_1^t)^*c_1, (\phi_2^t)^*\overline{c_2}, (\phi_1^t)^*\Q_1, (\phi_2^t)^*\overline{\Q_2})= (\mathscr L_{Z} c_1,\mathscr L_{Z} \overline{c_2}, \mathscr L_{Z} \Q_1, \mathscr L_{Z} \overline{\Q_2}). $$
Applying both $\frac d {dt}$ and $\mathscr L_Z$ to Gauss equation $G(\sigma_t)=0$ in $t=0$, by Proposition \ref{prop: transport connection and laplacian} one gets that 
\begin{align*}
0=&\Delta_{\hpair} \left(\frac d {dt}\Big|_0 u- \mathscr L_ Z u\right)- 2 e^{2u}\left(\frac d {dt}\Big|_0 u- \mathscr L_ Z u\right) - e^{-4u} {\hpair} ( \Q_1 ,\overline{\Q_2}) \left(\frac d {dt}\Big|_0 u- \mathscr L_ Z u\right)  \\
&=L_\sigma\left(\frac d {dt}\Big|_0 u- \mathscr L_ Z u\right). \ 
\end{align*}
Since $\sigma\in \CAS^*(S)$, $L_\sigma$ is injective, hence $\frac d {dt}\Big|_0 u= \mathscr L_ Z u$.
\end{proof}

The last proposition motivates the definition of the following quotient for $\CAS^*(S)$.

\begin{defn}
Define the equivalence relation $\sim$ on $\mathrm{CAS}^*(S)$ by $\sigma\sim \sigma'$ if and only if $\sigma$ and $\sigma'$ lie in the same connected component of a fiber of $[\hat \pi]$.

Define $$\Sph (S):= \faktor{\CAS^*(S)}{\sim}$$ as the quotient by $\sim$. We denote the quotient map by $\mathrm p_{\sim}\colon \CAS^*(S)\to \Sph(S)$.
\end{defn}

By construction, $\hat \pi$ descends to a map
\begin{align*}
    \pi\colon \Sph (S)&\to \mathcal M_3(S)\times \mathcal M_3(\overline S)\\
    [\sigma]&\mapsto [\hat \pi(\sigma)].
\end{align*}

\begin{prop}
    \label{prop: finite quotient of CAS}
$\Sph(S)$ has a structure of finite-dimensional complex manifold such that the quotient map $\mathrm p_\sim\colon \CAS^*(S)\to \Sph(S)$ is holomorphic, and the map $\pi\colon \Sph(S)\to  \mathcal M_3(S)\times \mathcal M_3(\overline S)$ is a local biholomorphism.
\end{prop}
\begin{proof}
Consider local trivializations of $\CD(S)\to \mathcal M_3(S)$ as in Remark \ref{rmk: nice sections of C(S)}.

    By Proposition \ref{prop: CAS* local biholo}, for all $\sigma\in \CAS^*(S)$, we can restrict $\hat \pi$ to a biholomorphism in a neighborhood of $\sigma$, namely  \[\hat \pi\colon {\hat U_\sigma}\xrightarrow{\sim}  A_1\times B_1\times A_2\times B_2\subset \CD(S)\times\CD(\overline S), \]  where $\hat U_{\sigma}$ is a connected neighborhood of $\sigma$, while $A_1\subset \mathcal M_3(S)$, $A_2\subset \mathcal M_3(\overline S)$, and $B_1, B_2\subset \mathrm{Diff}_0(S)$ are open subsets. Note that splitting $A_1\times B_1\times A_2\times B_2$ is not holomorphic, but the projections to $A_1$ and $A_2$ are holomorphic.

   Denoting with $\mathrm p_\sim\colon  \CAS^*(S)\to \Sph(S)$ the quotient map, we therefore have the following commutative diagram:
    \begin{equation}
    \label{eq: diagram quotient CAS}
    \begin{tikzcd}
    \hat U_{\sigma} \arrow[r, "\hat \pi \quad \sim"] \arrow[d, "\mathrm p_{\sim}"] \arrow[rd, , "{[\hat \pi]}" ] &A_1\times B_1\times A_2\times B_2 \arrow[d, "\text{projection}"]\\
        U_{\sigma}:=\mathrm p_\sim(\hat U_{\sigma}) \arrow[r, "\pi\quad \sim "] &A_1 \times A_2
    \end{tikzcd}
    \end{equation}
It is easy to check that the collection of subsets $\{U_\sigma\}_{\sigma\in \CAS^*(S)}$ of $\Sph(S)$ defines a base for a topology on $\Sph(S)$. Moreover, using the one-to-one restrictions of $\pi$ to each $U_\sigma$ and the fact that the projection is holomorphic, we get a holomorphic atlas on $\Sph(S)$ by pulling back the complex structure on $\mathcal M_3 (S)\times \mathcal M_3(\overline S)$: by the diagram \eqref{eq: diagram quotient CAS}, $\pi$ is a local biholomorphism and $\mathrm p_{\sim}$ is holomorphic.
\end{proof}

The following theorem shows that $\Sph(S)$ preserves the information about the holonomy. 

\begin{thm}
\label{thm: holonomy from the quotient of CAS}
The holonomy map 
\begin{align*}
    hol\colon \Sph(S)&\to \chi(\pi_1(S), \SL(3,\C))\\
    [\sigma]&\mapsto hol(\sigma)
\end{align*}
is well-defined and holomorphic. 
\end{thm}
\begin{proof}
If we prove that the holonomy map descends in a well-defined way to $\Sph(S)$, then it is holomorphic for free: by taking local charts for $\mathrm p_{\sim}$ as in the diagram \eqref{eq: diagram quotient CAS}, one has local holomorphic inverses of $\mathrm p_{\sim}$ in the form $([c_1, \Q_1], [\overline{c_2}, \Q_2])\mapsto ( [c_1, \Q_1], \phi_0, [\overline{c_2}, \Q_2], \phi_0')$ around each point in $\Sph(S)$. The composition with the holonomy map $\CAS^*(S)\to \chi(\pi_1(S), \mathrm {SL}(3,\C))$, which is holomorphic by Proposition \ref{prop: hol of CAS*}, gives precisely the map $hol$ in the statement of the theorem.

We are left to prove that the holonomy is constant along any $C^1$ path $(c_1^t, \overline{c_2}^t, \Q_1^t, \overline{\Q_2}^t, u_t)$.

In order to do this, we need to look at the construction of the corresponding flat connection as in the proof of Theorem \ref{thm: integration GC}. Let $N\colon \C\times \widetilde S\to \widetilde S$ be the trivial bundle with generator $\hat \xi\equiv 1$ and let $E=\C T\widetilde S\oplus N$. Then, the data $g_t=e^{2u_t}\vb*h(c_1^t, \overline{c_2}^t)$ and $C_t=\Q_1^t+\overline{\Q_2}^t$above define a flat connection 
   \[
   \begin{cases}
        \hat D^t_{\mathsf X} Y&=\nabla^t_{\mathsf X}Y+ g_t(\mathsf X,Y) \hat \xi \\
        \hat D^t_{\mathsf X} \hat \xi &= \mathsf X \\
         \hat D^t_{\mathsf X} (i\hat \xi) &= i \mathsf X,
    \end{cases}
\]
where $\nabla^t$ is the affine connection on $\C TS$ defined by $C_t=: g_t(\nabla^t- \nabla^{g_t}, \cdot)$. By Proposition \ref{prop: fibers are transports}, there exists a path of complex vector fields $Z_t$ such that $\dot C_t= \mathscr L_{Z_t}C_t$ and $\dot g_t= \mathscr L_{Z_t}g_t$. One can immediately see that, as a consequence, $\mathscr L_{Z_t}\nabla^t= \dot \nabla^t$, where we denote $(\mathscr L_{Z_t}\nabla^t)_X Y= \mathscr L_{Z_t} (\nabla^t_X Y)- \nabla^t_{\mathscr L_{Z_t}X} Y- \nabla^t_X (\mathscr L_{Z_t} Y)$.

For a generic section $s=X+ b \hat \xi$ of the bundle $E$, denote ${\mathscr L}_{Z_t} s={\mathscr L}_{Z_t}(X)+ ({\mathscr L}_{Z_t} b)\hat \xi$. By the remarks above, we can see that $\frac d {dt}\hat D^t=\mathscr L_{Z_t} \hat D^t$, where
$(\mathscr L_{Z_t} \hat D)_X s= \mathscr L_{Z_t}(\hat D_X s)-\hat D^t_{\mathscr L_{Z_t}X}s- \hat D^t_X \mathscr L_{Z_t}s $. In particular, observe that, if $s$ is $\hat D^t$-parallel, then $(\mathscr L_{Z_t} \hat D^t) s= \hat D^t( \mathscr L_{Z_t}s) $.

In order to study the holonomy, we now construct a useful basis of parallel sections. 
Let $x_0\in \widetilde S$, let $(e^1,e^2,e^3)$ be a $\C$-linear basis for $E_{x_0}$, and let $\mathring s_t^j$ be the unique $\hat D^t$-parallel section such that $\mathring{s}_t^j(x)=e^j$.

We observe that, for each $j$, there exists precisely one path of $\hat D^t$-parallel sections $s^j_t$ such that $\frac{d}{dt} s_t^j|_{x_0}= \mathscr L_{Z_t} s_t^j |_{x_0}$ and $s^j_0=e^j$: in fact, this is equivalent to solving the Cauchy problem
\[
\begin{cases}
    &\dot a_j(t)= \sum_{k} b_{jk}(t) a_k(t) \\
    &a_j(0)=0
\end{cases}\quad \text{where } s_t^j= \sum_k a^j_k(t) \mathring{s}_t^k, \text{ and }  \left(\mathscr L_{Y_t} \mathring{s}_t^j \right)|_{x_0}= \sum_k (b_{jk}(t)) e_k
\]
which has exactly one solution.
Moreover, we can see that $\frac{d}{dt}s^j_t=\mathscr L_{Z_t}s_t^j$ everywhere on $\widetilde S$: indeed,
\[
0= -(\mathscr L_{Z_t} \hat D_t- \frac d {dt} \hat D_t)(s^j_t)=  \frac d {dt}(\hat D_t s^j_t) + \hat D_t \left((\mathscr L_{Z_t}- \frac d {dt})(s^j_t) \right)= \hat D_t \left((\mathscr L_{Z_t}- \frac d {dt})(s^j_t) \right),
\]
so $(\mathscr L_{Z_t}- \frac d {dt})(s^j_t)$ is a $\hat D_t$-parallel section which vanishes in $x_0$, so it vanishes everywhere. 

Finally, we use the basis $(s^j_t)_j$ to prove that the holonomy of $\hat D^t$ is constant. The holonomy of $\hat D^t$, and of the corresponding complex affine sphere is given by the matrices $(m_t^{jk}(\gamma))_{jk}$ where $\gamma\in \pi_1(S)$ and 
\[
\gamma^* s^j_t= \sum_k m^{jk}_t(\gamma) s_t^k.
\]
Since $Z_t$ is $\pi_1(S)$-invariant we have $$\frac d {dt} (\gamma^* s^j_t)= \mathscr L_{Z_t}(\gamma^* s^j_t)= \sum_k m^{jk}_t(\gamma) \mathscr L_{Z_t}(s^k_t)= \sum_k m^{jk}_t(\gamma) \frac d {dt} s^k_t,\ $$ from which we conclude that $\frac d {dt}(m^{jk}_t(\gamma))=0$, hence the holonomy is constant.

\end{proof}

In conclusion, we have the following commutative diagram of holomorphic maps:
\[
\begin{tikzcd}
\CD(S)\times\CD(\overline S) \arrow[d, swap, "\substack{\mathrm{Diff}_0(S)\times \mathrm{Diff}_0(S) \\
\text{action}}"] & \arrow[l, "\hat \pi"] \CAS^*(S) \arrow[rd, "hol"] \arrow[d, "\mathrm p_\sim"] &\\
 \mathcal M_3(S)\times\mathcal M_3(\overline S) & \arrow[l, "\pi"] \Sph(S) \arrow[r, "hol"] & \chi(\pi_1(S), \SL(3,\C)).
\end{tikzcd}
\]

\subsection{The action by $\mathrm{MCG}(S)$ and the $\mathbb C^*$ twistor action}

We digress to discuss two actions on the spaces $\CAS(S)$, $\CAS^*(S)$, $\Sph(S)$, and $\mathcal{M}_3(S)\times \mathcal{M}_3(S)$ related to Theorem A. Recall the groups $\mathrm{Diff}_+(S),$ $\mathrm{Diff}_0(S),$ and $\mathrm{MCG}(S)$ from Section 2.

Given a positive hyperbolic complex affine sphere $\sigma\colon \widetilde S\to \mathbb C^3$ with Blaschke metric $g$ and Pick tensor $C$, for all $\phi\in \mathrm{Diff}_+(S)$ one gets that $\sigma\circ \phi$ is a positive hyperbolic complex affine sphere with Blaschke metric $\phi^*g$ and Pick tensor $\phi^*C$. Moreover, if $\sigma$ has holonomy $\rho$, then $\sigma\circ \phi$ has holonomy $\rho\circ \pi_1(\phi)$ where $\pi_1(\phi)\colon \pi_1(S)\to \pi_1(S)$ is the induced map in homotopy.

In the notation of the space $\CAS(S)$, $\mathrm{Diff}_+(S)$ acts (on the right) on $\CAS(S)$ through 
\[
\phi\cdot (c_1, \overline{c_2}, \Q_1, \overline{\Q_2}, u) = (\phi^*(c_1), \phi^*(\overline{c_2}), \phi^*(\Q_1), \phi^*(\overline{\Q_2}), \phi^*u).\ 
\]
In fact, $\phi$ defines an isometry between $\phi^* (\hpair)=\vb*h_{(\phi^*(c_1), \phi^*(\overline{c_2}))}$ and $\hpair$ conjugating the respective Levi-Civita connections, hence for all Bers metric $h$ 
$$
\phi^*( h(\Q_1, \overline{\Q_2}))= (\phi^*h)(\phi^*\Q_1, \phi^*\overline{\Q_2}) 
\qquad \text{and}\qquad
\phi^*(\Delta_{h} u)= \Delta_{\phi^*h} \phi^*u,
$$
so $G(\phi\cdot \sigma)=\phi^*(G(\sigma))=0$.

Observe that, for the same reason, $\phi^*(L_\sigma v)=L_{\phi\cdot \sigma}  (\phi^* v)$, so the action of $\mathrm{Diff}_+(S)$ restricts to $\CAS^*(S)$. Moreover, the action is holomorphic since it is a restriction to a complex Banach submanifold of a holomorphic action on $\CD(S)\times\CD(\overline S)\times W^{k,p}(S)$ (it is actually $\C$-linear on the last factor). Restricting to $\CD(S)\times\CD(\overline S)$, it is well known that this action descends to a holomorphic action of $\mathrm{MCG}(S)$ on $\mathcal{M}_3(S)\times \mathcal{M}_3(\overline{S}).$

\begin{prop}
The action of $\mathrm{Diff}_+(S)$ on $\CAS(S)$ descends to an action of $\mathrm{MCG}(S)$ on $\Sph(S)$ and the maps
\[\begin{tikzcd}[column sep=tiny]
& \Sph(S) \arrow[dl, swap, "\pi"] \arrow[dr, "hol"] & \\
\mathcal M_3(S)\times \mathcal M_3(\overline S) & & \chi(\pi_1(S), \SL(3,\C))
\end{tikzcd}\]
are $\mathrm{MCG}(S)$-equivariant.
\end{prop}
\begin{proof}
Since the $\mathrm{Diff}_0(S)$-orbit of an element in $\CAS^*(S)$ is path connected, it is contained in a fiber of the projection $[\hat \pi]\colon \CAS^*(S)\to \mathcal M_3(S)\times \mathcal M_3(\overline S)$, so it descends to an action of $\mathrm{MCG}(S)$ on $\Sph(S)$. The rest of the statement follows easily from the above remarks.
\end{proof}

The other action on $\CAS(S)$ that we want to highlight is the $\mathbb C^*$-action defined by 
\[
\zeta\cdot (c_1, \overline{c_2}, \Q_1, \overline{\Q_2}, u) = (c_1, \overline{c_2}, \zeta\Q_1, \frac 1 \zeta \overline{\Q_2}, u),
\]
which is easily well-defined, as a result of the fact that $h(\zeta \Q_1, \frac 1 \zeta \overline{\Q_2})= h( \Q_1, \overline{\Q_2})$, and holomorphic.
Similarly to above, this action restricts to an action on $\mathcal{CD}(S)\times \mathcal{CD}(\overline{S})$ that descends to $\mathcal{M}_3(S)\times \mathcal{M}_3(\overline{S}).$ We call each of these actions the \textbf{$\mathbb C^*$ twistor action}. The name comes from the fact that they are related to a well-known $\C^*$ action on the moduli space of flat connections. Since it is not relevant for the purpose of this paper, we will treat this connection in future work.

Observe that the $\mathbb C^*$ twistor action restricts to a well-known $\mathbb S^1$ action on the space of real affine spheres. It is also clear that $L_{\sigma}=L_{\zeta\cdot \sigma}$, hence the action restricts to $\CAS^*(S)$. 

Finally, if $\sigma_0$ and $\sigma_1$ lie in a path-connected component of a fiber of $[\hat \pi]$, with $\sigma_t$ denoting a connecting path, then $\zeta\cdot \sigma_0$ and $\zeta\cdot \sigma_1$ are connected by $\zeta\cdot \sigma_t$ and the whole path is contained in a fiber of $[\hat \pi]$. In conclusion, we get the following.

\begin{prop}
    The $\C^*$ twistor action on $\CAS(S)$ descends to a holomorphic action on $\Sph(S)$ and the map
    \[\pi\colon \Sph(S)\to \mathcal M_3(S)\times \mathcal M_3(\overline S)\]
is equivariant with respect to the $\C^*$ actions on $ \Sph(S)$ and $\mathcal M_3(S)\times \mathcal M_3(\overline S)$.
\end{prop}

\begin{remark}
    It is important to observe that the $\mathrm{MCG}(S)$ action and the $\C^*$ twistor action commute both on $\Sph(S)$ and $\mathcal M_3(S)\times \mathcal M_3(\overline S)$.
\end{remark}

\subsection{Proof of Theorem A}

We are now able to prove Theorem A. The main strategy of the proof consists in taking an inverse of the projection $\pi\colon \Sph(S)\to \mathcal M_3(S)\times \mathcal M_3(\overline S)$ defined on a big open subset $\Omega$ of $\mathcal M_3(S)\times \mathcal M_3(\overline S)$ and then composing it with $hol$.

Recalling the space $\mathrm{RAS}(S)$ from Section \ref{sec: BM and Gauss}, define
$$\mathbb{RAS}(S):= \faktor{\mathrm{RAS}(S)}{\mathrm{Diff}_0(S)}.$$
As a result of Theorems \ref{prop: finite quotient of CAS} and \ref{thm: holonomy from the quotient of CAS}, we have two holomorphic maps from $\Sph(S)$, namely
\begin{equation}
\label{eq: hol and pi from CAS}
\mathcal M_3(S)\times\mathcal M_3(\overline S) \xleftarrow{\pi} \Sph (S) \xrightarrow{hol} \chi(\pi_1(S), \SL(3,\C)).
\end{equation}
On the real locus, the maps above restrict to
\begin{equation}
\label{eq: real hol and pi from CAS}
\mathcal M_3(S) \cong
\Delta_{\mathcal M_3(S)\times\mathcal M_3(\overline S)} \xleftarrow{ \sim } \mathbb {RAS}(S) \xrightarrow{\sim} \mathrm{Hit}_3(S),\ 
\end{equation}
and the composition of the inverse of the LHS with the RHS coincides with the Labourie-Loftin parametrization $\mathcal L$.

For all $(c_1,\overline{c_2}, \Q_1, 0)\in \CD(S)\times \mathcal C(\overline S)$, one can easily see that $g=\hpair$ and $C=\Q_1$ satisfy the hypothesis of Theorem \ref{thm: integrating GE for CAS}, so $\sigma=(c_1, \overline{c_2}, \Q_1, 0, 0)\in \CAS(S)$. Moreover, the linearization $L_{\sigma}$ of the Gauss equation \eqref{Gauss equation section 6} is $L_{\sigma}=\Delta -2 \mathrm{id}$, which, in the notation of Theorem \ref{thm F extended}, is an isomorphism if and only if $(c_1, \overline{c_2})\in \mathcal B^{reg}$: as a result we have a holomorphic map
\begin{equation}
\label{eq: lift s1}
\begin{split}
\CD(S)\times\CD(\overline S)\supset W_+:=\{(c_1,\overline{c_2},\Q_1, 0) \ |\ \ccpair\in \mathcal B^{reg}\} &\to \CAS^*(S) \\
(c_1,\overline{c_2}, \Q_1, 0)  &\mapsto (c_1, \overline{c_2}, \Q_1, 0, 0),
\end{split}
\end{equation}
which has $\hat \pi$ as left inverse.
Notice that $W_+$ is a finite-rank vector bundle of $\mathcal B^{reg}$, and, as a result of Theorem \ref{thm F extended}, $W_+$ is an open dense subset of $\CD(S)\times \mathcal C(\overline S)$. By using the fact that $\hat \pi$ is a local biholomorphism and a left inverse of \eqref{eq: lift s1}, one can easily see that \eqref{eq: lift s1} is a holomorphic embedding with a closed image (as a subset of $\CAS^*(S)$) because it is a continuous, injective, and closed map. Moreover, by theorem \ref{thm F extended}, the fibers of the projection $W_+\to \mathcal M_3(S)\times \mathcal T(\overline S)$ are connected, so \eqref{eq: lift s1} descends to a well-defined map
\begin{equation}
\label{eq: embedding of Hit x Teich}
  \mathrm s_1\colon \mathcal M_3(S)\times \mathcal T(\overline S)\xrightarrow{\sim} \mathcal R \subset \Sph(S),\ 
\end{equation}
which is injective and continuous (as one can easily see by restricting to local charts as in \eqref{eq: diagram quotient CAS}) and whose image we denote by $\mathcal R$. Moreover, $\mathrm{s_1}$ is a holomorphic embedding since $\pi$ is a left inverse and a local biholomorphism. By construction, one can also immediately see that $\mathrm s_1$ is equivariant for both the $\mathrm{MCG}(S)$ action and the $\C^*$ twistor action.

In the same fashion, the embedding
\begin{equation}
\label{eq: lift s2}
\begin{split}
W_-:=\{((c_1,0), (\overline{c_2}, \overline{\Q_2}) )\ |\ \ccpair\in \mathcal B^{reg}\} &\to \CAS^*(S) \\
(c_1,\overline{c_2},0, \Q_2)  &\mapsto (c_1, \overline{c_2}, 0, \overline{\Q_2}, 0)
\end{split}
\end{equation}
descends to the holomorphic embedding
\begin{equation}
\label{eq: embedding of Teich x Hit}
 \mathrm s_2\colon \mathcal T(S)\times\mathcal M_3(\overline S)\xrightarrow {\sim}\overline{\mathcal R} \subset \Sph(S).\ 
\end{equation}
which is equivariant for both $\mathrm{MCG}(S)$ and $\C^*$, and whose image we denote by $\overline{\mathcal R}$.

It is easy to check that, if $\sigma\colon \widetilde S\to \mathbb C^3 $ is the complex affine sphere corresponding to the data $(c_1, \overline{c_2}, \Q_1, \overline{\Q_2}, u)$, then $\overline \sigma$ has conjugate holonomy and corresponds to $(c_2, \overline{c_1}, \Q_2, \overline{\Q_1}, \overline u)$, motivating the notation for $\overline{\mathcal R}$.

Finally, observe that $\mathrm s_1$ and $\mathrm s_2$ coincide on $\mathcal T(S)\times \mathcal T(\overline S)$, defining a diffeomorphism between $\mathcal T(S)\times \mathcal T(\overline S)$ and $\mathcal R\cap \overline{\mathcal  R}$, which, by Example 
\ref{bersmaps}, corresponds to Bers embeddings in $\hat{\mathbb X}_2$. As a consequence, \[hol\circ (\mathrm s_1 \cap \mathrm s_2)\colon \TSTS\to \chi(S, \SL(3,\C))
\]
coincides with Bers parametrization of the quasi-Fuchsian space, after identifying $\PSL(2,\C)$ with $\mathrm{SO}(2,1,\C)\cong Isom_0(\hat {\mathbb X}_2)\subset \SL(3,\C)$.

Theorem A will follow as a corollary of the following theorem.

\begin{thm}
\label{thm: main existence of inverse}
    There exists an open subset $\Omega\subset \mathcal M_3(S)\times\mathcal M_3(\overline S)$ containing the diagonal, $\mathcal M_3(S)\times\mathcal T(
    \overline S)$, and $\mathcal T(S)\times \mathcal M_3(\overline S)$
    over which $\pi$ admits a global biholomorphic inverse $\varsigma\colon \Omega\to \Sph(S)$ which extends $s_1$ and $s_2$ above, and the real parametrization $\Delta\to \mathrm{RAS}(S)$.

Moreover, $\Omega$ can be chosen to be invariant for the actions of both $\mathrm{MCG}(S)$ and $\mathbb C^*$.
\end{thm}

Before proving Theorem \ref{thm: main existence of inverse}, we show how it implies Theorem A.

\begin{proof}[Proof of Theorem A]
    Consider the inverse $\varsigma\colon \Omega\to \Sph(S)$ for $\pi$ given by Theorem \ref{thm: main existence of inverse} and compose it with the holonomy map to get $hol\circ \varsigma\colon \Sph(S)\to \chi(S, \SL(3,\C))$. By construction, this map is holomorphic and $\mathrm{MCG}(S)$-equivariant. Moreover, $\mathcal L= hol\circ \varsigma|_{\Delta}$, and $hol\circ\varsigma|_{\mathcal T(S)\times \mathcal T(\overline S)}= hol\circ (\mathrm{s}_1\cap \mathrm{s}_2)$ which, as shown above, coincides with Bers' Theorem after identifying $\PSL(2,\C)$ with $\mathrm{SO}(2,1, \C)$.
\end{proof}

The main tool in the proof of Theorem \ref{thm: main existence of inverse} is the generalized inverse function theorem, which says that if a local diffeomorphism restricts to a diffeomorphism along a submanifold whose image is a submanifold, then it restricts to a diffeomorphism between connected neighborhoods of them (see Exercise 1.8.14 in \cite{GuillPoll} and \cite{GIFT}). Theorem \ref{thm: main existence of inverse} is much easier to prove if we don't demand that $\Omega$ is $\C^*$ and $\mathrm{MCG}(S)$-invariant. However, the proof for the action-invariant version will require some attention on some topological issues and on the choice of the neighborhoods. In particular, to construct the open subsets, it's useful to recall the $\mathrm{MCG}(S)$-invariant Riemannian metric on $\mathcal M_3(S)$ defined in \cite{kim2017kahler} and \cite{Lab3}, which we call $g$ in the proof below. As a result, $g\oplus g$ defines a Riemannian metric on $\mathcal M_3(S)\times\mathcal M_3(\overline S)$ for which the factors $\mathcal M_3(S)\times\{0\}$ and $\{0\}\times\mathcal M_3(\overline S)$ are totally geodesic. Recall a subset $U$ of a vector space is star-shaped if there is a point $v\in U$, called the center, such that for every point $v'\in U$, the segment between $v$ and $v'$ is contained in $U.$ We call a subset of $\mathcal M_3(S)\times\mathcal M_3(\overline S)$ geodesically star-shaped if it's star-shaped with respect to geodesics for $g\oplus g$.

\begin{proof}[Proof of Theorem \ref{thm: main existence of inverse}]
By the generalized inverse function theorem mentioned above, there exist $ V_1, V_2\subset \Sph(S)$, and $U_1, U_2\subset \mathcal M_3(S)\times \mathcal M_3(\overline S)$ open subsets such that $\pi$ restricts to the biholomorphisms \begin{align*}
  \mathcal R\cap \overline{\mathcal R}  \subset V_1 &\xrightarrow{\sim} U_1 \supset \mathcal T(S)\times \mathcal T(\overline S) \\ 
\mathbb{RAS}(S)\subset V_2 &\xrightarrow{\sim} U_2 \supset  \Delta. \ 
\end{align*}
We denote the inverse maps by $\varsigma_1$ and $\varsigma_2$ respectively.

Up to restricting these open subsets, we can make the following assumptions on $U_1$ and $U_2$, which we denote by \emph{assumptions} $(*)$.

\vspace{3 pt}

\emph{Assumption $(*)$ on} $U_1$.

The intersections between $U_1$ and submanifolds of the form $\mathcal {CD}([c_1])\times \mathcal {CD}([\overline{c_2}])\subset \mathcal M_3(S)\times \mathcal M_3(\overline S)$, which can be seen as open subsets in a vector space, are star-shaped open subsets with center in $([c_1, 0], [\overline{c_2},0])$. Such open subsets can be obtained by taking local trivializations of the bundle $\mathcal M_3(S)\times \mathcal M_3(\overline S)\to \TSTS$.

\vspace{3 pt}

\emph{Assumption $(*)$ on} $U_2$.

If both $m=([c_1,  \Q_1],[\overline{c_2},, \overline{ 
\Q_2}])$ and $\zeta\cdot m=([c_1,  \zeta \Q_1],[ \overline{c_2},\frac 1 \zeta \overline{ \Q_2}])$ lie in $U_2$, then $\varsigma_2(\zeta \cdot m)=\zeta\cdot  \varsigma_2(m)$: to realize this, observe that on $\C^*\times U_2$ the functions $(\zeta, m)\mapsto \varsigma_2(\zeta\cdot m)$ and $(\zeta, m)\mapsto \zeta\cdot  \varsigma_2(m)$ are holomorphic maps that agree on the totally real submanifold $\mathbb S^1\times \Delta$ as a result of $\varsigma_2 |_\Delta$ being $\mathbb S^1$-invariant, so, up to shrinking $U_2$, they coincide on some open subset $W_0\times U_2$ (hence on its intersection with the preimage of $U_2$ itself).

Moreover, we assume that the intersections between $U_2$ and the submanifolds of the form $\{[c_1,\Q_1]\}\times \mathcal M_3(\overline S)$ are geodesically star-shaped for the metric $g\oplus g$ with center in $([c_1,\Q_1],[\overline{c_1},\overline{\Q_1}])$: this can be constructed (within the previous refinements) through the exponential map and the fact that the factors $\mathcal M_3(S)\times\{0\}$ and $\{0\}\times\mathcal M_3(\overline S)$ are totally geodesic makes the construction more simple.

\vspace{3 pt}

\emph{Step 1.}
We prove that if $\hat V_i\to \hat U_i$ is another biholomorphic restriction with $U_i$ satisfying assumption $(*)$, then its inverse coincides with the inverse of $V_i\to U_i$ on the intersection $U_i\cap \hat U_i$. We prove it for $i=1$, but the proof works in the very same fashion for $i=2$. If $([c_1,\Q_1],[c_2, \overline{\Q_2}])\in U_1\cap \hat U_1$, then by assumption $(*)$ the path $([c_1,t\Q_1],[c_2, t\overline{\Q_2}])$ lies in the intersection for all $t\in[0,1] $ (for $i=2$, take the geodesic contained in the leaf $\{[c_1, \Q_1]\}\times \mathcal M_3(\overline S)$ connecting it to the intersection point with $\Delta$). Let $\hat \varsigma_1$ be the inverse of $\hat V_1\to \hat U_1$. Now, the set $$\{t\in [0,1]\ |\ \hat \varsigma_1([c_1,t\Q_1],[c_2, t\overline{\Q_2}])= \varsigma_1 ([c_1,t\Q_1],[c_2, t\overline{\Q_2}])\}$$ contains $t=0$ and is closed; if by contradiction its maximum $T$ were less than 1, inverting $\pi$ in a neighborhood of $\varsigma_1([c_1,T\Q_1],[c_2, T\overline{\Q_2}])$ we get that such a $T$ can be extended a bit further, reaching a contradiction. We conclude that the inverses coincide in the intersection. 

\vspace{3 pt}

\emph{Step 2.} We extend the $U_i$ 's a bit further. 

In fact, observe that, for all $\phi\in \mathrm{MCG}(S)$, $\phi\cdot U_i$ satisfies assumption $(*)$ where the inverses are given by $\phi\circ \varsigma\circ \phi^{-1}\colon \phi\cdot U_i\to \phi\cdot V_i$. 

Moreover, for all $\zeta\in \mathbb C^*$, $\zeta \cdot U_1$ satisfies assumption $(*)$ and an inverse of $\pi$ is given by $\zeta\circ  \varsigma\circ \frac 1 \zeta\colon\colon \zeta\cdot U_1\to \zeta\cdot V_1$. As a result of the previous step, we get that the inverses match on the intersections. Therefore, recalling that the $\mathrm{MCG}(S)$ and the $\mathbb C^*$ twistor actions commute, $\pi$ admits a biholomorphic inverse on the restrictions
\begin{align*}
 V'_1:= \bigcup_{\phi\in \mathrm{MCG}(S)} \bigcup_{\zeta\in \mathbb C^*} \phi\cdot \zeta\cdot V_1\  &\to   \bigcup_{\phi\in \mathrm{MCG}(S)} \bigcup_{\zeta\in \mathbb C^*} \phi\cdot \zeta\cdot U_1 :=  U'_1 \\
 V'_2:= \bigcup_{\phi\in \mathrm{MCG}(S)}  \phi\cdot  V_2\ &\to  \bigcup_{\phi\in \mathrm{MCG}(S)} \phi\cdot U_1 :=  U'_1, 
\end{align*}
whose inverses we still denote by $\varsigma_1$ and $\varsigma_2$ respectively. 

Observe that $V'_1\to U'_1$ restricts to $\mathcal R\to \mathcal M_3(S)\times \mathcal T(\overline S)$ and $\overline{\mathcal R}\to \mathcal T(S)\times \mathcal M_3(\overline S)$ because for all $[(c_1, \overline{c_2}, \Q_1, 0, 0)]\in \mathcal R$ -- resp. $[(c_1, \overline{c_2}, 0, \overline{\Q_2}, 0)]\in \overline{\mathcal R}$ -- there exists $\zeta$ small enough such that $[(c_1, \overline{c_2}, \zeta \Q_1, 0, 0)]$ -- resp. $[(c_1, \overline{c_2}, 0, \zeta \overline{\Q_2}, 0)]$ -- lies in $V_1$.

Notice that a priori it does not seem clear whether $\varsigma_1$ and $\varsigma_2$ agree on $U'_1\cap U'_2$.

\vspace{3 pt}

\emph{Step 3.} We glue $U'_1$ and a neighborhood of $\Delta$.

The restriction $\pi|_{V'_1}$ coincides with the real analytic map $\mathbb{RAS}(S)\to \Delta$ on the intersection ${V'_1}\cap \mathbb{RAS}(S)$: by uniqueness, $\varsigma_1$ must coincide with $\varsigma_2$ in a neighborhood $W\subset U'_1$ of $\Delta\cap U'_1$. Extend now $W$ out of $U'_1$ to an open neighborhood of $\Delta$ (to construct this extension, one could take a thin enough neighborhood of $\Delta$ disjoint from the closure of $U'_1\setminus W$ and then unionize it with $W$), which we still denote by $W$. Up to shrinking it, we can also assume $W$ to be contained inside $U'_2$, and to satisfy assumption $(*)$. By construction, the restriction of $\varsigma_2$ on $W$ coincides with $\varsigma_1$ on $W\cap U'_1$. 

Finally, by the fact that $U_2'$ is $\mathrm{MCG}(S)$-invariant and that $\varsigma_2$ is $\mathrm{MCG}(S)$-equivariant, $\phi\cdot W\subset U_2'$ for all $\phi\in \mathrm{MCG}(S)$ and $\varsigma_2$ is well-defined on  
\[
U''_2 := \bigcup_{\phi\in \mathrm{MCG}(S)} \phi\cdot  W, \ 
\]
whose image we denote by $V''_2:=\varsigma(U'')\subset V'_2$.

Moreover, $\varsigma_1$ and $\varsigma_2$ coincide on $U_2''\cap U_1'$: in fact, if by contradiction there existed a point in which they were different, then, since both sections are $\mathrm{MCG}(S)$-equivariant, one would find a point in $W\cap U'_1$ in which they are different, giving a contradiction. 

We call $\varsigma$ the inverse constructed by patching $\varsigma_1$ and $\varsigma_2$ on $U'_1\cup U_2''$.

\vspace{3 pt}

\emph{Step 4.} We prove that $\varsigma$ can be extended to
\[
\Omega:= U'_1\cup \bigcup_{\zeta\in \mathbb C^*} \zeta\cdot U''_2\ ,
\]
concluding the proof.

On each $\zeta\cdot U''_2$, consider the inverse $\varsigma_2^\zeta:= \zeta \circ \varsigma_2\circ \frac 1 \zeta\colon \zeta\cdot U''_2\to \zeta \cdot V''_2$ for $\pi$. 

For all $\zeta \in \mathbb C^*$, on $\zeta\cdot U''_2\cap U_1'$, $\varsigma^\zeta _2$ and $\varsigma_1$ coincide, since if they didn't then one would find a point of $U''_2\cap U'_1$ in which they don't coincide, because $\varsigma_1$ is $\mathbb C^*$-equivariant.

We prove that different $\varsigma^\zeta$ 's coincide on the intersection. Recall that on $U_2$ the inverse $\varsigma_2$ is such that $\varsigma_2(\zeta\cdot m)= \zeta\cdot \varsigma_2(m)$ when both $m$ and $\zeta\cdot m$ lie in $U_2$; since the $\mathrm{MCG}(S)$ and the $\C^*$ actions commute, this property holds also on $U_2'$, hence on $U_2''$. Finally, if $\varsigma_2^\zeta$ and $\varsigma_2^{\zeta \hat \zeta}$ did not coincide on $(\zeta \cdot U'')\cap (\zeta \hat\zeta \cdot U'')$, then (by applying $\frac 1 \zeta$) we would have that $\varsigma_2$ and $\varsigma_2^{\hat \zeta}$ do not coincide on $U''\cap (\hat \zeta \cdot U'')$, contradicting the assumption $(*)$ for $U$.

\end{proof}

\subsection{Proof of Theorem C}\label{sec: ThmD}
Here we prove Theorem C. As stated in Section 1.2, to do so we analyze one of the simpler families of points in $\mathcal{CD}(S)\times \mathcal{CD}(\overline{S}),$  namely rays $$[0,\infty)\to \mathcal{CD}(S)\times \mathcal{CD}(\overline{S}), t\mapsto (c,\overline{c},t\Q,-t\overline{\Q}).$$  The Gauss equation (\ref{eq: PDE of Gauss eq}) is 
\begin{equation}\label{eqn: bad Gauss}
     \Delta_h u = e^{2u}+\frac{t^2}{4}e^{-4u}h(\Q,\overline{\Q}) -1,
\end{equation}
which, when restricted to real functions $u,$ coincides with the structural equation for minimal Lagrangian immersions into the $2$-dimensional complex hyperbolic space $\mathbb{CH}^2.$ Minimal Lagrangians in $\mathbb{CH}^2$ have been studied in a number of works, such as \cite{LofM} and \cite{HLM}, and we can use known results directly. 

We need only part (3) of the following theorem about real solutions, which is proved by a calculation using Gauss-Bonnet, but we include parts (1) and (2) because they shed light on the theory of complex affine spheres. Note that in previous works on Equation (\ref{eqn: bad Gauss}), authors have considered only real solutions. Here we are allowed to have complex solutions. So when stating previous results, we will be sure to specify that they apply only for real solutions.

\begin{thm}[Theorem 1.1 in Huang-Loftin-Lucia \cite{HLM}]\label{thm: HLL}
    Let $c$ be a complex structure on $S$ and $\Q$ a $c$-holomorphic cubic differential. There are constants $T_0=T_0(c,\Q)>0$ and $T=T(c,\Q)>T_0$ such that the following hold.
    \begin{enumerate}
        \item For any $t\in [0,T_0),$ there are at least two real solutions to (\ref{eqn: bad Gauss}).
        \item For $t=T_0,$ there is a real solution to (\ref{eqn: bad Gauss}).
        \item For $t>T$, there are no real solutions to (\ref{eqn: bad Gauss}).
    \end{enumerate}
\end{thm}
One hint toward Theorem \ref{thm: HLL} is that Equation (\ref{eqn: bad Gauss}) strongly resembles the better-known Gauss equation for minimal surfaces in $\mathbb{H}^3$, which is not particularly well behaved (see \cite{Uh}). Heuristically, issues should arise because the linearization at a solution $\sigma$ is expected to have a kernel for $t$ large. Indeed, at a solution $\sigma=(c,\overline{c},t\Q,t\overline{\Q},u)$, the linearization is given by 
\begin{equation}\label{eqn: bad linearization}
    L_{\sigma_t}(v) = \Delta_h v -v(e^{u}-t^2e^{-4u}h(\Q,\overline{\Q})).
\end{equation}
For $\Q$ very large, the expression in the brackets in (\ref{eqn: bad linearization}) is negative. If we pretend that expression is a negative constant, then Equation (\ref{eqn: bad linearization}) becomes the equation for eigenfunctions of the Laplacian, which could have a non-zero solution. The authors in \cite{HLM} realize this expectation by showing that that the linearization does have non-empty kernel at the solution at time $T_0$.

As in \cite{HLM}, one can construct a nice path of real solutions around $(c,\overline{c},0,0,0)$. Define $G_{\R}: \mathcal{CD}(S)\times C^\infty(S,\R)\to C^\infty(S,\R)$ by $G_{\R}(c,\Q,u)=G(c,\overline{c},\Q,-\overline{\Q}, u).$ Taking $\Q=0$ and $u=0$ and extending $G_{\R}$ to all of the Banach manifolds $\mathcal{CD}^l(S)\times W^{k+2,2}(S,\R)$, the linearization of $G_{\R}$ is an isomorphism in the direction of $u$, and it follows that there is a unique continuous path into $C^\infty(S,\R)$ of real solutions $t\mapsto u_t$ of $G(c,\overline{c},t\Q,-t\overline{\Q},u_t)=0,$ for $t$ in some small interval $(-\varepsilon,\varepsilon).$ 

Theorem C follows as a consequence of the following Proposition.
\begin{prop}
    There exists no global holomorphic map 
    $\varsigma\colon\CD(S)\times\CD(\overline S)\to \CD(S)\times \CD(\overline S)\times C^{\infty}(S,\C)$
    whose image is contained in $\CAS(S)$, and which inverts $\pi$ and maps any $(c,\overline c, 0,0)$ to $(c,\overline c, 0,0,0)$. 
\end{prop}

\begin{proof}

Denote $\sigma_t=(c, \overline c, t\Q, -t\overline {\Q}, u_t):=\varsigma(c, \overline c, t\Q, -t\overline {\Q})$. By the fact that $L_{\sigma_0}$ is invertible, $\sigma_0\in \CAS^*(S)$, and since $\CAS^*(S)\to \CD(S)\times\CD(\overline S)$ is a local homeomorphism, we get that, for $t>0$ small enough, $u_t$ must be the real solution to $G=0$ mentioned above, just because they can be continuously deformed to $t=0$. Moreover, since $\zeta$ is holomorphic, the path $\zeta(c,\overline c, t\Q, -t\overline{\Q}, u_t)$ is real analytic with respect to $t$, hence at each point $p\in S$, $u_t(p)$ being real for small $t$ implies it being real for all $t>0$. This would produce a real solution $u_t$ for $t>T$, contradicting Theorem \ref{thm: HLL}.
\end{proof}

\begin{remark}
We expect that, as one can possibly see by writing the connection forms explicitly or by means of the perspective of \cite{RT}, complex affine spheres projecting to $(c,\overline c, \Q, -\overline \Q)\in \CD(S)\times \CD(\overline S)$ with Riemannian Blaschke metric have the same holonomy as the corresponding minimal Lagrangian immersion in $\mathbb{CH}^2$.
\end{remark}


\printbibliography

@article {Bhd,
    AUTHOR = {Bers, Lipman},
     TITLE = {Holomorphic differentials as functions of moduli},
   JOURNAL = {Bull. Amer. Math. Soc.},
  FJOURNAL = {Bulletin of the American Mathematical Society},
    VOLUME = {67},
      YEAR = {1961},
     PAGES = {206--210},
       DOI = {10.1090/S0002-9904-1961-10569-7},
}

@book {GT,
    AUTHOR = {Gilbarg, David and Trudinger, Neil S.},
     TITLE = {Elliptic partial differential equations of second order},
    SERIES = {Classics in Mathematics},
      NOTE = {Reprint of the 1998 edition},
 PUBLISHER = {Springer-Verlag, Berlin},
      YEAR = {2001},
     PAGES = {xiv+517},
}

@incollection {Calabi,
    AUTHOR = {Calabi, Eugenio},
     TITLE = {Complete affine hyperspheres. {I}},
 BOOKTITLE = {Symposia {M}athematica, {V}ol. {X} ({C}onvegno di {G}eometria
              {D}ifferenziale, {INDAM}, {R}ome, 1971 \& {C}onvegno di
              {A}nalisi {N}umerica, {INDAM}, {R}ome, 1972)},
     PAGES = {19--38},
 PUBLISHER = {Academic Press, London-New York},
      YEAR = {1972},
   MRCLASS = {57E25 (53A15)},
  MRNUMBER = {365607},
MRREVIEWER = {H.\ W.\ Guggenheimer},
}

@misc{Nolte,
      title={Canonical Maps from Spaces of Higher Complex Structures to Hitchin Components}, 
      author={Alexander Nolte},
      year={2022},
      eprint={2204.04732},
      archivePrefix={arXiv},
      primaryClass={math.GT},
      url={https://arxiv.org/abs/2204.04732}, 
}

@article {DVV,
    AUTHOR = {Dillen, F. and Vrancken, L. and Verstraelen, L.},
     TITLE = {Complex affine differential geometry},
   JOURNAL = {Atti Accad. Peloritana Pericolanti. Cl. Sci. Fis. Mat. Natur.},
  FJOURNAL = {Atti della Accademia Peloritana dei Pericolanti. Classe di
              Scienze, Fisiche, Matematiche e Naturali},
    VOLUME = {66},
      YEAR = {1988},
     PAGES = {231--260},
}

@incollection {Abe,
    AUTHOR = {Abe, Kinetsu},
     TITLE = {Affine differential geometry of complex hypersurfaces},
 BOOKTITLE = {Geometry and topology of submanifolds, {III} ({L}eeds, 1990)},
     PAGES = {1--31},
 PUBLISHER = {World Sci. Publ., River Edge, NJ},
      YEAR = {1991},
}

@article {CY2,
    AUTHOR = {Cheng, Shiu Yuen and Yau, Shing-Tung},
     TITLE = {Complete affine hypersurfaces. {I}. {T}he completeness of
              affine metrics},
   JOURNAL = {Comm. Pure Appl. Math.},
  FJOURNAL = {Communications on Pure and Applied Mathematics},
    VOLUME = {39},
      YEAR = {1986},
    NUMBER = {6},
     PAGES = {839--866},
       DOI = {10.1002/cpa.3160390606},
}

@article {CY1,
    AUTHOR = {Cheng, Shiu Yuen and Yau, Shing Tung},
     TITLE = {On the regularity of the {M}onge-{A}mp\`ere equation {${\rm
              det}(\partial \sp{2}u/\partial x\sb{i}\partial
              sx\sb{j})=F(x,u)$}},
   JOURNAL = {Comm. Pure Appl. Math.},
  FJOURNAL = {Communications on Pure and Applied Mathematics},
    VOLUME = {30},
      YEAR = {1977},
    NUMBER = {1},
     PAGES = {41--68},
       DOI = {10.1002/cpa.3160300104},
}

@article {Gol,
    AUTHOR = {Goldman, William M.},
     TITLE = {Convex real projective structures on compact surfaces},
   JOURNAL = {J. Differential Geom.},
  FJOURNAL = {Journal of Differential Geometry},
    VOLUME = {31},
      YEAR = {1990},
    NUMBER = {3},
     PAGES = {791--845},
}

@incollection {Wa,
    AUTHOR = {Wang, Chang Ping},
     TITLE = {Some examples of complete hyperbolic affine {$2$}-spheres in
              {${\bf R}^3$}},
 BOOKTITLE = {Global differential geometry and global analysis ({B}erlin,
              1990)},
    SERIES = {Lecture Notes in Math.},
    VOLUME = {1481},
     PAGES = {271--280},
 PUBLISHER = {Springer, Berlin},
      YEAR = {1991},
       DOI = {10.1007/BFb0083648},

}

@book {Beg,
    AUTHOR = {Begehr, Heinrich G. W.},
     TITLE = {Complex analytic methods for partial differential equations},
      NOTE = {An introductory text},
 PUBLISHER = {World Scientific Publishing Co., Inc., River Edge, NJ},
      YEAR = {1994},
     PAGES = {x+273},
       DOI = {10.1142/2162},
}

@misc{Blatt,
      title={On the analyticity of solutions to non-linear elliptic partial differential systems}, 
      author={Simon Blatt},
      year={2020},
      eprint={2009.08762},
      archivePrefix={arXiv},
      primaryClass={math.AP}
}

@article {Lof,
    AUTHOR = {Loftin, John C.},
     TITLE = {Affine spheres and convex {$\Bbb{RP}^n$}-manifolds},
   JOURNAL = {Amer. J. Math.},
  FJOURNAL = {American Journal of Mathematics},
    VOLUME = {123},
      YEAR = {2001},
    NUMBER = {2},
     PAGES = {255--274},
}

@article {Lab2,
    AUTHOR = {Labourie, Fran\c{c}ois},
     TITLE = {Flat projective structures on surfaces and cubic holomorphic
              differentials},
   JOURNAL = {Pure Appl. Math. Q.},
  FJOURNAL = {Pure and Applied Mathematics Quarterly},
    VOLUME = {3},
      YEAR = {2007},
    NUMBER = {4},
     PAGES = {1057--1099},
 MRCLASS = {53C20 (53A15 53C56 57M50)},
  MRNUMBER = {2402597},
MRREVIEWER = {John\ C.\ Loftin},
       DOI = {10.4310/PAMQ.2007.v3.n4.a10},
}

@article {Lab1,
    AUTHOR = {Labourie, Fran\c{c}ois},
     TITLE = {Anosov flows, surface groups and curves in projective space},
   JOURNAL = {Invent. Math.},
  FJOURNAL = {Inventiones Mathematicae},
    VOLUME = {165},
      YEAR = {2006},
    NUMBER = {1},
     PAGES = {51--114},
       DOI = {10.1007/s00222-005-0487-3},

}

@article {Hi,
    AUTHOR = {Hitchin, N. J.},
     TITLE = {Lie groups and {T}eichm\"{u}ller space},
   JOURNAL = {Topology},
  FJOURNAL = {Topology. An International Journal of Mathematics},
    VOLUME = {31},
      YEAR = {1992},
    NUMBER = {3},
     PAGES = {449--473},
       DOI = {10.1016/0040-9383(92)90044-I},
  
}

@misc{ElE,
      title={A metric uniformization model for the Quasi-Fuchsian space}, 
      author={Christian El Emam},
      year={2023},
      eprint={2307.07388},
      archivePrefix={arXiv},
      primaryClass={math.DG}
}

@book {IT,
    AUTHOR = {Imayoshi, Y. and Taniguchi, M.},
     TITLE = {An introduction to {T}eichm\"{u}ller spaces},
      NOTE = {Translated and revised from the Japanese by the authors},
 PUBLISHER = {Springer-Verlag, Tokyo},
      YEAR = {1992},
     PAGES = {xiv+279},
       DOI = {10.1007/978-4-431-68174-8},
}

@article {Hash,
    AUTHOR = {Hashimoto, Yoshiaki},
     TITLE = {A remark on the analyticity of the solutions for non-linear
              elliptic partial differential equations},
   JOURNAL = {Tokyo J. Math.},
  FJOURNAL = {Tokyo Journal of Mathematics},
    VOLUME = {29},
      YEAR = {2006},
    NUMBER = {2},
     PAGES = {271--281},
        DOI = {10.3836/tjm/1170348166},
}

@article {Kato,
    AUTHOR = {Kato, Keiichi},
     TITLE = {New idea for proof of analyticity of solutions to analytic
              nonlinear elliptic equations},
   JOURNAL = {SUT J. Math.},
  FJOURNAL = {SUT Journal of Mathematics},
    VOLUME = {32},
      YEAR = {1996},
    NUMBER = {2},
     PAGES = {157--161},
}

@book {Ni,
    AUTHOR = {Nicolaescu, Liviu I.},
     TITLE = {Lectures on the geometry of manifolds},
   EDITION = {Second},
 PUBLISHER = {World Scientific Publishing Co. Pte. Ltd., Hackensack, NJ},
      YEAR = {2007},
     PAGES = {xviii+589},
       DOI = {10.1142/9789812770295},
    
}

@article{EE,
author = {Clifford J. Earle and James Eells},
title = {{A fibre bundle description of Teichmüller theory}},
volume = {3},
journal = {Journal of Differential Geometry},
number = {1-2},
publisher = {Lehigh University},
pages = {19 -- 43},
year = {1969},
doi = {10.4310/jdg/1214428816},
}

@article {DZ,
    AUTHOR = {Dumitrescu, Sorin and Zeghib, Abdelghani},
     TITLE = {Global rigidity of holomorphic {R}iemannian metrics on compact
              complex 3-manifolds},
   JOURNAL = {Math. Ann.},
  FJOURNAL = {Mathematische Annalen},
    VOLUME = {345},
      YEAR = {2009},
    NUMBER = {1},
     PAGES = {53--81},
       DOI = {10.1007/s00208-009-0342-8},
}

@book {Ast,
    AUTHOR = {Astala, Kari and Iwaniec, Tadeusz and Martin, Gaven},
     TITLE = {Elliptic partial differential equations and quasiconformal
              mappings in the plane},
    SERIES = {Princeton Mathematical Series},
    VOLUME = {48},
 PUBLISHER = {Princeton University Press, Princeton, NJ},
      YEAR = {2009},
     PAGES = {xviii+677},
}

@article {BEE,
    AUTHOR = {Bonsante, Francesco and El Emam, Christian},
     TITLE = {On immersions of surfaces into {$SL(2,\Bbb C)$} and geometric
              consequences},
   JOURNAL = {Int. Math. Res. Not. IMRN},
  FJOURNAL = {International Mathematics Research Notices. IMRN},
      YEAR = {2022},
    NUMBER = {12},
     PAGES = {8803--8864},
       DOI = {10.1093/imrn/rnab189},
}

@article {AB,
    AUTHOR = {Ahlfors, Lars and Bers, Lipman},
     TITLE = {Riemann's mapping theorem for variable metrics},
   JOURNAL = {Ann. of Math. (2)},
  FJOURNAL = {Annals of Mathematics. Second Series},
    VOLUME = {72},
      YEAR = {1960},
     PAGES = {385--404},
       DOI = {10.2307/1970141},
}

@incollection {BD,
    AUTHOR = {Beilinson, A. A. and Drinfeld, V. G.},
     TITLE = {Quantization of {H}itchin's fibration and {L}anglands'
              program},
 BOOKTITLE = {Algebraic and geometric methods in mathematical physics
              ({K}aciveli, 1993)},
    SERIES = {Math. Phys. Stud.},
    VOLUME = {19},
     PAGES = {3--7},
 PUBLISHER = {Kluwer Acad. Publ., Dordrecht},
      YEAR = {1996},
}

@article {HLM,
    AUTHOR = {Huang, Zheng and Loftin, John and Lucia, Marcello},
     TITLE = {Holomorphic cubic differentials and minimal {L}agrangian
              surfaces in {$\Bbb C\Bbb H^2$}},
   JOURNAL = {Math. Res. Lett.},
  FJOURNAL = {Mathematical Research Letters},
    VOLUME = {20},
      YEAR = {2013},
    NUMBER = {3},
     PAGES = {501--520},
        DOI = {10.4310/MRL.2013.v20.n3.a8},
}

@book{NS,
  title={Affine Differential Geometry: Geometry of Affine Immersions},
  author={Nomizu, K. and Sasaki, T.},
  series={Cambridge Tracts in Mathematics},
  year={1994},
  publisher={Cambridge University Press}
}

@book {Gun,
    AUTHOR = {Gunning, R. C.},
     TITLE = {Lectures on vector bundles over {R}iemann surfaces},
 PUBLISHER = {University of Tokyo Press, Tokyo; Princeton University Press,
              Princeton, NJ},
      YEAR = {1967},
     PAGES = {v+243},
   MRCLASS = {57.30 (30.00)},
  MRNUMBER = {230326},
MRREVIEWER = {W.\ Kaup},
}

@misc{B,
    title={G2 geometry and integrable systems},
    author={David Baraglia},
    year={2010},
    eprint={1002.1767},
    archivePrefix={arXiv},
    primaryClass={math.DG}
}

@misc{ESbelt,
      title={Holomorphic dependence for the Beltrami equation in Sobolev spaces}, 
      author={Christian El Emam and Nathaniel Sagman},
      year={2024},
      eprint={2410.06175},
      archivePrefix={arXiv},
      primaryClass={math.CV},
      url={https://arxiv.org/abs/2410.06175}, 
}

@article {Sik,
    AUTHOR = {Sikora, Adam S.},
     TITLE = {Character varieties},
   JOURNAL = {Trans. Amer. Math. Soc.},
  FJOURNAL = {Transactions of the American Mathematical Society},
    VOLUME = {364},
      YEAR = {2012},
    NUMBER = {10},
     PAGES = {5173--5208},
      ISSN = {0002-9947,1088-6850},
   MRCLASS = {14D20 (14L24 53D30 57M50)},
  MRNUMBER = {2931326},
MRREVIEWER = {Benjamin\ M. S. Martin},
       DOI = {10.1090/S0002-9947-2012-05448-1},
       URL = {https://doi.org/10.1090/S0002-9947-2012-05448-1},
}

@misc{Wen,
      title={Higgs bundles and local systems on Riemann surfaces}, 
      author={Richard A. Wentworth},
      year={2015},
      eprint={1402.4203},
      archivePrefix={arXiv},
      primaryClass={math.DG}
}

@article {SU,
    AUTHOR = {Bers, Lipman},
     TITLE = {Simultaneous uniformization},
   JOURNAL = {Bull. Amer. Math. Soc.},
  FJOURNAL = {Bulletin of the American Mathematical Society},
    VOLUME = {66},
      YEAR = {1960},
     PAGES = {94--97},
       DOI = {10.1090/S0002-9904-1960-10413-2},

}

@article {ChG,
    AUTHOR = {Choi, Suhyoung and Goldman, William M.},
     TITLE = {Convex real projective structures on closed surfaces are
              closed},
   JOURNAL = {Proc. Amer. Math. Soc.},
  FJOURNAL = {Proceedings of the American Mathematical Society},
    VOLUME = {118},
      YEAR = {1993},
    NUMBER = {2},
     PAGES = {657--661},
       DOI = {10.2307/2160352},
}

@misc{ADL,
      title={Projective Structures with (Quasi-)Hitchin Holonomy}, 
      author={Daniele Alessandrini and Colin Davalo and Qiongling Li},
      year={2021},
      eprint={2110.15407},
      archivePrefix={arXiv},
      primaryClass={id='math.GT'}
}

@book{Dodson, 
place={Cambridge}, 
series={London Mathematical Society Lecture Note Series}, 
title={Geometry in a Fréchet Context: A Projective Limit Approach}, publisher={Cambridge University Press}, 
author={Dodson, C. T. J. and Galanis, George and Vassiliou, Efstathios}, 
year={2015}, 
collection={London Mathematical Society Lecture Note Series}}

@article{Treves,
author = {Fran{\c{c}}ois Treves},
title = {{On local solvability of linear partial differential equations}},
volume = {76},
journal = {Bulletin of the American Mathematical Society},
number = {3},
publisher = {American Mathematical Society},
pages = {552 -- 571},
year = {1970},
}

@misc{AMTW,
      title={Fiber bundles associated with Anosov representations}, 
      author={Daniele Alessandrini and Sara Maloni and Nicolas Tholozan and Anna Wienhard},
      year={2023},
      eprint={2303.10786},
      archivePrefix={arXiv},
      primaryClass={math.GT}
}

@misc{Dav,
      title={Nearly geodesic immersions and domains of discontinuity}, 
      author={Colin Davalo},
      year={2023},
      eprint={2303.11260},
      archivePrefix={arXiv},
      primaryClass={math.DG}
}

@misc{RT,
      title={Complex Lagrangian minimal surfaces, bi-complex Higgs bundles and $\mathrm{SL}(3,\mathbb{C})$-quasi-Fuchsian representations}, 
      author={Nicholas Rungi and Andrea Tamburelli},
      year={2024},
      eprint={2406.14945},
      archivePrefix={arXiv},
      primaryClass={id='math.DG'}
}

@article {DS2,
    AUTHOR = {Dumas, David and Sanders, Andrew},
     TITLE = {Uniformization of compact complex manifolds by {A}nosov
              homomorphisms},
   JOURNAL = {Geom. Funct. Anal.},
  FJOURNAL = {Geometric and Functional Analysis},
    VOLUME = {31},
      YEAR = {2021},
    NUMBER = {4},
     PAGES = {815--854},
        DOI = {10.1007/s00039-021-00572-6},
}

@book {Fol,
    AUTHOR = {Folland, Gerald B.},
     TITLE = {Introduction to partial differential equations},
   EDITION = {Second},
 PUBLISHER = {Princeton University Press, Princeton, NJ},
      YEAR = {1995},
     PAGES = {xii+324},
 }

@article {Kim,
    AUTHOR = {Kim, Young-Heon},
     TITLE = {Holomorphic extensions of {L}aplacians and their determinants},
   JOURNAL = {Adv. Math.},
  FJOURNAL = {Advances in Mathematics},
    VOLUME = {211},
      YEAR = {2007},
    NUMBER = {2},
     PAGES = {517--545},
       DOI = {10.1016/j.aim.2006.09.009},

}

@incollection {Uh,
    AUTHOR = {Uhlenbeck, Karen K.},
     TITLE = {Closed minimal surfaces in hyperbolic {$3$}-manifolds},
 BOOKTITLE = {Seminar on minimal submanifolds},
    SERIES = {Ann. of Math. Stud.},
    VOLUME = {103},
     PAGES = {147--168},
 PUBLISHER = {Princeton Univ. Press, Princeton, NJ},
      YEAR = {1983},
}

@article {LofM,
    AUTHOR = {Loftin, John and McIntosh, Ian},
     TITLE = {Minimal {L}agrangian surfaces in {$\Bbb{CH}^2$} and
              representations of surface groups into {$SU(2,1)$}},
   JOURNAL = {Geom. Dedicata},
  FJOURNAL = {Geometriae Dedicata},
    VOLUME = {162},
      YEAR = {2013},
     PAGES = {67--93},
      DOI = {10.1007/s10711-012-9717-1},

}

@article {DW,
    AUTHOR = {Dumas, David and Wolf, Michael},
     TITLE = {Polynomial cubic differentials and convex polygons in the
              projective plane},
   JOURNAL = {Geom. Funct. Anal.},
  FJOURNAL = {Geometric and Functional Analysis},
    VOLUME = {25},
      YEAR = {2015},
    NUMBER = {6},
     PAGES = {1734--1798},
     DOI = {10.1007/s00039-015-0344-5},
}

@article {DS1,
    AUTHOR = {Dumas, David and Sanders, Andrew},
     TITLE = {Geometry of compact complex manifolds associated to
              generalized quasi-{F}uchsian representations},
   JOURNAL = {Geom. Topol.},
  FJOURNAL = {Geometry \& Topology},
    VOLUME = {24},
      YEAR = {2020},
    NUMBER = {4},
     PAGES = {1615--1693},
       DOI = {10.2140/gt.2020.24.1615},
}

@article {CTT,
    AUTHOR = {Collier, Brian and Tholozan, Nicolas and Toulisse,
              J\'{e}r\'{e}my},
     TITLE = {The geometry of maximal representations of surface groups into
              {${\rm SO}_0(2,n)$}},
   JOURNAL = {Duke Math. J.},
  FJOURNAL = {Duke Mathematical Journal},
    VOLUME = {168},
      YEAR = {2019},
    NUMBER = {15},
     PAGES = {2873--2949},
       DOI = {10.1215/00127094-2019-0052},
}

@article {Lab3,
    AUTHOR = {Labourie, Fran\c{c}ois},
     TITLE = {Cyclic surfaces and {H}itchin components in rank 2},
   JOURNAL = {Ann. of Math. (2)},
  FJOURNAL = {Annals of Mathematics. Second Series},
    VOLUME = {185},
      YEAR = {2017},
    NUMBER = {1},
     PAGES = {1--58},
        DOI = {10.4007/annals.2017.185.1.1},
}

@book {Bookanalytic,
    AUTHOR = {Kriegl, Andreas and Michor, Peter W.},
     TITLE = {The convenient setting of global analysis},
    SERIES = {Mathematical Surveys and Monographs},
    VOLUME = {53},
 PUBLISHER = {American Mathematical Society, Providence, RI},
      YEAR = {1997},
     PAGES = {x+618},
       DOI = {10.1090/surv/053},
}

@article {Tsuboi,
    AUTHOR = {Tsuboi, Takashi},
     TITLE = {On the group of real analytic diffeomorphisms},
   JOURNAL = {Ann. Sci. \'Ec. Norm. Sup\'er. (4)},
  FJOURNAL = {Annales Scientifiques de l'\'Ecole Normale Sup\'erieure.
              Quatri\`eme S\'erie},
    VOLUME = {42},
      YEAR = {2009},
    NUMBER = {4},
     PAGES = {601--651},
       DOI = {10.24033/asens.2104},

}

@article{kim2017kahler,
  title={K{\"a}hler metric on the space of convex real projective structures on surface},
  author={Kim, Inkang and Zhang, Genkai},
  journal={Journal of Differential Geometry},
  volume={106},
  number={1},
  pages={127--137},
  year={2017},
  publisher={Lehigh University}
}

@article {EleSeppi,
    AUTHOR = {El Emam, Christian and Seppi, Andrea},
     TITLE = {On the {G}auss map of equivariant immersions in hyperbolic
              space},
   JOURNAL = {J. Topol.},
  FJOURNAL = {Journal of Topology},
    VOLUME = {15},
      YEAR = {2022},
    NUMBER = {1},
     PAGES = {238--301},
}

@book {GuillPoll,
    AUTHOR = {Guillemin, Victor and Pollack, Alan},
     TITLE = {Differential topology},
      NOTE = {Reprint of the 1974 original},
 PUBLISHER = {AMS Chelsea Publishing, Providence, RI},
      YEAR = {2010},
     PAGES = {xviii+224},
       DOI = {10.1090/chel/370},
}

@misc{GIFT,
  title = {Lecture notes},
author = {Wang, Zuoqin},
  howpublished = {\\
      \url{http://staff.ustc.edu.cn/~wangzuoq/Courses/21F-Manifolds/Notes/Lec10.pdf}},
}

\end{document}